\documentclass[11 pt]{amsart}

\usepackage[a4paper]{geometry}
\usepackage{amsmath}
\usepackage{amsthm}
\usepackage{mathrsfs}
\usepackage{amsfonts}
\usepackage{amsmath,amscd}
\usepackage{array}
\usepackage{amssymb}
\usepackage[textsize=tiny, backgroundcolor=yellow]{todonotes}
\usepackage{pdflscape}
\usepackage{stmaryrd}
\usepackage{xparse}
\usepackage{hyperref}
\usepackage{amsrefs}
\usepackage{comment}
\usepackage{float}
\usepackage{enumitem}

\usepackage{mathtools}

\DeclarePairedDelimiter\floor{\lfloor}{\rfloor}

\newtheorem{theorem}{Theorem}[section]
\newtheorem{lemma}[theorem]{Lemma}
\newtheorem{conjecture}[theorem]{Conjecture}
\newtheorem{definition}[theorem]{Definition}
\newtheorem{corollary}[theorem]{Corollary}
\newtheorem{proposition}[theorem]{Proposition}

\newtheorem{fact}{Property}

\newtheorem{remark}[theorem]{Remark}
\newtheorem{eg}[theorem]{Example}  

\newcommand{\CC}{\mathbb{C}}
\newcommand{\QQ}{\mathbb{Q}}

\newcommand{\ZZ}{\mathbb{Z}}
\newcommand{\OO}{\mathcal{O}}
\newcommand{\HH}{\mathcal{H}}
\newcommand{\x}{{\bf x}}
\newcommand{\y}{{\bf y}}
\newcommand{\M}[2]{\overline{\mathcal{M}}_{#1, #2}}
\newcommand{\W}[2]{\overline{\mathcal{W}}_{#1, #2}}

\NewDocumentCommand\A{mg}{\IfNoValueTF{#2}
    {\mathscr{H}_{#1}}
    {\mathscr{H}_{#1, #2}}%
}

\newcommand{\Jac}[1]{{\rm Jac}(#1)}
\newcommand{\Gmax}[1]{G_{#1}}
\newcommand{\AW}{E}
\newcommand{\ld}{l}
\newcommand{\LL}{\mathscr{L}}
\newcommand{\mirror}{\Psi}
\newcommand{\pot}{\mathcal{F}}

\newcommand{\apot}[2]{\mathscr{A}^{#1}_{#2}}
\newcommand{\prim}{\zeta}

\newcommand{\chat}[1]{\hat{c}_{#1}}

\newcommand{\Fourptcorr}[4]{\langle #1, \; #2, \; #3, \; #4 \rangle}
\newcommand{\fourptcorr}[4]{\langle #1, \, #2, \, #3, \, #4 \rangle}
\newcommand{\Fiveptcorr}[5]{\langle #1, \; #2, \; #3, \; #4, \; #5 \rangle}
\newcommand{\Sixptcorr}[6]{\langle #1, \; #2, \; #3, \; #4, \; #5, \; #6 \rangle}
\newcommand{\sixptcorr}[6]{\langle #1, \, #2, \, #3, \, #4, \, #5, \, #6 \rangle}

\newcommand{\bgen}{\phi}
\newcommand{\Hessbase}[1]{\bgen_{#1}}
\newcommand{\HessbaseA}[1]{\mirror(\Hessbase{#1})}
\newcommand{\agen}{\theta}

\newcommand{\J}[1]{J_{#1}}
\newcommand{\Jfunc}{\mathscr{J}}
\newcommand{\num}{\ell}
\newcommand{\m}{m}
\newcommand{\n}{n}
\newcommand{\bb}{b}
\newcommand{\expon}{a}
\newcommand{\K}{K}
\newcommand{\V}{C}

\newcommand{\one}{\mathbf{1}}
\newcommand{\phase}[1]{#1}
\newcommand{\valu}{c}
\newcommand{\X}[2]{\mathfrak{#1}_{#2}}

\newcommand{\curve}{\mathscr{C}}
\newcommand{\plaincurve}{C}
\newcommand{\BigL}{\mathfrak{L}}
\newcommand{\row}{\rho^T}
\newcommand{\col}{\rho}

\DeclareMathOperator{\numer}{num}
\DeclareMathOperator{\denom}{denom}
\DeclareMathOperator{\spn}{span}
\DeclareMathOperator{\Ch}{Ch}
\DeclareMathOperator{\wt}{wt}

\DeclareMathOperator{\Res}{Res} 

\newcommand{\nRes}{\widetilde{\Res}}

\usepackage{color}

\newcommand{\pinklight}[1]{\colorbox{red}{#1}}

\begin{document}


       \author{Weiqiang He}
      \address{Department of Mathematics, Sun Yat-sen University, Guangzhou, 510275, China}
       \email{hewq@mail2.sysu.edu.cn}

	 \author{Si Li}
      \address{Yau Mathematical Sciences Center, Tsinghua University, Beijing, 100084, China}
       \email{sili@mail.tsinghua.edu.cn}

	 \author{Yefeng Shen}
      \address{Department of Mathematics,  University of Oregon, Eugene, OR 97403, USA.}
       \email{yfshen@uoregon.edu}

	 \author{Rachel Webb}
      \address{Department of Mathematics, University of Michigan, Ann Arbor, MI 48109, USA.}
       \email{webbra@umich.edu}

       \title{Landau-Ginzburg Mirror Symmetry Conjecture}

       \begin{abstract}

We prove the Landau-Ginzburg mirror symmetry conjecture between invertible quasi-homogeneous polynomial singularities at all genera. That is, we show that the FJRW theory (LG A-model) of such a polynomial is equivalent to the Saito-Givental theory (LG B-model) of the mirror polynomial.

       \end{abstract}





       \maketitle

    \tableofcontents


\section{Introduction}

Mirror symmetry has been a driving force in geometry and physics for the last twenty years.
During that time, we have made tremendous progress in our understanding of mirror symmetry, but several important mathematical questions remain unanswered.

Historically, mathematical research focused on mirror symmetry between Calabi-Yau/Calabi-Yau models or Toric/Landau-Ginzburg models, rarely investigating the Landau-Ginzburg pairs.
This was mainly due to the lack of a mathematical theory for a Landau-Ginzburg (LG) A-model, although there were geometric realizations of the Landau-Ginzburg B-model in various contexts.
In the mid 2000's, Fan, Jarvis and Ruan invented FJRW theory \cite{FJR} motivated by the physical work \cite{W} of Witten.
This invention is a mathematical theory for a Landau-Ginzburg A-model, allowing mathematicians to investigate mirror symmetry between two Landau-Ginzburg models. 
In this paper, we prove a general LG/LG mirror theorem, which can be viewed as a Landau-Ginzburg parallel of the mirror theorem \cite{Givental-equiv, Givental-mirror, LLY, LLY2, LLY3, LLY4} between Calabi-Yau manifolds established by Givental and Lian-Liu-Yau. For a survey on the LG/LG mirror symmetry and an outline of the current and related works, see \cite{Li-LG}. 
For several purely algebraic constructions of LG A-model and their relationships with FJRW theory, see \cite{PV} and \cite{CLL}.

The LG/LG mirror pairs originate from an old physical construction of Berglund-H\"ubsch \cite{BH} that was completed by Krawitz \cite{K}. Let us briefly review this construction, called the BHK mirror \cite{CR}. Let $W:\mathbb{C}^N \to\mathbb{C}$ be a quasihomogeneous polynomial with an isolated critical point at the origin. We define its \emph{maximal group of diagonal symmetries} to be
\begin{equation}\label{eq:diag}
\Gmax{W} =
\left\{(\lambda_1,\dots,\lambda_N)\in(\mathbb{C}^{\times})^N\Big\vert\,W(\lambda_1\,x_1,\dots,\lambda_N\,x_N)=W(x_1,\dots,x_N)\right\}.
\end{equation}
In the BHK mirror construction,  the polynomial $W$ is required to be \emph{invertible} \cite{K, CR}, i.e., the number of variables must equal the number of monomials of $W$. By rescaling the variables, we can always write $W$ as
\begin{equation}\label{eq:form_of_W}
W=\sum_{i=1}^{N}\prod_{j=1}^{N}x_j^{a_{ij}}.
\end{equation}
We denote its \emph{exponent matrix} {by $\AW_W =\left(a_{ij}\right)_{N\times N}$. The mirror polynomial of $W$ is
\[
W^T=\sum_{i=1}^{N}\prod_{j=1}^{N}x_j^{a_{ji}},
\]
i.e., the exponent matrix $E_{W^T}$ of the mirror polynomial is the transpose matrix of $\AW_W$.

The mathematical LG A-model is the FJRW theory of $(W, \Gmax{W})$, and one geometry of the LG B-model is the Saito-Givental theory of $W^T$, where the genus zero theory is Saito's theory of primitive forms of $W^T$ \cite{Saito-primitive} and the higher genus theory is from the Givental-Teleman's formula \cite{G2,T}.
There is a longstanding conjecture that these A- and B-models are equivalent.

\begin{conjecture}[LG Mirror Symmetry Conjecture]\label{LG_conjecture}
Up to a change of variables, the  generating function of the FJRW theory at all genera for $(W, \Gmax{W})$ can be identified with the generating function of the Saito-Givental theory of $W^T$.
\end{conjecture}

We remark that FJRW theory is defined in \cite{FJR} for any pair $(W, G)$ where $G$ is an \emph{admissible} subgroup of $\Gmax{W}$. 
The BHK mirror construction applies in this more general situation, yielding a mirror partner $(W^T, G^T)$ where $G^T$ is a well-defined group dual to $G$ constructed in \cite{K, BHe}.
Although $G$ is never a trivial group, $G^T$ could be a trivial group. Further more, $G^T$ is a trivial group if and only if $G=\Gmax{W}$. If $G^T$ is nontrivial, we do not know the full mathematical construction of LG B-model for $(W^T, G^T)$.
For this reason, Conjecture \ref{LG_conjecture} is stated only for $\Gmax{W}$ on the A-side.

Krawitz \cite{K} gave a clue to Conjecture \ref{LG_conjecture} by finding an explicit isomorphism that matches the \emph{FJRW ring} of $(W,\Gmax{W})$ to the \emph{Jacobi algebra} of $W^T$ for almost all invertible polynomials (when no variables of $W$ have weight $1/2$); see Theorem \ref{mirror-algebra}.

However, it is much harder to work on the whole conjecture, which requires a thorough understanding of both FJRW theory and Saito-Givental theory. Although 
the powerful Givental-Teleman formula \cite{G2,T} reduces Conjecture \ref{LG_conjecture} to its genus zero part, Conjecture \ref{LG_conjecture} was proved only in a handful of cases previous to our work. These include $A$-type singularities \cite{JKV, FSZ}, 
 $ADE$ (or simple) singularities \cite{FJR}, simple elliptic singularities \cite{KS, MS}, and 
exceptional unimodular singularities \cite{LLSS}.
All these cases were proved with case-by-case calculations and reconstruction on both sides of the mirror. They have one common feature: the \emph{central charge} of $W$ is very small (in fact all no greater than ${7\over6}$). Because of their small central charges, these singularities have special structure, which was used in a critical way to prove Conjecture \ref{LG_conjecture}. The authors of this paper faced several difficulties in extending the techniques from these earlier special results to general singularities with arbitrarily large central charges. We explain two major difficulties and our solutions here:

\begin{enumerate}[leftmargin=*]
\item B-side: a choice of primitive forms. 

Genus zero invariants of the Landau-Ginzburg B-model are determined by a choice of primitive forms, which is equivalent to an appropriate choice of splitting of the Hodge filtration associated to the singularity. Different primitive forms/splittings lead to different  invariants. This is related to the famous holomorphic anomaly \cite{BCOV}. A standard way to identify primitive forms is in terms of the good basis \cite{Saito-primitive} of the Brieskorn lattice of the singularity. Unfortunately, prior to this work, an explicit good basis was known in only a few cases, due to the difficulty of computing higher residue pairings \cite{Saito-residue}. This has long been a major difficulty in studing Landau-Ginzburg B-model invariants. For ADE and exceptional unimodular singularities, there is a unique good basis simply by the degree constraints. For simple elliptic singularities, the correct primitive form is checked in \cite{KS, MS} by hand.

In Theorem \ref{thm-good-basis} of our paper, we identify explicitly a good basis for every invertible polynomial that is mentioned in Conjecture \ref{LG_conjecture}. The good bases defined in Theorem \ref{thm-good-basis} produce the correct genus zero invariants to mirror FJRW theory (Theorem \ref{FM_iso_thm}). To establish Theorem \ref{thm-good-basis}, we find a refinement of the degree constraints using the maximal symmetry group. This allows us to compute higher residue pairings explicitly for all invertible polynomials and study their Hodge-theoretic properties in terms of concrete data. \\

\item A-side: computation of FJRW invariants. 

Aside from the algebraically nice (\emph{concave}) invariants that can be computed by orbifold Grothendieck-Riemann-Roch formulas, it is very difficult to directly compute FJRW invariants.
All previous results have used case-by-case methods to reconstruct genus zero invariants from concave ones. This method grows intractably complex as the central charge increases. In fact, the degree constraints (which is the most powerful tool in small central charges) is no longer useful since in many cases we will inevitably run into some FJRW invariants which are not known how to compute based on the current technique.  

In this work, we systematically explore the combinatorial properties of Landau-Ginzburg models. Using the classification from \cite{KreS} of invertible polynomials in terms of \emph{atomic types} (see Theorem \ref{atomic_thm}), we prove Theorem \ref{reconstruction_theorem}, a strong reconstruction and computation theorem for both A-model and B-model invariants. This theorem states that for any invertible polynomial, the full genus zero data is determined by its Frobenius algebra and some special invariants (called 4-point correlators) which are of atomic type only. The A-side special invariants can be exactly calculated by algebraic methods. 
This solves the A-side computation problem. \end{enumerate}

Our main result in this paper is;

\begin{theorem}\label{main-theorem}
The LG mirror symmetry conjecture holds for all invertible polynomials at all genera when no variables of $W$ have weight $1/2$.
\end{theorem}
We highlight key ingredients toward establishing this theorem here:
\begin{itemize}[leftmargin=*]
\item (Good basis): Krawitz's mirror map sends a natural basis in A-model to a good basis in B-model (Theorem \ref{thm-good-basis}). In this way, the A-model determines a special good basis and corresponding primitive form, which lead to a concise computation in the B-model. 
\item (Vanishing conditions): Because of its deformation-theoretic origins, the B-model has a certain $G_{W^T}$-symmetry, which forces many invariants to vanish (Lemma \ref{vanishing_lemma}). The corresponding invariants vanish on the A-side for geometric reasons. These vanishing results allow us to state pleasant combinatorial properties of non-vanishing invariants (Lemma \ref{properties_lemma}). 
\item (Splitting principle): The vanishing conditions together with the WDVV equations lead to a crucial \emph{Splitting Principle} (Proposition \ref{two_from_a_poly}): in order to compute all genus zero invariants, it is enough to compute the invariants of the (simpler) atomic polynomials with some special properties.
\item (Atomic reconstruction): Using exhaustive reconstruction techniques, we reconstruct all genus zero invariants for atomic polynomials from a few special invariants, listed in part (1) of Theorem \ref{reconstruction_theorem}. Combined with the Splitting Principle, this shows that all genus-zero invariants are determined by these special invariants.
\item (Atomic formulas): We evaluate on both the A and B side the special invariants that remain after the atomic reconstruction. This is parts (2) and (3) of Theorem \ref{reconstruction_theorem}. 

\end{itemize}

Finally, we remark that the special cases left out in Theorem \ref{main-theorem} are only the invertible polynomials containing a special chain summand\footnote{Conjecture \ref{LG_conjecture} is true for the $A_1$-singularity $W=x^2$ although it also contains a weight $1/2$ variable.} $W=x_1^{a_1}x_2+x_2^{a_2}x_3+\dots+x_{N-1}^{a_{N-1}}x_N+x_N^{a_N}$ with $a_N=2$. Our reconstruction result in fact shows that Conjecture \ref{LG_conjecture} holds once the Frobenius algebras and some genus zero 4-point invariants are identified.
Two such examples are the exceptional unimodular singularities $W=Z_{13}, W_{13}$, for which Conjecture \ref{LG_conjecture} was proved in \cite{LLSS}.
The identification of the Frobenius algebras of the other special cases requires the computation of some unknown FJRW invariants. 
\\

\noindent{\bf Outline.} 
In Section \ref{review-AB}, we review the $A$-model FJRW-theory and $B$-model Saito-Givental theory as well as the mirror construction. 
In Section \ref{section-main-result}, we outline the proof of the main theorem via several reconstruction results. 
In Section \ref{section-b-model}, we find a good basis using Krawitz's mirror map and explore combinatorial properties of non-vanishing invariants. 
In Section \ref{reconstruction-1}, we develop technical preparations for our reconstruction theorem, including WDVV equations, Splitting principle, and atomic reconstruction of Fermat type. We also prove the conjecture for Fermat polynomials as a warm-up towards the general cases. 
In Section \ref{computation}, we prove the atomic formulas via explicit calculations on both sides.
In Section  \ref{chap:reconstruction}, we complete a proof of Theorem \ref{main-theorem} via atomic reconstructions of chain type and loop type.\\

\noindent{\bf Acknowledgements.}
We would like to thank Huijun Fan, Tyler Jarvis, Yongbin Ruan, and Kyoji Saito for their insight and consistent support of this project.
We would like to thank Huai-Liang Chang, Jeremy Gu\'er\'e, Changzheng Li for helpful discussions.
In addition, W.H. would like to thank Pedro Acosta, Huazhong Ke, Dusty Ross.
S.L. and Y.S. would like to thank Hiroshi Iritani and Todor Milanov.  W.H.  is partially supported by grant 11901597 of National Natural Science Foundation of China. 
S.L. is partially supported by grant 11801300 of National Natural Science Foundation of China  and grant Z180003 of Beijing Municipal Natural Science Foundation.  
Y.S. is partially supported by Simons Collaboration grant 587119.
R.W. acknowledges the support of a National Science Foundation Graduate Research Fellowship under Grant No. DGE-1247046.
The authors started their collaboration on this project while they were visiting the University of Michigan for the ``Workshop on B-model aspects of Gromov-Witten Theory" in March, 2014. We thank the university for its hospitality.


\section{A review of the $A$- and $B$-models}\label{review-AB}
\subsection{A-model: FJRW theory}\label{sec:FJRW-theory}

One mathematical construction of an LG A-model was given by Fan, Jarvis, and Ruan \cite{FJR, FJR2}, based on a proposal of Witten \cite{W}.
This construction is called FJRW theory after its creators.
Let $W$ be a nondegenerate quasihomogeneous polynomial and let $G$ be an \emph{admissible} group.
 Briefly speaking, FJRW theory is an intersection theory on the moduli space of solutions to the Witten equation on orbifold curves for the pair $(W,G)$.

We begin with a polynomial $W \in \CC[x_1, \ldots x_N]$ that is \emph{quasihomogeneous}; that is, there exist positive rational numbers $q_1, q_2, \ldots, q_N$ such that
\[
W(\lambda^{q_1}x_1, c^{q_2}x_2, \ldots, \lambda^{q_N}x_N) = \lambda\,W(x_1,x_2, \ldots, x_N), \quad \text{for each} \; \lambda \in \CC^\times.
\]
The numbers $q_1, \ldots, q_N$ are called the \emph{weights} of $W$. The \emph{central charge} of $W$, which can be thought of as the ``dimension'' of the LG theory, is defined by
\begin{equation}\label{central-charge}
\chat{W}=\sum_{j=1}^{N}(1-2q_j).
\end{equation}
We call $W$ \emph{nondegenerate} if it has an isolated critical point at the origin and it contains no monomial of the form $x_ix_j$ for $i\neq j$. This implies that each weight is unique and $q_j\in\mathbb{Q}\cap(0,{1\over2}]$ \cite{Saito-quasihomogeneous}. We call a nondegenerate quasihomogeneous $W$ \emph{invertible} if it has the same number of monomials as variables. We say $W$ is a \emph{disjoint sum} of polynomials $W_1$ and $W_2$ and write $W = W_1 \oplus W_2$ if the variables in $W_1$ and $W_2$ are distinct.

All invertible polynomials have been classified by Kreuzer and Skarke.
\begin{theorem}[\cite{KreS}, Theorem 1]\label{atomic_thm}
A polynomial is invertible if and only if it is a disjoint sum of the three following atomic types, where $a\geq2$ and $a_i\geq2$:
\begin{itemize}
\item \emph{Fermat:}	  $ x^a.$
\item \emph{Chain:}		$x_1^{a_1}x_2 + x_2^{a_2}x_3 + \ldots + x_{N-1}^{a_{N-1}}x_N + x_N^{a_N}.$
\item \emph{Loop:}		$x_1^{a_1}x_2 + x_2^{a_2}x_3 + \ldots + x_N^{a_N}x_1.$
\end{itemize}
\end{theorem}

Finally, we define $\Gmax{W}$ to be the \emph{maximal group of diagonal symmetries} of $W$ in \eqref{eq:diag}. Since our goal is to prove the LG Mirror Symmetry Conjecture \ref{LG_conjecture}, in what follows we only discuss the FJRW theory of $(W,\Gmax{W})$ for invertible polynomials $W$ with the form in \eqref{eq:form_of_W}.

\subsubsection{The state space}
The FJRW theory of a pair $(W, \Gmax{W})$ is a \emph{state space}\footnote{The state space of a pair $(W,G)$ is typically denoted $\A{W}{G}$. Because we restrict our attention to $G=\Gmax{W}$, we will consistently drop the group from our notation.} $\A{W}$ and a cohomological field theory $\{\Lambda_{g,k}^W\}$, which is a set of linear maps
$$\Lambda_{g,k}^W: (\A{W})^{\otimes k}\to H^*(\M{g}{k})$$
for  $2g-2+k>0$.
Here $\M{g}{k}$ is the moduli space of stable $k$-pointed curves of genus $g$. The state space is defined as
\[
\A{W}= \bigoplus_{\gamma \in \Gmax{W}} \A{\gamma} \qquad \text{where} \qquad \A{\gamma}:= \left( H^{N_{\gamma}}({\rm Fix}(\gamma), W_{\gamma}^{\infty}; \CC )\right)^{\Gmax{W}}.
\]
Here ${\rm Fix}(\gamma)$ is the fixed locus of $\gamma$ and $N_\gamma$ is its dimension as a $\CC$-vector space. Furthermore, $W_{\gamma}$ is the restriction of $W$ to ${\rm Fix}(\gamma)$, and $W_{\gamma}^{\infty}$ is ${\rm Re}(W_{\gamma})^{-1}((M, \infty))$ for $M\gg0$. Thus, $\A{W}$ is the dual to the space of Lefschetz thimbles.

For each class $\xi \in \A{\gamma}$, we call $\gamma$ the \emph{sector} of $\xi$. If ${\rm Fix}(\gamma)=0\in\CC^N$,  we say that $\gamma$ is \emph{narrow}; otherwise we say it is \emph{broad}. Note that if $\gamma$ is narrow then $\A{\gamma}$ is 1-dimensional.

There is an alternative expression for $\A{W}$. Let
\[
\Jac{W} = \CC[x_1, \ldots x_N]/\left({\partial W\over\partial x_1},\ldots,{\partial W\over\partial x_N}\right)
\]
be the \emph{Jacobi algebra} of $W$.
It is a theorem of Wall (see \cite{Wa1} and \cite{Wa2}) that the vector space $H^{N_{\gamma}}(\CC^N, W_{\gamma}^{\infty}; \CC )$ is isomorphic to $\Jac{W_{\gamma}}d\mathbf{x}_{\gamma}$, where $d\mathbf{x}_{\gamma}$ is the product of the differentials of the variables fixed by $\gamma$.
Thus,
\[
\A{W}=\bigoplus_{\gamma\in \Gmax{W}}\left( \Jac{W_{\gamma}}d\mathbf{x}_{\gamma}\right)^{\Gmax{W}}.
\]
With this identification, we write $\xi = \lceil m\; ; \; \gamma\rfloor$ where $\xi$ corresponds to the monomial $m\in \Jac{W_{\gamma}}$.

We define a grading on $\A{W}$ as follows.
Since $\Gmax{W}$ is a finite abelian group, for any element $\gamma \in \Gmax{W}$, we may write $\gamma =\left(\exp(2 \pi \sqrt{-1} \Theta_\gamma^{(1)}), \ldots, \exp(2 \pi \sqrt{-1} \Theta_\gamma^{(N)})\right)$ for some unique $\Theta_\gamma^{(j)} \in [0,1)$. The number $\Theta_\gamma^{(j)}$ is called the \emph{$j$-th phase} of $\gamma$.
For $\xi\in \A{\gamma}$, we define
\begin{equation}\label{A-model-degree}
\deg_W(\xi) = {1\over 2}N_\gamma+\sum_{j=1}^N(\Theta_\gamma^{(j)}-q_j).
\end{equation}
Note that the degree of $\xi$ depends only on its sector.

We have a pairing $\eta_{\gamma}: \A{\gamma} \times \A{\gamma^{-1}} \rightarrow \CC$ which is induced by the intersection pairing on Lefschetz thimbles. The direct sum of these pairings gives us a nondegenerate pairing
$$\langle \; , \; \rangle: \A{W} \times \A{W} \rightarrow \CC.$$
Under the identification of $\A{\gamma}$ with $ (\Jac{W_{\gamma}}d\mathbf{x}_{\gamma})^{\Gmax{W}}$, this pairing is equal to the residue pairing on differential forms. See \cite{Ce, CIR} for expositions of this fact.

\subsubsection{The cohomological field theory}
The construction of the cohomological field theory $\{\Lambda_{g,k}^W\}$ is highly nontrivial. We will only summarize it here, and refer the interested reader to the original papers \cite{FJR} and \cite{FJR2} for more details.

The construction uses the moduli space of stable $W$-orbicurves. Let $\curve$ be a stable orbicurve of genus $g$ with marked points $p_1, \ldots, p_k$.
At each marked point and node we have a local chart $\CC/\ZZ_m$ for some positive interger $m$.
We require that the actions on the two branches of a node be inverses.

Let $\rho: \curve \rightarrow \plaincurve$ be the forgetful morphism from the orbifold curve $\curve$ to the underlying coarse curve.
If $W=\sum_{i=1}^{n}\prod_{j=1}^{n}x_j^{a_{ij}}$ is invertible, a $W$-structure consists of data $(\curve, \BigL)$ where $\BigL$ is a set of orbifold line bundles $\{\LL_1, \ldots \LL_N\}$ over $\curve$ satisfying
\[
\bigotimes_{j=1}^{N} \LL_j^{\otimes \expon_{ij}} \cong \rho^*\left(K_{\plaincurve} \otimes\bigotimes_{j=1}^{k} \OO(p_j)\right) \qquad \text{for each} \; i,
\]
where $K_{\plaincurve}$ is the canonical bundle of $\plaincurve$ and $\OO(p_j)$ is the holomorphic line bundle of degree one whose sections may have a simple pole at $p_j$.

If the local group at a marked point of an orbicurve is $\ZZ_m$, the line bundles $\LL_1, \ldots \LL_N$ induce a representation $\ZZ_m \rightarrow (\CC^{\times})^N$. The representation is required to be faithful. The image of this representation will always be in $\Gmax{W}$. The image of $1\in\ZZ_m$ singles out some $\gamma\in\Gmax{W}$ at each marked point; these group elements are called the \emph{decorations}.

Given an invertible polynomial $W$, the moduli space of pairs $(\curve, \BigL)$ is called the \emph{moduli space of stable $W$-orbicurves} and denoted by $\W{g}{k}$. According to \cite{FJR}, it is a Deligne-Mumford stack, and there is a forgetful morphism ${\rm st}: \W{g}{k}\to\M{g}{k}$. The forgetful morphism is flat, proper, and quasi-finite (see Theorem 2.2.6 of \cite{FJR}).
The decorations $\gamma_i$ at the marked points $p_i$ decompose $\W{g}{k}$ into open and closed substacks $\W{g}{k}(\gamma_1,\ldots,\gamma_k)$.
Furthermore, the stack $\W{g}{k}(\gamma_1,\ldots,\gamma_k)$ is stratified, and each closure in it is denoted by $\W{g}{k}(\Gamma_{\gamma_1,\ldots,\gamma_k})$ for some $\Gamma_{\gamma_1,\ldots,\gamma_k}$. Here $\Gamma_{\gamma_1,\ldots,\gamma_k}$ is called a \emph{$\Gmax{W}$-decorated dual graph} of an underlying stable curve of genus $g$ and $k$ marked points. We call $\Gamma_{\gamma_1,\ldots,\gamma_k}$ \emph{fully $\Gmax{W}$-decorated} if we assign some $\gamma_+\in\Gmax{W}$ and $\gamma_-=\gamma_+^{-1}$ on two sides of each node.

In \cite{FJR2} the authors perturb the polynomial $W$ to polynomials of Morse type and construct virtual cycles from the solutions of perturbed Witten equations. That is, they construct
$$[\W{g}{k}(\Gamma_{\gamma_1,\ldots,\gamma_k})]^{\rm vir}\in H_{*}(\W{g}{k}(\Gamma_{\gamma_1,\ldots,\gamma_k}),\CC)\otimes\prod_{j=1}^{k}\A{\gamma_j}.$$
As a consequence, they obtain a cohomological field theory $\{\Lambda_{g,k}^W:\A{W}^{\otimes k}\to H^*(\M{g}{k},\CC)\}$ with a flat identity $\lceil 1\; ; \; \J{W}\rfloor$, where
\[
\Lambda_{g,k}^W(\xi_1, \ldots, \xi_k)
:={|\Gmax{W}|^g\over \deg({\rm st})} {\rm PD}\, {\rm st}_*\left([\W{g}{k}(\gamma_1,\ldots,\gamma_k)]^{\rm vir}\cap\prod_{j=1}^{k}\xi_j\right)
\in H^*(\M{g}{k}).
\]
Here $ {\rm PD}$ is the Poincar\'e dual and $\J{W}$ is the \emph{exponential grading operator}, defined by
\begin{equation}\label{exponent-op}
\J{W} =\left(\exp(2\pi\sqrt{-1}q_1), \ldots,\exp(2\pi\sqrt{-1}q_N)\right)\in \Gmax{W}.
\end{equation}

\subsubsection{The FJRW potential}
The cohomological field theory allows us to define \emph{FJRW invariants} (or \emph{genus-$g$ $k$-point correlators}) as
\[
\langle\xi_1\psi_1^{\ell_1},\cdots,\xi_k\psi_k^{\ell_k}\rangle_{g}^{W}
=\int_{\M{g}{k}}\Lambda_{g,k}^{W}(\xi_1,\cdots,\xi_k) \prod_{i=1}^k\psi_i^{\ell_i}.
\]
Here $\psi_i:=c_1(L_i)$ is the $i$-th psi class, where $L_i$ is the $i$-th tautological line bundle on $\M{g}{k}$.
The invariant is \emph{primary} if there are no psi classes, i.e., $\ell_i=0$ for all $1\leq i\leq k.$ We call the classes $\xi_1, \ldots, \xi_k$ the \emph{insertions} of the correlator.

The FJRW invariants induce various structures on $\A{W}$.
The pairing $\langle,\rangle$ and the primary genus-zero $3$-point correlators define a product $\star$ on $\A{W}$, by
\begin{equation}\label{product_def}
\langle\alpha \star \beta, \gamma\rangle = \langle\alpha,\beta,\gamma\rangle^W_{0}, \quad \text{where} \; \alpha, \beta, \gamma \in \A{W}.
\end{equation}
This definition makes the pairing Frobenius with respect to $\star$, so that the {\rm FJRW} ring $(\A{W}, \star)$ is a commutative and associative Frobenius algebra with the unit $\lceil 1\; ; \; \J{W}\rfloor$.

The primary genus-zero correlators define a Frobenius manifold structure on $\A{W}$. Let $\mathcal{B}$ be a set whose elements are a basis for $\A{W}$. The pre-potential of the Frobenius manifold is
\begin{equation}\label{eq:prepotential}
\pot^{\rm FJRW}_{0,W} = \sum_{k \geq 3} \sum_{(\xi_1, \ldots, \xi_k) \in \mathcal{B}^k }\langle \xi_1, \ldots, \xi_k \rangle_{g}^{W} \frac{t_{\xi_1} \ldots t_{\xi_k}}{k!}.
\end{equation}

The Frobenius manifold pre-potential encodes the genus-0 data of the FJRW theory of $(W,G)$. The FJRW invariants of all genera are encoded in the \emph{total ancestor {\rm FJRW}-potential}
\[
\apot{\rm FJRW}{W} = \exp\left( \sum_{g \geq 0} \hbar^{g-1} \sum_{k \geq 0} \langle \xi_{i_1}\psi_1^{l_1}, \ldots, \xi_{i_k}\psi_k^{l_k} \rangle_{g}^{W} \frac{t_{i_1, l_1} \ldots t_{i_k, l_k}}{k!}\right).
\]

\subsubsection{Properties of the A-model}
Several properties of FJRW theory will be useful in our proof of Landau-Ginzburg mirror symmetry.
First, the following theorem tells us how the FJRW theory of $W$ behaves when $W=W_1\oplus W_2$ is a disjoint sum of the atomic polynomials in Theorem \ref{atomic_thm}. In this case, $\Gmax{W}=\Gmax{W_1}\times \Gmax{W_2}.$

\begin{theorem}[\cite{FJR}, Theorem 4.2.2]\label{tensor_a_thm}
Let $W_1$ and $W_2$ be invertible polynomials with no variables in common. Then as Frobenius algebras,
\[
\A{W_1} \otimes \A{W_2} \cong \A{W_1\oplus W_2}
\]
via the isomorphism $(\lceil m \; ; \; \gamma \rfloor, \lceil n \; ; \; \delta \rfloor) \mapsto \lceil mn \; ; \; \gamma\delta \rfloor $. Moreover,
\begin{align*}
&\Lambda_{g,k}^{W_1\oplus W_2}(\lceil m_1n_1 \; ; \; \gamma_1\delta_1 \rfloor, \; \ldots, \; \lceil m_kn_k \; ; \; \gamma_k\delta_k \rfloor)\\
&=\Lambda_{g,k}^{W_1}(\lceil m_1 \; ; \; \gamma_1 \rfloor, \, \ldots, \,\lceil m_k \; ; \; \gamma_k \rfloor)\, \, \Lambda_{g,k}^{W_2}(\lceil n_1 \; ; \; \delta_1 \rfloor, \, \ldots,\, \lceil n_k \; ; \; \delta_k \rfloor).
\end{align*}
\end{theorem}

\begin{remark}\label{remark-stablization}
The LG mirror symmetry conjecture \ref{LG_conjecture} is known for $A_1$-singularity $W=x^2$. Also if $W=x^2$ then $\A{W}=\CC$. Theorem \ref{tensor_a_thm} implies that for any invertible polynomial $W$, the FJRW theory is invariant under the stabilization $W \to W\oplus y^2$. Because of these facts, from now on, we will assume $a\geq3$ for the Fermat polynomial $x^a$.
\end{remark}

Second, certain vanishing properties of the FJRW correlators will be critical when we reconstruct the pre-potential in \eqref{eq:prepotential}.
In the A-model, these come from two of the so-called correlator ``axioms'', which are summarized in the following proposition.

\begin{proposition}[\cite{FJR}, Proposition 2.2.8 and Theorem 4.1.8]
Let $\xi_i\in \A{\gamma_i}$ and let $\Theta^{(j)}_{\gamma_i}$ be the $j$-th phase of $\gamma_i\in \Gmax{W}$.
If $\langle \xi_1, \ldots, \xi_k \rangle_0^W\neq0$, then the following equalities hold:
\begin{eqnarray}
\sum_{i=1}^k \deg_W(\xi_i) &=& \chat{W} + k -3. \label{dim_ax}\\
\ld_j &\coloneqq & q_j(k-2) - \sum_{i=1}^k \Theta_j^{\gamma_i}\in\mathbb{Z} \quad \text{for}\;\; j=1, \cdots, N. \label{lbd_ax}
\end{eqnarray}
\end{proposition}

Formula \eqref{dim_ax} is called the \emph{Dimension Axiom} because it is a consequence of the degree of the class $\Lambda_{0}^{W}(\xi_1, \ldots, \xi_k)$.
Formula \eqref{lbd_ax} is called the \emph{Integer Degree Axiom} because $\ld_j$ is the degree of the line bundle $\rho_*\LL_j$ on the underlying coarse curve, when that curve is smooth.
Formula \eqref{lbd_ax} follows from the fact that line bundles must have integer degrees, so if $\ld_j \not \in \ZZ$ then the corresponding component of $\W{0}{k}$ is empty. We  call $\ld_j$ the $j^{th}$ \emph{line bundle degree} of $\langle \xi_1, \ldots, \xi_k \rangle_0^W$.

\begin{remark}\label{rmk:lbd_determines_sector}
One useful application of formula \eqref{lbd_ax} is due to Krawitz: if the correlator $\langle \xi_1, \,\xi_2,\, \xi_3 \rangle^W_0$ is nonzero and $\xi_i\in \A{\gamma_i}$, then $\gamma_3 = \J{W}(\gamma_1\gamma_2)^{-1}$. Then from (\ref{product_def}) and the definition of the pairing, $\xi_1 \star \xi_2 \in \A{\gamma_1\gamma_2\J{W}^{-1}}$.
\end{remark}

In the remainder of this paper, we will only use primary genus-zero correlators, so we will drop the genus-subscript $g$ from the correlator notation. Moreover, when context makes the polynomial clear we will suppress $W$, writing a genus-0 A-model correlator as $\langle \xi_1, \ldots, \xi_k \rangle$.


\subsection{B-model: Saito-Givental theory}\label{section-B-model}
 In this section, we follow the B-model convention and use $f$ for a quasihomogeneous polynomial with isolated singularity at the origin:
\[
    f(\lambda^{p_1}x_1,\cdots, \lambda^{p_N}x_N)=\lambda f(x_1,\cdots, x_N).
\]
Outside of this section, $f\equiv W^T$, and $p_i\equiv q_i^T$ is the weight of $x_i$ in $W^T$.

The central charge of $f$ is $\hat c_f=\sum_i (1-2p_i)$. We will always let $d^Nx\equiv dx_1\wedge\cdots\wedge dx_N$.

The Frobenius algebra structure of the B-model is simply $\Jac{f}$ with the grading coming from the quasihomogeneous weights, equipped with the residue pairing.
Note that  $\Jac{f_1\oplus f_2} = \Jac{f_1} \otimes \Jac{f_2}$ (compare Theorem \ref{tensor_a_thm}).

The genus zero invariants (or the Frobenius manifold structure)  are induced from Saito's theory of primitive forms \cite{Saito-primitive}. Since the Frobenius manifold is generically semisimple, the higher genus invariants are given by the famous Givental-Teleman formula \cite{G2, T}.

\subsubsection{Saito's triplet for primitive forms: Brieskorn lattice, higher residue pairing and the good basis.}
Here we review the basics of Saito's theory of primitive forms.
Because we wish to prove Conjecture \ref{LG_conjecture}, we will only discuss the theory for quasihomogeneous $f$.
See \cite{Saito-primitive, Saito-existence, Saito-uniqueness} for discussions of arbitrary isolated singularities.

Let $\Omega^k_{\CC^N,0}$ be the space of germs of holomorphic $k$-forms at the origin in $\CC^N$. Define
$$
     \mathcal H_f^{(0)}=\Omega^N_{\CC^N,0}\llbracket z\rrbracket/(df\wedge+zd)\Omega^{N-1}_{\CC^N,0}
$$
which is a formally completed version of the Brieskorn lattice associated to $f$ (see \cite{Saito-residue}). Here $z$ is a formal variable.
There exists a natural semi-infinite Hodge filtration on $\mathcal H_f^{(0)}$ given by $\mathcal H_f^{(-k)}:= z^k \mathcal H_f^{(0)}$ such that
$$
    \mathcal H_f^{(-k)}/\mathcal H_f^{(-k-1)}\simeq \Omega_f, \quad \mbox{where}\ \Omega_f:=\Omega^N_{\CC^N,0}/df \wedge \Omega^{N-1}_{\CC^N,0}.
$$
We define a natural $\QQ$-grading, or \emph{weight}, on $\Jac{f}$, on $\mathcal H_f^{(0)}$, and on $\Omega_f$ which is generated by
\begin{equation}\label{defn-weight}
  \wt(x_i)=q^T_i, \quad \wt(dx_i)=q^T_i, \quad \wt(z)=1.
\end{equation}
For a homogeneous element of the form $\eta=z^k \phi(x_i)d^Nx$, we have
$$
   \wt(\eta)=\wt(\phi)+k+\sum_{i=1}^N q^T_i.
$$

In \cite{Saito-residue}, K. Saito constructs a \emph{higher residue pairing}
$
   K_f: \mathcal{H}_f^{(0)}\otimes \mathcal{H}_f^{(0)}\to z^{N}\CC[[z]]
$
satisfying the following properties.
\begin{enumerate}
\item $K_f$ is equivariant with respect to the $\QQ$-grading, i.e.,
$$
   \wt(K_f(\alpha, \beta))=\wt(\alpha)+\wt(\beta)
$$
for homogeneous elements  $\alpha, \beta\in \mathcal{H}_f^{(0)}$.
\item $K_f(\alpha, \beta)=(-1)^N \overline{K_f(\beta,\alpha)}$, where the bar operator takes $z\to -z$.
\item $K_f(v(z)\alpha, \beta)=K_f(\alpha, v(-z)\beta)=v(z)K_f(\alpha,\beta)$ for $v(z)\in \CC[[z]]$.
\item The leading $z$-order of $K_f$ defines a pairing
$$
    \mathcal{H}_f^{(0)}/z  \mathcal{H}_f^{(0)}\otimes  \mathcal{H}_f^{(0)}/z  \mathcal{H}_f^{(0)} \to \CC, \quad \alpha\otimes \beta\mapsto \lim_{z\to 0}z^{-N}K_f(\alpha,\beta)
$$
which coincides with the usual residue pairing
$
  \Omega_f\otimes \Omega_f \to \CC.
$
\end{enumerate}
The last property implies that $K_f$ defines a semi-infinite extension of the residue pairing, which explains the name ``higher residue". Following \cite{Saito-primitive}, we define a good section and a good basis for $f$.

\begin{definition}[Good basis] A good section $\sigma$ is a splitting of the projection $\mathcal H_f^{(0)}\to \Omega_f$,
\[
   \sigma: \Omega_f\to \mathcal H_f^{(0)},
\]
such that $\sigma$ preserves the $\QQ$-grading, and $K_f(\mbox{Im}(\sigma), \mbox{Im}(\sigma))\subset z^N \CC$.
A basis of the image $\mbox{Im}(\sigma)$ of a good section $\sigma$ is a good basis of $\mathcal H_f^{(0)}$ (or $f$).
\end{definition}

Equivalently, a good basis consists of  homogeneous elements $\{\eta_\alpha\}\subset \mathcal H_f^{(0)}$ such that $\{\eta_\alpha\}$ represents a basis of $\Omega_f$ and $K_f(\eta_\alpha, \eta_\beta)\in z^N \CC$ for all $\alpha$ and $\beta$.

\begin{eg}\label{eg-ADE}
The ADE singularities are those for which $\hat c_f<1$. For these singularities any homogeneous basis of $\Omega_f$ is a good basis, and any two such choices are ``equivalent'' (i.e. there exists a unique good section) \cite{Saito-primitive}.
\end{eg}

\begin{proposition}\label{prop-sum-good-basis}
Let $f(\x,\y)=f_1(\x)\oplus f_2(\y)$ be the disjoint sum of two isolated quasihomogeneous singularities, where $\x=\{x_1,\cdots,x_{N_1}\}$ and $\y=\{y_1,\cdots,y_{N_2}\}$. If $\{\eta_i(\x)\}_{i \in I}$ and $\{\varphi_\alpha(\y)\}_{\alpha \in A}$  are good bases of $\mathcal H_{f_1}^{(0)}$ and $\mathcal H_{f_2}^{(0)}$ respectively, then $\{\eta_i(\x)\varphi_\alpha(\y)\}_{(i, \alpha) \in I \times A}$ is a good basis of $\mathcal H_f^{(0)}$.
\end{proposition}
\begin{proof} It follows from the construction of the higher residue pairing in \cite{Saito-residue} that
$$
   K_f(\eta_i(\x)\varphi_\alpha(\y), \eta_j(\x)\varphi_\beta(\y))=\pm K_{f_1}(\eta_i(\x), \eta_j(\x)) K_{f_2}(\varphi_\alpha(\y), \varphi_\beta(\y)).
$$
The proposition is a direct consequence of this equality.
\end{proof}

A good basis is not unique in general.
Landau-Ginzburg mirror symmetry favors a particular choice of good basis, which we call the \emph{standard basis}.
This basis was used by Krawitz in \cite{K} to describe the mirror map between Frobenius algebras.
We define the standard basis for an atomic polynomial below, and we get a basis for a general invertible polynomial with Proposition \ref{prop-sum-good-basis}.

In this definition and later, we use $\Hessbase{f}$ to denote the element of the standard basis that spans the 1-dimensional subspace of $\Jac{f}$ of highest degree. It is a fact that $\wt(\Hessbase{f}) = \chat{f}$.

\begin{definition}\label{milnor_basis}
The \emph{standard basis} of an atomic polynomial $f$ is $\{\phi_\alpha\}_{\alpha=1}^\mu$, where $\mu=\dim_{\mathbb{C}}\Jac{f}$, $\phi_{\mu} = \Hessbase{f}$, and the monomials $\phi_{\alpha}$ are defined as follows.
\begin{itemize}
\item If $f=x^a$ is a Fermat, then $\{\phi_{\alpha}\} = \{x^{r} \mid 0\leq r\leq a-2\}$ and $\Hessbase{f} = x^{a-2}$.
\item If $f = x_1^{a_1}+x_1x_2^{a_2}+\dots+x_{N-1}x_N^{a_N}$ is a chain, then
\[
\{\phi_{\alpha}\} = \left \{\prod_{i=1}^{N}x_i^{r_i}\right\}_{\mathbf{r}}  \qquad \text{and} \qquad \Hessbase{f} = x_N^{a_N-2}\prod_{i=1}^{N-1}x_i^{a_i-1},
\]
 where $\mathbf{r}=(r_1,\cdots,r_N)$ with $r_i\leq  a_i-1$ for all $i$ and $\mathbf{r}$  is not of the form $ (*, \cdots, *, k, a_{N-2l}-1,  \cdots, 0, a_{N-2}-1, 0, a_{N}-1)$ with $k\geq 1$.
\item If $f=x_1^{a_1}x_N+x_1x_2^{a_2}+\dots+x_{N-1}x_N^{a_N}$ is a loop, then
\[
\{\phi_{\alpha}\} =\left \{\prod_{i=1}^{N}x_i^{r_i} \; \middle | \; 0\leq r_i<a_i \right\} \qquad \text{and} \qquad \Hessbase{f} = \prod_{i=1}^{N}x_i^{a_i-1}.
\]
\end{itemize}
\end{definition}

%
%
%
%

Because we are interested in mirror symmetry, the forms here are dual to the forms in Theorem \ref{atomic_thm}
See Example \ref{eg-mirror} for further clarification.
\begin{theorem}\label{thm-good-basis}
The standard basis in Definition \ref{milnor_basis} is a good basis of $f$.
\end{theorem}
This theorem will be proved in Section \ref{section-proof-good-basis}.

\begin{definition}\label{def-nRes}
We define the normalized residue $\nRes$ on $\Jac{f}$ by setting $\nRes(\phi_f)=1$. It induces a pairing $\eta$  on $\Jac{f}$ defined by
$
  \eta_{\alpha\beta}=\nRes(\phi_\alpha \phi_\beta).
$
\end{definition}


As shown in \cite{Saito-primitive}, a good basis of $f$ gives rise to a \emph{primitive form}, which is a certain family of holomorphic volume forms with respect to a universal unfolding of $f$.
The primitive form induces a Frobenius manifold structure on $\Jac{f}$ (which was called a \emph{flat structure} in \cite{Saito-primitive}). We will not give the precise definition of primitive form here. Instead, we present a perturbative description developed in \cite{LLSaito, LLSS} which is a formal solution of the Riemann-Hilbert-Birkhoff problem described in \cite{Saito-primitive}.
We also use the perturbative description of the primitive form to compute the invariants of our Landau-Ginzburg B-model.

\subsubsection{A perturbative formula}\label{section-perturbative-formula}
Given a polynomial $g(x)$, we will denote $[g(x)d^Nx]$ its class in $\mathcal H_f^{(0)}$ in this section. Let polynomials $\{\bgen_{\alpha}\}$ represent a basis of $\Jac{f}$ such that $\{[\bgen_{\alpha}d^Nx]\}$ is a good basis for $\mathcal H_f^{(0)}$.

Let $B$ denote the subspace $\spn_{\CC}\{[\bgen_{\alpha}d^Nx]\}$ of $\HH_f^{(0)}$, and $\HH_f=\HH_f^{(0)}\otimes_{\CC\llbracket z \rrbracket} \CC((z))$ be the Laurent extension. Then
$$
\HH_f^{(0)} = B\llbracket z\rrbracket \quad \text{and} \quad \HH_f=B((z)).
$$
Let $\mathbf{s}=\{\mathbf{s}_\alpha\}$ be the linear coordinates on $\Jac{f}$ dual to the basis $\{\bgen_\alpha\}$, so the coordinates $\mathbf{s}$ parametrize a local universal deformation $F=f+\sum_\alpha s_\alpha \bgen_\alpha$ of $f$. The following formula gives a perturbative way to compute the associated primitive form.

\begin{theorem}[\cite{LLSS}, Theorem 3.7]\label{B_theorem}
There is a unique pair $(\zeta, \Jfunc)$ with $\zeta \in B\llbracket z\rrbracket \llbracket \mathbf{s}\rrbracket$ and $\Jfunc \in [d^Nx]+z^{-1}B[z^{-1}]\llbracket\mathbf{s}\rrbracket$ such that
\begin{equation}\label{Jfunc_eq}
e^{(F-f)/z} \zeta(z, \mathbf{s}) = \Jfunc \quad \text{in}\; \HH_f\llbracket\mathbf{s}\rrbracket.
\end{equation}
Here $B[z^{-1}]\llbracket\mathbf{s}\rrbracket$ is formal power series in $\mathbf{s}$ valued in $B[z^{-1}]$.
\end{theorem}

Furthermore, $\zeta$ is the series expansion in $\mathbf{s}$ of the \emph{primitive form} associated to the good basis $B$, and $\Jfunc$ plays the role of the FJRW J-function in the following sense. By Theorem \ref{B_theorem}, we may write
\begin{equation}\label{J-func}
\Jfunc = \left[d^Nx\left(1 + z^{-1} \sum_{\alpha}\Jfunc_{-1}^{\alpha}\bgen_{\alpha} + z^{-2}\sum_{\alpha}\Jfunc_{-2}^{\alpha}\bgen_{\alpha} + \ldots\right)\right].
\end{equation}
Let
\begin{equation}\label{flat-coord}
t_{\alpha}(\mathbf{s}) = \Jfunc_{-1}^{\alpha}(\mathbf{s}) \in \CC\llbracket\mathbf{s}\rrbracket.
\end{equation}
We call $\mathbf{t}=\{t_{\alpha}\}$ the \emph{flat coordinates} for $\Jac{f}$. In fact
\begin{equation}\label{first-order}
t_{\alpha} = s_{\alpha} + \OO(\mathbf{s}^2),
\end{equation}
and we may write each $s_{\alpha}$ as a function of $\mathbf{t}$. Then in terms of the flat coordinates, the Frobenius manifold prepotential $\pot_{0,f,\prim}^{\rm SG}$ associated to the primitive form $\zeta$ satisfies
\begin{equation}\label{compute_potential_eq}
\partial_{t_\alpha}\pot_{0,f,\prim}^{\rm SG}(\mathbf{t})=\sum_{\beta}\eta_{\alpha,\beta}\Jfunc_{-2}^{\beta}(\mathbf{t})
\end{equation}
where $\eta$ is the matrix in Definition \ref{def-nRes}. The B-model correlators are defined via
\begin{equation}\label{sg-potential-g=0}
\langle \bgen_{\alpha_1}, \;\ldots, \; \bgen_{\alpha_k}\rangle = \frac{ \partial^k \pot_{0,f,\prim}^{\rm SG}}{\partial t_{a_1}\ldots \partial t_{a_k}} (0).
\end{equation}


The proof of Theorem \ref{B_theorem} in \cite{LLSS} outlines an algorithm for recursively solving $\zeta$ and $\Jfunc$ as follows. Let $\zeta_{(\leq k)}$ be the $k$-th Taylor expansion in terms of $\mathbf{s}$. To zeroth order (in $\mathbf{s}$), equation \eqref{Jfunc_eq} is
\[
\zeta_{(\leq 0)}= [d^Nx]+z^{-1}B[z^{-1}].
\]
Because $\zeta$ has only positive powers of $z$, this is uniquely solved by $\zeta_{(\leq 0)} = [d^Nx]$. Suppose we have solved for $\zeta_{(\leq k)}$, which satisfies
\[
e^{(F-f)/z} \zeta_{(\leq k)} \in [d^Nx]+z^{-1}B[z^{-1}]\llbracket\mathbf{s}\rrbracket \quad \text{modulo} \; \mathbf{s}^{k+1}.
\]
Let $R_{k+1}$ be the $(k+1)^{th}$-order component of $e^{(F-f)/z} \zeta_{(\leq k)}$. Let $R_{k+1} = R^+_{k+1}+R^-_{k+1}$ where $R^+_{k+1}$ is the part with nonnegative powers of $z$. Then $\zeta_{(\leq k+1)} = \zeta_{(\leq k)}-R^+_{k+1}$ uniquely solves Equation (\ref{Jfunc_eq}) up to order $k+1$ in $\mathbf{s}$.

\subsubsection{B-model Saito-Givental potential}
Saito's theory of primitive forms gives the genus zero invariants (see Formula \eqref{sg-potential-g=0}) in the LG B-model. For higher genus, Givental \cite{G2} proposed a remarkable formula for the total ancestor potential of a semi-simple Frobenius manifold.
The uniqueness of Givental's formula was established by Teleman \cite{T}.
According to the work of Milanov \cite{M}, the total ancestor potential can be extended uniquely to the origin,  which is a nonsemisimple point we are interested in.

Saito's genus zero theory together with the total ancestor potential is now referred to as the Saito-Givental theory of a singularity.
We will call the extended total ancestor potential at the origin a Saito-Givental potential and denote it by $\mathscr{A}^{\rm SG}_{f,\zeta}$, where the subscript $\zeta$ shows its dependence on the chosen primitive form $\zeta$.


\subsection{Krawitz's mirror map}\label{section-Krawitz-mirror}
Recall that given an invertible polynomial
$$W = \sum_{i=1}^N\prod_{j=1}^N x_j^{\expon_{ij}},$$
its \emph{exponent matrix} is
$
\AW_W:= (\expon_{ij})_{N\times N},
$
and the mirror polynomial (also called the transpose polynomial) $W^T$ is defined by $\AW_{W^T} = (\AW_W)^T$, so
$$W^T = \sum_{j=1}^N\prod_{i=1}^N x_i^{\expon_{ij}}.$$
The inverse matrix $\AW_W^{-1}$ plays an important role in the mirror map constructed by Krawitz in \cite{K}. Let us write
\begin{equation}\label{exponent-inverse}
\AW_W^{-1} =
\left( \begin{array}{ccc}
\rho_1^{(1)} & \cdots & \rho_N^{(1)} \\
\vdots  & \vdots & \vdots  \\
\rho_1^{(N)} & \cdots & \rho_N^{(N)}
\end{array} \right),
\end{equation}
and define
\begin{eqnarray*}
&\col_j:=\left(\exp(2\pi\sqrt{-1}\rho_j^{(1)}), \ldots, \exp(2\pi\sqrt{-1}\rho_j^{(N)})\right),\\
&\row_j:=\left(\exp(2\pi\sqrt{-1}\rho_1^{(j)}), \ldots, \exp(2\pi\sqrt{-1}\rho_N^{(j)})\right).
\end{eqnarray*}
According to \cite{K}, the group $G_W$ is generated by $\{\col_j\}_{j=1}^N$ and $G_{W^T}$ is generated by $\{\row_j\}_{j=1}^N$.
Recall $q_j$ is the weight of $x_j$ in $W$. Let $q_j^T$ be the weight of $x_j$ in $W^T$. We remark that
\begin{equation}\label{weight}
q_j=\sum_{i=1}^{N}\rho_i^{(j)} \quad \text{and} \quad
q_j^T=\sum_{i=1}^{N}\rho_j^{(i)}.
\end{equation}

\begin{eg}\label{eg-mirror}  The transpose of the chain polynomial in Theorem \ref{atomic_thm} is $x_1^{a_1}+x_1x_2^{a_2}+\cdots + x_{N-1}x_N^{a_N}.$
The transpose of the loop polynomial is $x_1^{a_1}x_N+x_2^{a_2}x_1+\cdots+ x_{N-1}x_N^{a_N}.$
\end{eg}

The next theorem defines Krawitz's mirror map. Its proof consists of Theorems 2.4 and Theorem 3.1 in \cite{K}, Theorem 2.3 in \cite{A}, and Remark \ref{rmk:lbd_determines_sector}.
\begin{theorem}[\cite{K}, Krawitz's mirror map]\label{mirror-algebra}
Let $W$ be an invertible polynomial with no chain variables of weight $1/2$.
Then the ring homomorphism
$\mirror: \Jac{W^T} \rightarrow (\A{W},\star)$
generated by
\begin{equation}\label{Krawitz-mirror-map}
\mirror(x_i) =
\left\{
\begin{array}{ll}
\lceil \,x_i\; ; \; 1\rfloor, & \textit{if}\ x_i\ \textit{is a variable in a 2-variable loop summand with }\expon_{i}=2, \\
\lceil 1\; ; \; \col_i\cdot J_W\rfloor, & otherwise.
\end{array}
\right.
\end{equation}
 is a degree-preserving isomorphism of Frobenius algebras, in the sense that $\wt(\bgen) = \deg_W(\mirror(\bgen))$ for every monomial $\bgen \in \Jac{W^T}$.
Furthermore,
\begin{equation}\label{equ:mirror_sectors}
\mirror\left( \prod_{j=1}^{N} x_j^{\alpha_j} \right) \in \A{\gamma} \quad \text{where} \quad \gamma = \prod_{j=1}^{N} \col_j^{\alpha_j+1}=\left(\prod_{j=1}^{N} \col_j^{\alpha_j}\right)J_W.
\end{equation}
\end{theorem}

We will call $\mirror$ ``Krawitz's mirror map'', or simply ``the mirror map.''
In this paper, we show that by appropriate rescaling\footnote{The rescaling consists of Formula \eqref{regular-rescale} and Formula \eqref{constant-c}.}, Krawitz's mirror map identifies the FJRW and Saito-Givental potentials of all genus, proving mirror symmetry. From now on, for any monomial $\bgen \in \Jac{W^T}$, we will use the following notation for the degree:
\begin{equation}\label{AB-degree}
\deg(\mirror(\bgen)) = \deg_W(\mirror(\bgen)) \qquad \text{and} \qquad \deg(\bgen) = \wt(\bgen).
\end{equation}

\begin{remark}\label{rmk:mirror_tensor}
When $W = \bigoplus_j W_j $, both the $A$ and the $B$-model Frobenius algebras decompose as tensor products of the Frobenius algebras of the $W_j$, and in this case the mirror map is a tensor product of mirror maps.
\end{remark}


\section{Main results}\label{section-main-result}

The main result of this paper is Theorem \ref{main-theorem}, which can be more precisely stated as

\begin{theorem}[Landau-Ginzburg Mirror Symmetry Theorem]\label{main-cor} Let $W$ be an invertible polynomial with no chain variables of weight $1/2$.
Then there exists a primitive form $\prim$ of $W^T$ such that the Krawitz isomorphism $\Jac{W^T}\cong \A{W}$ identifies the Saito-Givental potential $\mathscr{A}_{W^T,\prim}^{\rm SG}$ with the FJRW potential $\mathscr{A}_{W}^{\rm FJRW}$.
\end{theorem}

In fact, it suffices to prove this theorem at the level of Frobenius manifolds, i.e., at genus zero.
This is because in the cases we deal with, the work of Teleman \cite{T} and Milanov \cite{M} shows that the genus zero data completely determines the higher genus data of the LG models.
Thus, in the remainder of this article, we only need to prove the following theorem.

\begin{theorem}[Frobenius Manifold Mirror Symmetry Theorem]\label{FM_iso_thm} Let W be as in Theorem \ref{main-cor}.
There exists an isomorphism between a Frobenius manifold on $\Jac{W^T}$ and the Frobenius manifold on $\A{W}$. More explicitly,  there exists a primitive form $\prim$ of $W^T$ such that the Krawitz isomorphism $\Jac{W^T}\cong \A{W}$ induces
\begin{equation}\label{main-thm}
\pot_{0,W^T,\prim}^{\rm SG}=\pot_{0,W}^{\rm FJRW}.
\end{equation}
\end{theorem}

As explained in Section \ref{section-B-model}, a primitive form is associated to a good basis.  The good basis yielding mirror symmetry in Theorem \ref{FM_iso_thm} is the standard basis of Definition \ref{milnor_basis}.

Theorem \ref{FM_iso_thm} is proved by showing that $\pot_{0,W^T,\prim}^{\rm SG}$ and $\pot_{0,W}^{\rm FJRW}$ are completely determined by a handful of 4-point correlators. We then explicitly compute these correlators to show they differ only by a sign. We may exactly match the potentials by rescaling the primitive form and the B-model ring generators as in \cite[Section 6.5]{FJR},
\begin{equation}\label{regular-rescale}
x_i\to (-1)^{-\deg(x_i)}x_i, \quad \zeta\to (-1)^{-\chat{W^T}}\zeta.
\end{equation}
Thus, Theorem \ref{FM_iso_thm} is a consequence of the following theorem.
\begin{theorem}\label{LG_pot_thm}\label{reconstruction_theorem}
Let $W$ be an invertible polynomial as described in Theorem \ref{main-cor}.
\begin{enumerate}
\item Using the pairing, the ring structure, the properties of FJRW theory and Saito-Givental theory, and WDVV equations, the potentials $\pot_{0, W}^{\rm FJRW}$ and $\pot_{0, W^T, \prim}^{\rm SG}$ are completely determined by the correlators
\begin{itemize}
\item $\Fourptcorr{x_i}{ x_i}{ x_i^{\expon_i-2}}{ \Hessbase{W^T} }$ when $x_i$ is the variable in a Fermat $x_i^{\expon_i}$ with $\expon_i \neq 2$.
\item $\Fourptcorr{ x_N}{ x_N}{ x_{N-1}x_N^{\expon_N-2}}{ \Hessbase{W^T} }$ when $x_N$ is the last variable of a chain.
\item $\Fourptcorr{x_i}{ x_i}{ x_{i-1}x_i^{\expon_i-2}}{ \Hessbase{W^T} }$ when $x_i$ is a variable in a loop.
\end{itemize}
Here we use B-model notation, and $\phi_{W^T}$ is the element in $\Jac{W^T}$ of highest degree, normalized as in Definition \ref{milnor_basis}. The A-model correlators are obtained by mapping the insertions via Krawitz's mirror map in Theorem \ref{mirror-algebra}.
\item The values of these correlators are $q_i$ on the A-side.
\item The values of these correlators are $-q_i$ on the B-side.
\end{enumerate}
\end{theorem}

\begin{remark}
The correlators in Theorem \ref{reconstruction_theorem} may be described as $\Fourptcorr{ x_i}{ x_i}{ M_i/x_i^2}{ \Hessbase{W^T} }$ where $W^T = \sum_i M_i$ and $M_i=\prod_{j=1}^Nx_j^{a_{ij}}$ is any monomial of a Fermat or loop summand, or the final monomial of a chain summand. 
Here we define $M_i/x_i^2 := \prod_{j=1}^Nx_j^{a_{ij}-2\delta_{ij}}.$
Similar notation will be used throughout the paper. 
Such a formulation of correlators and their values was first discovered for simple elliptic singularities in \cite{MS} and then verified for exceptional unimodular singularities in \cite{LLSS}. 
\end{remark}

\begin{remark} When $W$ contains a chain summand $x_1^{a_1}x_2+x_2^{a_2}x_3+\cdots + x_N^{a_N}$ with $\expon_N=2$, we show that the B-model statements in Theorem \ref{LG_pot_thm} hold.
If we further know that the Frobenius algebra structures on $\A{W}$ and $\Jac{W^T}$ coincide, and that part (2) of Theorem \ref{LG_pot_thm} hold in A-model, then Conjecture \ref{LG_conjecture} will follow. Two such examples  for $Z_{13}, W_{13}$ of exceptional unimodular singularities are established in this way in \cite{LLSS}.
\end{remark}



\section{Good basis and B-model vanishing}\label{section-b-model}
In this section, we introduce some tools in the B-model for the proof of Part (1) of Theorem \ref{LG_pot_thm}.
More explicitly, we will use the symmetries of an invertible polynomial to prove Theorem \ref{thm-good-basis} and establish the \emph{Dimension Axiom} and \emph{Integer Degree Axiom} in the B-model (Lemma \ref{vanishing_lemma}). 

\subsection{Good basis of invertible polynomials}\label{section-proof-good-basis}
In this subsection we prove Theorem \ref{thm-good-basis}.
As a consequence, we obtain a Frobenius manifold structure on the base space of the universal unfolding of the corresponding singularity.
This structure can be computed perturbatively as described in Section \ref{section-perturbative-formula}, furnishing the genus zero data in the B-model. We will adopt the same notation as in Section \ref{section-B-model} and write $f$ instead of $W^T$ for the mirror polynomial.

We only need to prove Theorem \ref{thm-good-basis} for chains and loops, since Fermat polynomials are the A-type singularities discussed in Example \ref{eg-ADE}.
We will use the following notation:
\begin{enumerate}
\item
If $g(x)$ is a polynomial,
$[g]_f$ will denote the class in $\mathcal H_f^{(0)}$ represented by $g(x) d^Nx$.
\item The linear coordinates on $\CC^N$ are $x_1, \cdots, x_N$ and $x_{N+k}\equiv x_k$.
\end{enumerate}
In the notation of Section \ref{section-Krawitz-mirror}, the inverse of the exponent matrix of $f=W^T$ is
\[
\AW_f^{-1}=(\AW_W^{-1})^T=\left(
\begin{array}{ccccc}
\rho_1^{(1)}&\rho_1^{(2)}&\cdots&\rho_1^{(N)}\\
\vdots &\vdots&\ddots&\vdots\\
\rho_N^{(1)}&\rho_N^{(2)}&\cdots&\rho_N^{(N)}\\
\end{array}
\right).
\]
Let $\row_j$ be the linear transformation
$$
  \row_j\cdot x_i= \exp(2\pi \sqrt{-1}\rho_i^{(j)})\,x_i.
$$
This transformation preserves $f$; that is, $\row_j\cdot f=f,$ for all $j$. Hence $\row_j$ induces an action on the Brieskorn lattice
\[
    \row_j: \mathcal H_f^{(0)}\to \mathcal H_{f}^{(0)}.
\]
Moreover, it is easy to see that the higher residue pairing $K_f$ is $\row_j$-invariant.
These symmetries are enough to prove that the standard basis is a good basis.

Let $x_1^{r_1}\cdots x_N^{r_N}$ and $x_1^{r_1^\prime}\cdots x_N^{r_N^\prime}$ be monomials in the standard basis for either the chain or loop type.  Let
$$(m_1, \cdots, m_N)=(r_1+r_1^\prime, \cdots, r_N+r_N^\prime).$$
The $\row_j$-invariance of $K_f$ implies the \emph{integral conditions}
\[
   \sum_{i=1}^N (m_i+2)\rho_i^{(j)}=k_i\in \ZZ, \quad \text{for all}\; j
\]
(the extra $2$ comes from two copies of $d^Nx$).
This is equivalent to
\begin{equation}\label{good-basis-cond}
(k_1,k_2,\cdots,k_N)\, \AW_f
=
(m_1+2,m_2+2,\cdots,m_N+2).
\end{equation}
The remainder of the proof splits into two cases, corresponding to the possible types of $f$.\\

\noindent{\bf The chain case.} Let  $f=x_1^{a_1}+x_1x_2^{a_2}+\cdots + x_{N-1}x_N^{a_N}$.
The exponent matrix has the form
\[
\AW_f = \left( \begin{array}{ccccc}
\expon_1& & &&\\
1 & \expon_2 & &&\\
&  \ddots & \ddots &&\\
&&&\expon_{N-1} & \\
&&&1 & \expon_N \end{array} \right).
\]
Equation \eqref{good-basis-cond} in this case becomes
$$
   m_1=k_1 a_1+k_2-2, \; \; \;m_2=k_2a_2+k_3-2, \;\;\; \ldots, \; \;\; m_{N-1}=k_{N-1}a_{N-1}+k_N-2, \;\;\; m_N=k_{N}a_N-2
$$
where $0\leq m_i\leq 2a_i-2$ and the $k_i$ are integers.
We investigate possible values for the $k_i$ and $m_i$ .
This analysis is easiest if we begin by tracing all possible values of $k_i$ back from $k_N$.
The only possibilities are
\begin{enumerate}
\item $(k_1,\cdots,k_N)=(1,1,\cdots, 1)$.
\item $(k_1,\cdots,k_N)=(1, \cdots, 1, 0,2, \cdots, 0,2)$.
\item $(k_1,\cdots, k_N)=(1, \cdots, 1, 2, 0, 2, \cdots, 0, 2)$.
\end{enumerate}
In case (3), we have
$$(m_1, \cdots, m_N)=(a_1-1, \cdots,a_{N-2l-2}-1,a_{N-2l-1}, 2a_{N-2l}-2,\cdots, 0 ,2a_{N-2}-2,0,2a_N-2).$$
This can not appear if both $x_1^{r_1}\cdots x_N^{r_N}$
and $x_1^{r_1^\prime}\cdots x_N^{ r_N^\prime}$ are in the standard basis. For cases (1) and (2), we check directly that
$$\deg(x_1^{m_1}\cdots x_N^{m_N})=\deg(x_1^{r_1}\cdots x_N^{r_N})+\deg(x_1^{r_1^\prime}\cdots x_N^{ r_N^\prime})=\hat c_f.$$
Since $K_f$ preserves the $\QQ$-grading, we have
$$
\deg K_f([x_1^{r_1}\cdots x_N^{r_N}]_f, [x_1^{r_1^\prime}\cdots x_N^{ r_N^\prime}]_f)=\deg(x_1^{r_1}\cdots x_N^{r_N})+\deg(x_1^{r_1^\prime}\cdots x_N^{ r_N^\prime})+2\sum_i q_i=N.
$$
It follows that $
  K_f([x_1^{r_1}\cdots x_N^{r_N}]_f, [x_1^{r_1^\prime}\cdots x_N^{ r_N^\prime}]_f)$ lies in $z^N \CC.
$
\\

\noindent{\bf The loop case.} Let  $f=x_1^{a_1}x_N+x_1x_2^{a_2}+\cdots+ x_{N-1}x_N^{a_N}$.
The exponent matrix of $f$ is
\[
\AW_f = \left( \begin{array}{ccccc}
\expon_1& & &&1\\
1 & \expon_2 & &&\\
&  \ddots & \ddots &&\\
&&&\expon_{N-1} & \\
&&&1 & \expon_N \end{array} \right).
\]
With the convention $k_1 \equiv k_{N+1}$, equation \eqref{good-basis-cond} above implies
\begin{equation}\label{loop-good-basis}
m_i+2=k_i a_i+k_{i+1}, \quad i=1,\cdots, N.
\end{equation}
Let
$h_i=k_i-1$ for each $i$. Equation \eqref{loop-good-basis} becomes
$$m_i+2=(h_i+1) a_i+h_{i+1}+1.$$
Since $0\leq m_i\leq 2a_i-2$, we get
\begin{equation}\label{loop-modify}
1-a_i\leq h_i a_i+h_{i+1}\leq a_i-1, \quad i=1,\cdots, N.
\end{equation}

If there is some $h_{i+1}=0$, then the above equation implies $h_i=0$, and recursively,
$$(h_1,\cdots,h_N)=(0,0,\cdots, 0).$$

Otherwise, we can assume none of the $h_i$ is zero. There are two situations. Either there is one $h_i$ with $|h_i|=1$ or all $|h_i|\geq2$. For the first case, we assume some $h_{i+1}=\pm1$. Since $h_i\neq0$ by assumption, the inequality \eqref{loop-modify} implies $h_i=\mp1$.  We can repeat this process and get the following solution when $N$ is an even number:
$$(h_1,\cdots,h_N)=(\pm1,\mp1,\cdots,\pm1,\mp1).$$

Finally we prove it is impossible to have all $|h_i|\geq2$. Equation \eqref{loop-modify} implies
\begin{equation}\label{loop-second-modify}
-1+{1-h_{i+1}\over a_{i}}\leq h_{i} \leq 1-{1+h_{i+1}\over a_{i}}.
\end{equation}
If all $|h_{i+1}|\geq 2$, this implies
\begin{equation}\label{loop-cont}
|h_{i}|<|h_{i+1}|.
\end{equation}
In fact, if $h_{i+1}\geq2$, then the RHS of inequality \eqref{loop-second-modify} implies $h_{i}<1$. By assumption, we know $h_{i}\leq-2$. However, since
$$-h_{i+1}<-1+{1-h_{i+1}\over a_{i}},$$
inequality \eqref{loop-cont} follows from the LHS of \eqref{loop-second-modify}. A similar argument works for $h_{i+1}\leq-2$. We repeat this process and we find
$$|h_{i}|=|h_{i+N}|<\cdots<|h_{i+1}|<|h_i|,$$
which is impossible.
Thus the only possibilities for the $k_i$'s are
\begin{enumerate}
\item $(k_1,\cdots,k_N)=(1,1,\cdots, 1)$, and
\item $(k_1,\cdots,k_N)=(1\pm1,1\mp1,\cdots, 1\pm1,1\mp1)$, if $N$ is even.
\end{enumerate}
In each case, again we have
$$\deg(x_1^{m_1}\cdots x_N^{m_N})=\deg(x_1^{r_1}\cdots x_N^{r_N})+\deg(x_1^{r_1^\prime}\cdots x_N^{ r_N^\prime})=\hat c_f.$$
By the same degree reason as in the chain case, we know
$
  K_f([x_1^{r_1}\cdots x_N^{r_N}]_f, [x_1^{r_1^\prime}\cdots x_N^{ r_N^\prime}]_f)$ lies in $z^N \CC.
$




\subsection{Vanishing conditions in B-model}
We will now prove the B-model properties that are the analogs of the Dimension Axiom \eqref{dim_ax} and Integer Degrees Axiom \eqref{lbd_ax} on the A-side. These give us vanishing conditions for B-model correlators which we will later use to reconstruct the potential $\pot_{0,W^T,\prim}^{\rm SG}$.

\begin{lemma}\label{vanishing_lemma} Let $\mirror$ be Krawitz's mirror map (Theorem \ref{mirror-algebra}).
The A-model correlator
\begin{equation}\label{b-correlator}
\mirror(X) = \left\langle \; \mirror\left( \prod_{i=1}^N x_i^{e_{1,i}}\right), \ldots, \mirror\left( \prod_{i=1}^N x_i^{e_{k,i}}\right) \; \right\rangle
\end{equation}
satisfies the Dimension Axiom \eqref{dim_ax} if and only if
\begin{equation}\label{b_dim_ax}
\sum_{\nu=1}^k \deg\left(\prod_{i=1}^{N} x_i^{e_{\nu,i}}\right)= \chat{W^T}+k-3,
\end{equation}
and $\mirror(X)$ satisfies the Integer Degrees Axiom \eqref{lbd_ax} if and only if
\begin{equation}\label{b_lbd_eq}
 -2q_j - \sum_{\nu=1}^k \sum_{i=1}^N\rho_i^{(j)}e_{\nu,i}  \in \ZZ \quad \text{for}\;\;j=1, \ldots, N.
 \end{equation}
Moreover, if the B-model correlator
$$X:=\left\langle \; \prod_{i=1}^N x_i^{e_{1,i}}, \ldots, \prod_{i=1}^N x_i^{e_{k,i}}\; \right\rangle$$ is nonzero, then both \eqref{b_dim_ax} and \eqref{b_lbd_eq} hold.
\end{lemma}
\begin{proof}
The equivalence of \eqref{dim_ax} and \eqref{b_dim_ax} follows from the fact that $\mirror$ is degree-preserving (Theorem \ref{mirror-algebra}) and $\hat c_W=\hat c_{W^T}$.
Also from Theorem \ref{mirror-algebra} we know
\[
\mirror\left( \prod_{i=1}^{N} x_i^{e_{\nu,i}}\right)\in \A{\gamma} \quad\text{and}\quad \gamma=\left(\prod_{i=1}^{N}\col_i^{e_{\nu,i}}\right) J_W.
\]
By directly calculating the quantity $\ld_j$ in \eqref{lbd_ax} using \eqref{weight} and \eqref{exponent-op}, we get
\[
\ld_j\equiv q_j (k-2) - \sum_{\nu=1}^k \left(\sum_{i=1}^N\rho_i^{(j)}e_{\nu,i}+q_j\right) \mod \ZZ, \quad \text{for}\;\;j=1, \ldots N.
\]
This is exactly Equation \eqref{b_lbd_eq}.

Now assume $X \neq 0$. Then \eqref{b_dim_ax} holds because the potential $\pot_{0,W^T,\prim}^{\rm SG}$ has eigenvalue $\hat{c}_{W^T}-3$ with respect to the Euler vector field $\sum_{\alpha=1}^\mu(1-\deg(\phi_\alpha))s_\alpha\frac{\partial}{\partial s_\alpha}$. This well-known fact also follows explicitly from the perturbative formula \eqref{Jfunc_eq} which respects the $\QQ$-grading.

Finally, we prove that if $X \neq 0$ then \eqref{b_lbd_eq} holds. To do so we introduce a $G_{W^T}$-action on the B-model.
Since $G_{W^T}$ is generated by $\{\row_j\}_{j=1}^{N}$, it suffices to define each $\row_j$-action as follows.
\[
\row_j\cdot x_i=\exp(2\pi\sqrt{-1}\rho_i^{(j)})\,x_i, \quad \row_j\cdot z=z, \quad \row_j\cdot s_{\alpha} =c_\alpha^{-1} s_{\alpha}.
\]
where $c_\alpha$ is the nonzero constant such that
$$\row_j \cdot \phi_\alpha=c_\alpha \phi_\alpha.$$
We can check that the action of $\row_j$ is compatible with the relations
 \begin{equation}\label{relation_eq}
 {\partial f\over\partial x_i}gd^Nx = -z {\partial g\over\partial x_i}d^Nx \quad \text{in} \quad  \mathcal H_f^{(0)}
 \end{equation}
 for each monomial $g$ in $\CC[x_1, \ldots x_N]$.
Thus the perturbative formula \eqref{Jfunc_eq} shows that $f,$ $F$, $\prim$, and $\Jfunc{}$ are all invariant under the $G_{W^T}$-action.

Furthermore, according to \eqref{J-func} and \eqref{flat-coord}, $\row_j$ acts on $t_\alpha$ by a factor of $c_\alpha^{-1}$.
Each $\row_j$ acts on the $\nu$-th insertion of $X$ by
\[
\row_j \cdot  \prod_{i=1}^{N} x_i^{e_{\nu,i}}= \exp\left(2 \pi \sqrt{-1}\sum_{i=1}^N \rho_i^{(j)}e_{\nu,i} \right)\prod_{i=1}^N x_i^{e_{\nu,i}}.
\]
Therefore $\row_j$ acts on the corresponding monomial in the prepotential \eqref{sg-potential-g=0} by a factor of
\[
\exp\left(-2 \pi \sqrt{-1}\sum_{\nu=1}^{k}\sum_{i=1}^N \rho_i^{(j)}e_{\nu,i} \right).
\]
On the other hand, the (higher) residue pairing is invariant under the $\row_j$-action. Since $\row_j\cdot d^Nx=\exp(2 \pi \sqrt{-1}q_j)d^Nx$ by \eqref{weight}, it follows that the pairing
$$
K_f(\phi_\alpha d^Nx, \phi_\beta d^Nx)
$$
is zero unless $c_\alpha c_\beta=\exp(-4 \pi \sqrt{-1}q_j)$. Then \eqref{compute_potential_eq} implies
$$
   \row_j \cdot \pot_{0,f,\prim}^{\rm SG}(\mathbf{t})= \exp(4 \pi \sqrt{-1}q_j)\,\pot_{0,f,\prim}^{\rm SG}(\mathbf{t}).
$$
Matching the above two factors, we find
$$
\exp\left(-2 \pi \sqrt{-1}\sum_{\nu=1}^{k}\sum_{i=1}^N \rho_i^{(j)}e_{\nu,i} \right)=  \exp(4 \pi \sqrt{-1}q_j).
$$
This is exactly \eqref{b_lbd_eq}.
\end{proof}


\section{Reconstruction}\label{reconstruction-1}


In this section, we will introduce the key lemma that turns the WDVV equations into a powerful tool for reconstructing genus zero potentials.
Finally, we will completely reconstruct an arbitrary sum of Fermat polynomials as an example of our proof strategy in the general case.


\subsection{A reconstruction lemma from WDVV equations}

We introduce a powerful reconstruction lemma that follows from the WDVV equations. The statement of this lemma requires the following definition.

\begin{definition} We say that an element $\xi$ or $\A{W}$ is \emph{primitive} if whenever $\xi = \xi_1 \star \xi_2$, either $\deg(\xi_1)=0$ or $\deg(\xi_2)=0$.
\end{definition}

It is easy to see that for ${\rm Jac}(W^T)$, the set of primitive elements is a subset of $\{x_1, \ldots, x_N\}$.
By mirror symmetry, the set of primitive elements in $\A{W}$ is a subset of $\{\mirror(x_1), \ldots, \mirror(x_N)\}$.  The next lemma says that the prepotential $\mathcal{F}_0$ in each theory is completely determined by correlators with mostly primitive insertions.
\begin{lemma}[\cite{FJR}, Lemma 6.2.6]\label{reconstruction_lemma}
A $k$-point correlator $\langle \xi_1, \ldots, \xi_{k-3}, \alpha, \beta, \epsilon \star \phi \rangle$ satisfies
\begin{align}
\langle \xi_1, \ldots, \xi_{k-3}, \gamma, \delta, \epsilon \star \phi \rangle = &\langle \xi_1, \ldots, \xi_{k-3}, \gamma, \epsilon, \delta \star \phi \rangle \nonumber
+ \langle \xi_1, \ldots, \xi_{k-3}, \gamma \star \epsilon, \delta, \phi \rangle \nonumber \\
&-\langle \xi_1, \ldots, \xi_{k-3}, \gamma \star \delta, \epsilon, \phi \rangle +S \label{reconst_eq}
\end{align}
where $S$ is a linear combination of correlators with fewer than $k$ insertions. 
If $k=4$, then there are no such terms in the equation, i.e., $S=0$.
In addition, the $k$-point correlators are uniquely determined by the pairing, the three-point correlators, and by correlators of the form $\langle \xi_1, \ldots \xi_n \rangle$ with $n<k$ where $\xi_i$ is primitive for $i\leq n-2$. 
\end{lemma}

Since the proof of Lemma \ref{reconstruction_lemma} uses only the WDVV equations, it holds for both $\pot^{\rm SG}_{0, W^T, \prim}$ and $\pot^{\rm FJRW}_{0, W}$.
This lemma implies that to compare $\pot_{0, W^T, \prim}^{\rm SG}$ and $\pot_{0, W}^{\rm FJRW}$, it suffices to compare correlators of the form
\begin{equation}\label{corr_form_eq}
X = \langle x_N, \ldots, x_N, x_{N-1}, \ldots, x_{N-1}, \ldots, x_1, \ldots, x_1, \alpha, \beta \rangle, \quad \alpha, \beta\in \Jac{W^T}.
\end{equation}
Here we are using B-side notation; the corresponding A-model correlator is
\[
\langle \mirror(x_N), \ldots, \mirror(x_N), \mirror(x_{N-1}), \ldots, \mirror(x_{N-1}), \ldots, \mirror(x_1), \ldots, \mirror(x_1), \mirror(\alpha), \mirror(\beta) \rangle.
\]
The indicies of the primitive inserstions (including $\alpha$ and $\beta$ if they are primitive) in the correlator $X$ in \eqref{corr_form_eq} are arranged in decreasing order.

In the remainder of this paper, we will apply the vanishing conditions of Lemma \ref{vanishing_lemma} to correlators of the form \eqref{corr_form_eq}.
In this context, let
\[\alpha = \prod_{i=1}^{N} x_i^{\m_i^X}, \quad \beta = \prod_{i=1}^{N} x_i^{\n_i^X},\]
and let $\num^X_i$ be the number of insertions in $X$ equal to $x_i$, ignoring $\alpha$ and $\beta$. Thus
\[\sum_{i=1}^N \num^X_i = k-2.\]

Now let $\bb^X_i$ be the real numbers defined by the equation
\begin{equation}\label{b_def}
\left(\begin{array}{c}
\bb^X_1\\
\vdots\\
\bb^X_N \end{array}\right):= \AW_W^{-1} \left(\begin{array}{c}
\num^X_1 + \m^X_1 + \n^X_1 + 2\\
\vdots\\
\num^X_N + \m^X_N + \n^X_N + 2 \end{array}\right),
\end{equation}
and let
\begin{equation}\label{ind-k}
\K^X_i= \num^X_i - \bb^X_i + 1.
\end{equation}

When there is no possibility of confusion, we will drop the superscript $X$ from the notation.

When we apply Lemma \ref{vanishing_lemma} to $X$, we produce the following lemma.

\begin{lemma}\label{properties_lemma}
Any nonvanishing A- or B-model correlator of the form in \eqref{corr_form_eq} can be written so it satisfies the following properties:
\begin{enumerate}
\item [(P1).] All the numbers $\K_i$ are integers, i,e.,
\begin{equation}\label{fact2}
\K_i\in\mathbb{Z}.
\end{equation}
\item [(P2).] The following equation holds:
\begin{equation}\label{fact1}
\sum_{i=1}^N \K_i=1.
\end{equation}
\item [(P3).]\label{fact3}
The maximum values for $\m_i$ and $\n_i$ are as follows:
\begin{itemize}
\item $\expon_i-2$ if $W$ is a Fermat polynomial,
\item $\expon_i-1$ for a chain summand $x_1^{a_1}x_N+x_1x_2^{a_2}+\dots+x_{N-1}x_N^{a_N}$ in $W^T$, subject to the additional condition that both $(\m_{1}, \m_{2}, \ldots, \m_{N})$ and $(\n_1, \n_{2}, \ldots, \n_{N})$ are not of the form $(a_{1}-1, 0, a_{3}-1, \ldots, 0, a_{N-2}-1, 0, a_{N}-1)$ with $N$ odd or $(\ldots, k, a_{N-2l}-1, 0, \ldots, 0, a_{N-2}-1, 0, a_{N}-1)$ with $k\geq1$,
\item $\expon_i-1$ for a loop summand $x_1^{a_1}+x_1x_2^{a_2}+\dots+x_{N-1}x_N^{a_N}$ in $W^T$.
\end{itemize}
\end{enumerate}
\end{lemma}
\begin{proof}
Let $X$ be a nonvanishing correlator of the form in \eqref{corr_form_eq}. After we write the insertions of $X$ in the standard basis, this correlator satisfies (P3).

We will prove (P1) and (P2) for the B-model only. The same proof works for the A-model because the A-model Axioms \eqref{dim_ax} and \eqref{lbd_ax} correspond to the B-model vanishing conditions by Lemma \ref{vanishing_lemma}.

Since $X$ is as in \eqref{corr_form_eq}, in the context of Lemma \ref{vanishing_lemma} we have
\begin{equation}\label{exponent-basic}
e_{\nu,i}=\left\{
\begin{array}{lll}
\nu_i & \nu=k-1\\
n_i & \nu=k\\
1 & \ell_1+\cdots+\ell_{i-1}+1\leq \nu\leq \ell_1+\cdots+\ell_i\\
0 & \text{otherwise}.
\end{array}
\right.
\end{equation}


First we will show that (P1) is equivalent to the Integer Degrees Axiom. On the B-side, this axiom says that $X$ is zero unless
\[
-2q_j - \sum_{\nu=1}^k \sum_{i=1}^N \col_i^{(j)}e_{\nu,i} \in \ZZ, \quad \text{for}\; j= 1, \ldots, N.
\]
Then using \eqref{weight} and \eqref{exponent-basic}, we have
\begin{align*}
-2q_j - \sum_{\nu=1}^k \sum_{i-1}^N \col_i^{(j)}e_{\nu,i} &= -2 \sum_{i=1}^N\col_i^{(j)} - \sum_{i=1}^N\col_i^{(j)} \num_i - \sum_{i=1}^N\col_i^{(j)}\m_i - \sum_{i=1}^N\col_i^{(j)} \n_i= -\bb_j.
\end{align*}
The last equality follows from \eqref{b_def}. So $X$ satisfies the Integer Degrees Axiom if and only if $\bb_j \in \ZZ$ for all $j$, which is true if and only if $\K_j \in \ZZ$ for all $j$.

Next we derive (P2) from the Dimension Axiom.
Let $q_i^T$ be the $i^{th}$ weight of $W^T$. Then by \eqref{b_dim_ax}, the correlator $X$ vanishes unless
\begin{equation}\label{temporary_eq}
\sum_{\nu=1}^{k}\sum_{i=1}^{N}q_i^Te_{\nu,i} = \sum_{i=1}^N(1-2q_i^T) + \sum_{i=1}^N \num_i - 1.
\end{equation}
According to Equation \eqref{exponent-basic}, the left hand side of \eqref{temporary_eq} is
$$\sum_{i=1}^N \num_i q_i^T + \sum_{i=1}^N \m_i q_i^T + \sum_{i=1}^N \n_i q_i^T.$$
This implies
$$1=\sum_{i=1}^N(\num_i - (\num_i+\m_i+\n_i+2)q_i^T + 1)=\sum_{i=1}^N \K_i.$$
Here the last equality uses \eqref{weight}, \eqref{b_def}, and \eqref{ind-k}.
\end{proof}

Lemmas \ref{reconstruction_lemma} and \ref{properties_lemma} tell us when correlators are in a particularly nice form. We make this precise with the following definition.

\begin{definition}\label{standard_form_def}
A genus-0 correlator is \emph{of type} $\X{X}{-1}$ if
\begin{enumerate}
\item it has at least four insertions,
\item it is in the form of \eqref{corr_form_eq}, and
\item it satisfies properties (P1), (P2), and (P3) in Lemma \ref{properties_lemma}.
\end{enumerate}
\end{definition}

%


Because Krawitz's mirror map matches the pairing and the 3-point correlators, to compare $\pot_{0, W}^{\rm FJRW}$ and $\pot_{0, W^T, \prim}^{\rm SG}$ it suffices to compare correlators of type $\X{X}{-1}$.

%
%
%

\subsection{A warm up example: the Fermat polynomial}

In this section we prove Part (1) of Theorem \ref{reconstruction_theorem} in the special case where $W = x_1^{\expon_1} + x_2^{\expon_2} + \ldots + x_N^{\expon_N}$ is a sum of Fermat polynomials, as a way to illustrate our general proof strategy. According to Remark \ref{remark-stablization}, we can assume $\expon_{i}>2$ for all $i$.

First, we reduce the reconstruction problem to the summands of $W$. We only need to consider correlators of type $\X{X}{-1}$.
\begin{lemma}\label{fermat_splitting_lemma}
Let $X$ be a correlator of type $\X{X}{-1}$. Then there is a unique $j \in \{1, \ldots, N\}$ such that $\K_{j}=1, \num_j\geq 2$ and $\K_{i}=\num_i=0$ for $i \neq j$. Furthermore,
$$X = \Fourptcorr{ x_j}{ x_j}{ x_j^{\expon_j-2} \alpha}{ x_j^{\expon_j - 2} \beta } \quad \text{for some } \alpha, \beta \in \Jac{W-W_j}.$$
\end{lemma}
\begin{proof}
By Definition \eqref{b_def}, we have $$\expon_i \bb_i = \num_i+\m_i+\n_i+2 \quad \text{for each} \; i.$$
Then \eqref{ind-k} implies
\begin{equation}\label{eq3}
\K_i  = \num_i - \frac{\num_i + \m_i + \n_i + 2}{\expon_i} + 1.
\end{equation}
Property (P3) in Lemma \ref{properties_lemma} implies that $\m_i+\n_i \leq 2\expon_i-4$, so we have
\begin{equation}\label{eq7}
\K_i \geq \num_i\left(1 - \frac{1}{\expon_i}\right) + \frac{2}{\expon_i} - 1 > -1.
\end{equation}
Since $\K_i \in \ZZ$, we have $\K_i \geq 0$. Then Property \eqref{fact1} in Lemma \ref{properties_lemma} implies that there is a unique $j \in \{1, \ldots, N\}$ such that $\K_{j}=1$ and $\K_{i}=0$ for $i \neq j$.

Moreover, by Equation $(\ref{eq7})$ we know $\num_i = 0$ for $i \neq j$. Then since
$$\sum_{i=1}^{N} \num_i \geq 2,$$
we know $\num_j \geq 2$.
Furthermore, using (\ref{eq7}), we have
$$\K_j \geq \num_j\left(1 - {1\over\expon_j}\right) + {2\over\expon_j} - 1.$$
Plugging in $\K_j=1$ it is easy to show that $\num_j \leq 2$, so $\num_j = 2$. Then (\ref{eq3}) shows that $\m_j+\n_j = 2 \expon_j-4$ and the result follows.

\end{proof}

Now we complete the proof of Part (1) of Theorem \ref{LG_pot_thm} in the Fermat case.
\begin{proposition}\label{lem:fermat_alpha}
Let $W = x_1^{\expon_1} + x_2^{\expon_2} + \ldots + x_N^{\expon_N}$ be a sum of Fermat polynomials with all $\expon_i >2$.
The potentials $\pot_{0, W}^{\rm FJRW}$ and $\pot_{0, W^T, \prim}^{\rm SG}$ are completely determined by the Frobenius algebra structure and the correlators
\[
\Fourptcorr{ x_j}{ x_j}{ x_j^{\expon_j-2}}{ \prod_{i=1}^{N} x_i^{a_i-2} }, \quad j=1,\ldots, N.
\]
\end{proposition}
\begin{proof}
By Lemma \ref{fermat_splitting_lemma}, we only need to reconstruct $ \fourptcorr{ x_j}{ x_j}{ x_j^{\expon_j-2} \alpha}{ x^{\expon_j - 2} \beta }$ from the correlator above. Apply the Reconstruction Lemma \ref{reconstruction_lemma} with $\gamma = x_j$, $\epsilon = x_j^{\expon_j-2}$, $\phi = \alpha$, and $\delta = x_j^{\expon_j-2} \beta$.
Then $\gamma \epsilon$ and $\gamma\delta$ both vanish because they have a factor of $x_j^{\expon_j-1}=0$.
The final correlator (with $\delta \phi$) is $ \fourptcorr{ x_j}{ x_j}{ x_j^{\expon_j-2}}{ x_j^{\expon_j - 2} \alpha \beta }$. However, by the Dimension Axiom in Lemma \ref{vanishing_lemma}, this correlator is nonzero only if $\deg(\alpha\beta)=\hat c_W-\hat c_{W_j}$. Since $\alpha\beta\in \Jac{W-W_j}$, up to a constant, we must have
$$\alpha \beta=\prod_{i\neq j} x_i^{a_i-2}.$$
\end{proof}


\subsection{Splitting Principle}

Recall from Definition \ref{standard_form_def} that a correlator $X$ is of type $\X{X}{-1}$ if it has at least four insertions, satisfies properties (P1), (P2), and (P3) in Lemma \ref{properties_lemma}, and has the form
$$X = \langle x_N, \ldots, x_N, x_{N-1}, \ldots, x_{N-1}, \ldots, x_1, \ldots, x_1, \alpha, \beta \rangle.$$
To prove Theorem \ref{reconstruction_theorem}, it suffices to show that any correlator of type $\X{X}{-1}$ can be reconstructed from the correlators in Theorem \ref{reconstruction_theorem}.
In this section we reconstruct correlators of type $\X{X}{-1}$ from correlators of ``type $\X{X}{0}$'' (see Definition \ref{def-X0}), which are associated to a particular atomic summand of $W$. 
This is called \emph{Splitting Principle}.

We will use the following notation. Suppose that $W = \bigoplus W_j$ is a disjoint sum, where each summand $W_j$ is of atomic type as described in Theorem \ref{atomic_thm}.
If $x_i$ is a variable appearing in $W_j$, we say that $x_i \in W_j$, or simply $i \in W_j$. Likewise, if $\alpha$ is a monomial in variables appearing in $W_j$, we say $\alpha \in W_j$. We define $\K_W: =\sum_{i\in W}\K_i$. For any ordered subset of indices $S \subseteq \{1,2,\cdots, N\}$, we define
\[\K_{S} : = \sum_{i \in S} \K_i, \qquad \mathbf{\K_S}:=(\K_i)_{i \in S},
\]
that is, $\mathbf{\K_S}$ is a vector of the $K_i$ such that $i$ is in $S$ and $\K_S$ is the sum of the components of this vector. We define $\mathbf{\num_S}$, $\mathbf{\m_S}$, and $\mathbf{\n_S}$ similarly.

The goal of this section is to reduce the proof of Theorem \ref{reconstruction_theorem} Part (1) to a reconstruction for each atomic type.
More specifically, in this section we prove Proposition \ref{two_from_a_poly}, which says that any correlator of type $\X{X}{-1}$ can be reconstructed from correlators satisfying $\sum_{i \in W_j} \num_i \geq 2$, for some $j$. That is, we reconstruct from correlators with at least two primitive insertions coming from some summand $W_j$ of $W$.
We say these correlators have type $\X{X}{0}$.


Throughout the remainder of this paper, we take $i+N \equiv i$ whenever $i$ is in a length-$N$ loop summand of $W$.

\subsubsection{Preliminaries on loop indices}

Let $X$ be a correlator of type $\X{X}{-1}$. We will prove the main result of this section, Proposition \ref{two_from_a_poly}, by analyzing possible values for $\K_W$.

\begin{definition}
We say that $i \in W$ is a \emph{loop index} if
\begin{equation}\label{loop-index-def}
\expon_i \bb_i + \bb_{i+1} = \num_i + \n_i + \m_i + 2.
\end{equation}
For each summand $W_j$, we say a set of loop indices $S \subset W_j$ \emph{obeys the Negative-Positive rule} (the NP-rule) if it has the property that for any index $i\in S$, if $\K_i < 0$, then the index $i+1 \in S$.
\end{definition}
Note that if $W$ is a loop, every $i \in W$ is a loop index, and if $W$ is an $N$-variable chain, every $i < N$ is a loop index.
The following lemma summarizes some useful inequalities for loop indices.

\begin{lemma}\label{K-index-ineq}
If $i \in W$ is a loop index, then the following inequalities hold:
\begin{eqnarray}
\m_{i}+\n_{i} &=& \expon_{i}(\num_{i} - \K_{i} + 1) + (\num_{i+1} - \K_{i+1} + 1) - \num_{i}-2. \label{mplusn_eq}\\
\expon_i\K_i + \K_{i+1} &\geq& (\expon_i-1)(\num_i-1) + \num_{i+1}. \label{eq2}\\
(\expon_i-1) \num_i + \num_{i+1} &\leq& \expon_i (\K_i + 1) + \K_{i+1} - 1. \label{ell_eq}\\
\K_i + \K_{i+1} &\geq& (1-\expon_i)(1+\K_i). \label{key_equation}
\end{eqnarray}
\end{lemma}
\begin{proof}
We obtain \eqref{mplusn_eq} by substituting $\bb_i = \num_i - \K_i + 1$ into \eqref{loop-index-def}.
Then \eqref{eq2} follows from using Property (P3) in Lemma \ref{properties_lemma} which says that $\n_i + \m_i \leq 2 \expon_i - 2$.
Rearranging slightly, we get \eqref{ell_eq}.
Then the last inequality follows by using $\num_i \geq 0$ and adding $(1-\expon_i)\K_i$ to both sides of \eqref{eq2}.
\end{proof}
From inequality \eqref{key_equation}, we get the following corollary.
\begin{corollary}\label{key-corollary}
If $i\in W$ is a loop index, then
$$\K_i<0\implies \K_i+\K_{i+1}\geq0$$
Furthermore, the equality holds when $(\K_i, \K_{i+1})=(-1, 1)$.
\end{corollary}
The lemma and corollary above will be used repeatedly in our reconstruction for the loop and chain polynomials. In addition, they determine $\mathbf{\K_S}$ when $S$ is a set of loop indices that obeys the NP-rule.
\begin{lemma}\label{loop_condition_lemma}
Let $S\subset W$ be a set of loop indices that obeys the NP-rule. Then
\begin{equation}\label{sum-k-indices}
\K_S\geq 0.
\end{equation}
Furthermore, we have the following cases:
\begin{itemize}
\item If $\K_S= 0$, then $\mathbf{\K_S}$ is a concatenation of (0)s and (-1,1)s.
\item If $\K_S= 1$, then $\mathbf{\K_S}$ is a concatenation of (0)s and (-1,1)s with one copy of (1), (-1,2), or (-2,3). If $(\K_i, \K_{i+1}) = (-2,3)$, then $\expon_i=2$.
\end{itemize}
\end{lemma}
\begin{proof}
If there exists some index $i\in S$ such that $\K_{i}<0$, then $i+1\in S$ by assumption and Corollary \ref{key-corollary} implies $\K_{i} + \K_{i+1} \geq 0$. Furthermore, $\K_{i-1}\geq0$ if $i-1\in S$ and 
\[\sum_{\substack{j\in S, \\ j\neq i,\, i+1}} \K_j\leq \K_S.\]
If \eqref{sum-k-indices} fails, we can repeat the process above for all negative $\K_i$ and eventually get a contradiction. Thus $\mathbf{K_S} \geq 0$.

Let $A = \{i\in S \mid \K_i \geq 0, \; \K_{i-1} \geq 0\}$. Then Corollary \ref{key-corollary} implies
\begin{equation}\label{rest-np}
\sum_{i \in A} \K_i \leq \K_S.
\end{equation}

If $\K_S= 0$, then we get $\K_i = 0$ for each $i\in A$.
Another application of Corollary \ref{key-corollary} shows that the rest of the $K_i$'s are pairs of $(-1,1)$.

If $\K_S= 1$, then \eqref{rest-np} implies there is at most one $j\in A$ such that $\K_j=1$. If there is one such $j\in A$, then as the same discussion as above shows $\K_i=0$ for $i\in A, i\neq j$ and the rest of the $K_i$'s are pairs of $(-1,1)$. If there is no such $j\in A$, then $\K_i=0$ for all $i\in A$. For the rest of the $K_i$'s, besides pairs of $(-1,1)$, there will be exactly one pair $(\K_i,\K_{i+1})$ such that $\K_i<0, \K_i+\K_{i+1}=1$. Then the statement follows from \eqref{key_equation} and $a_i\geq2$.
\end{proof}

Once we know $\mathbf{K_W}$, we can often solve for $\mathbf{\num_W}$ and $\mathbf{\m_W}+\mathbf{\n_W}$, as in the following two lemmas.

\begin{lemma}\label{klmn_lemma}
Let $i\in W$ be a loop index. Then
\begin{eqnarray}
(\K_i, \K_{i+1}) = (-1,1) &\implies& (\num_i, \num_{i+1}) = (0,0)\ \textit{and}\ \m_i + \n_i = 2 \expon_i - 2.\label{klmn_lemma-1}\\
(\K_{i}, \K_{i+1}) = (-1,2) &\implies& (\num_{i}, \num_{i+1}) = (1,0)\ \textit{or}\ (0,1)\ \textit{or}\ (0, 0).\label{klmn_lemma-2}\\
(\K_{i}, \K_{i+1}) = (-2,3) &\implies& (\num_{i}, \num_{i+1}) = (0,0).\label{klmn_lemma-3}\\
(\num_i, \num_{i+1})=(0,0) &\implies& (\K_{i}, \K_{i+1}) \neq (1,0).\label{klmn_lemma-4}\\
(\K_i, \num_i) = (0,1), \K_{i+1} \leq 0 &\implies& (\K_{i+1}, \num_{i+1}) = (0,0)\ \textit{and}\ \m_i + \n_i = 2 \expon_i - 2.\label{klmn_lemma-5}
\end{eqnarray}
\end{lemma}
\begin{proof}
We can check this by using Lemma \ref{K-index-ineq}.
More explicitly, we obtain the values of $(\num_i, \num_{i+1})$ by plugging the values of $(\K_i, \K_{i+1})$ into \eqref{ell_eq}. Then the values of $(\m_i,\n_i )$ will follow from \eqref{mplusn_eq}. For \eqref{klmn_lemma-4}, we get it from \eqref{mplusn_eq} and $m_i + \n_i\geq 0$. For the last property, we apply  \eqref{eq2} to obtain $0 \geq \K_{i+1} \geq \num_{i+1}$. This implies $\K_{i+1} = \num_{i+1} = 0$ and the statement follows again from \eqref{mplusn_eq}.
\end{proof}

\begin{lemma}\label{bound_ell_lemma}
Let $i \in S$ be a loop index where $\mathbf{\K_S}$ is a concatenation of (0)s and (-1,1)s, and suppose $i+1 \in S$ or $\K_{i+1} \leq 0$. Then $\num_i \leq 1$, and $\num_i=1$ implies $\m_i + \n_i = 2 \expon_i - 2$.
\end{lemma}
\begin{proof}
Suppose $\num_i \geq 2$. Then $\K_i=0$ by \eqref{klmn_lemma-1}. If $i+1 \in S$, since $\mathbf{\K_S}$ is a concatenation of (0)s and (-1,1)s, we have $\K_{i+1}=0$ or $-1$. So $\K_{i+1} \leq 0$. Then from \eqref{eq2},
\[
0 \geq \K_{i+1} \geq (\expon_i-1)(\num_i-1) + \num_{i+1}.
\]
But the right hand side is strictly positive, which is a contradiction.

So $\num_i \leq 1$ as desired. If $ \num_i = 1$, then we saw in the previous paragraph that $\K_i=0$. Since $\K_{i+1} \leq 0$, the remainder of the result follows from\eqref{mplusn_eq} and \eqref{klmn_lemma-2}.
\end{proof}

\subsubsection{Reduction to atomic types}

We are now ready to prove the first big lemma.
\begin{lemma}\label{splitting_lemma}
Let $X$ be a correlator of type $\X{X}{-1}$ for $W = \bigoplus W_i$. There is a unique $j$ such that $\K_{W_j}=1$. If $i \neq j$ then  $\K_{W_i}=0$.
\end{lemma}
\begin{proof}
We will prove that $\K_{W_j} \geq 0$ for each $j$. Then the result follows from \eqref{fact1}. We have three cases, depending on the atomic type of $W_j$.
If $W_j$ is a Fermat, then $\K_{W_j} \geq 0$ by Lemma \ref{fermat_splitting_lemma}.
If $W_j$ is a loop, then $\K_{W_j} \geq 0$ by Lemma \ref{loop_condition_lemma}.

So assume $W_j$ is a chain with variables $x_1, \ldots, x_N$. We know
$$\expon_N \bb_N = \num_N + \m_N + \n_N + 2 \leq \num_N + 2\expon_N.$$ This implies
\begin{equation}\label{eq1}
\K_N=\num_N-\bb_N+1 \geq \num_N(1 - 1/\expon_N) - 1 \geq -1.
\end{equation}
Moreover, if $\K_N=-1$, then the inequality above implies $\num_N=0$.

Assume for contradiction that $\K_{W_j} \leq -1$.
If $\{1, \ldots, N-1\}$ does not obey the NP-rule, then $\K_{N-1} < 0$. So by (\ref{key_equation}), $\K_{N-1} + \K_N \geq 0$ and also $\K_{N-2} \geq 0$. Thus $\{1, \ldots, N-2\}$ obeys the NP-rule. So Lemma \ref{loop_condition_lemma} shows $\sum_{i<N-1} \K_i \geq 0$. Combining $\K_{N-1} + \K_N \geq 0$, this contradicts our assumption that $\K_{W_j} \leq -1$.
Thus $\{1, \ldots, N-1\}$ obeys the NP-rule. So \eqref{sum-k-indices} in Lemma \ref{loop_condition_lemma} and \eqref{eq1} imply
$$\K_N=-1, \quad \sum_{i<N} \K_i = 0.$$
We use Lemma \ref{loop_condition_lemma} to find three possibilities for $\mathbf{\K_{W_j}}$.
In each case we use (\ref{ell_eq}) to compute $\mathbf{\num_{W_j}}$ and (\ref{mplusn_eq}) to compute $\mathbf{\m_{W_j}+\n_{W_j}}$. Also recall that \eqref{eq1} implies $\num_N=0$ so $\m_N+\n_N = 2 \expon_N-2$. We list all the possibilities here, using the notation $\mathbf{\K} = \mathbf{\K_{W_j}}$ and so forth. Also we let $M_i = 2 \expon_i-2$. We will omit the subscript in $M_i$ in the way that the $M$ which appears in the $i^{th}$ spot represents $M_i$.
\begin{itemize}
\item
$\mathbf{\K}= (\ldots, 0,-1),$
$\mathbf{\n+\m}=(\ldots,\expon_{N-1}, M)$;
\item
$\mathbf{\K}=(-1, 1, \ldots, -1, 1, -1),$
$\mathbf{\n+\m}=(M, 0, \ldots, M, 0, M)$;
\item
$\mathbf{\K}=(\ldots, 0, -1, 1, \ldots, -1, 1, -1),$
$\mathbf{\n+\m}=(\ldots, \expon_r, M, 0, \ldots, M, 0, M)$.
\end{itemize}
In each case, $\alpha$ and $\beta$ cannot both satisfy Property (P3) in Lemma \ref{properties_lemma}. This contradicts our assumption that $X$ is of type $\X{X}{-1}$.
\end{proof}

Now we know what $(\K_{W_1}, \K_{W_2}, \ldots )$ looks like: it is a tuple of zeros with a single 1. The next lemma investigates the form of $\mathbf{\K_{W_j}}$ when $\K_{W_j} = 0$.

\begin{lemma}\label{kis0_lemma}
Let $X$ be a correlator of type $\X{X}{-1}$. If $\K_{W}=0$, then
\begin{enumerate}
\item If $W$ is a Fermat then $\num=0$.
\item If $W$ is a loop then for all $i \in W$ we have $\num_i \leq 1$.
\item If $W$ is a chain with $\K_N \geq 0$ then for all $i \in W$ we have $\num_i \leq 1$.
\end{enumerate}
Furthermore, if $\num_i=1$ then $\m_i+\n_i \geq \expon_i$.
\end{lemma}
\begin{proof}

The claim for the Fermat type follows from (\ref{eq7}) when we substitute $\K = 0$.

If $W$ is a loop, by Lemma \ref{loop_condition_lemma} the tuple $\mathbf{\K_W}$ is some concatenation of of $(-1,1)$s and $(0)$s. Also, for every $i \in W$, certainly $i+1 \in W$. So this result follows from Lemma \ref{bound_ell_lemma}.

Finally, let $W$ be a chain. We will show that $\mathbf{\K}$ is a concatenation of (0)s and (-1,1)s. If the set $\{1, \ldots, N-1\}$ obeys the NP-rule, then $\sum_{i<N} \K_i \geq 0$. Since $\K_N \geq 0$ and $\sum_i \K_i = 0$, we must have $\sum_{i<N} \K_i = \K_N = 0$ and by Lemma \ref{loop_condition_lemma}, the vector $(\K_1, \ldots, \K_{N-1})$ is a concatentation of (0)s and (-1,1)s.

On the other hand, if $\{1, \ldots, N-1\}$ does not obey the NP-rule, then $\K_{N-1} \leq -1$. But then (\ref{key_equation}) shows that $\K_{N-2}$ cannot be negative, so the set $\{1, \ldots, N-2\}$ obeys the NP-rule. Also (\ref{key_equation}) shows that $\K_{N-1} + \K_N \geq 0$, so $\sum_{i  \leq N-2} \K_i \leq 0$. Then Lemma \ref{loop_condition_lemma} tells us that $\sum_{i  \leq N-2} \K_i = \K_{N-1} + \K_N = 0$. So $(\K_1, \ldots, \K_{N-2})$ is a concatentation of (0)s and (-1,1)s. Also $\K_N =-\K_{N-1}$, and plugging into (\ref{key_equation}) tells us that $0 \geq (1-\expon_i)(1+ \K_{N-1}).$ This means that $\K_{N-1} = -1$, so $\K_N$ = 1.

Thus $\mathbf{\K}$ is a concatenation of (0)s and (-1,1)s. Then Lemma \ref{bound_ell_lemma} proves this lemma for $i < N$.
Thus we only need to check when $i=N$. We've seen above that $\K_N$ is 0 or 1. If $\K_N=1$, we saw above that $\K_{N-1}=-1$ and so by Lemma \ref{klmn_lemma}, $\num_N =0$. If  $\K_N=0$, then (\ref{eq1}) says 
\[\num_N \leq \expon_N/(\expon_N-1).\]
If $\expon_N\geq 3$, then $\num_N \leq 1$, and if $\num_N=1$ then $\m_N + \n_N = 2 \expon_N-3 \geq \expon_N$.

If $\num_N=\expon_N=2$, then there are two possibilities:
\begin{itemize}
\item $\mathbf{\K} = (\ldots, -1, 1, 0)$, $\mathbf{\num} = (\ldots, 0, 0, 2)$, $\mathbf{m+n} = (\ldots, M, 1, M)$
\item $\mathbf{\K} = (\ldots, 0, 0)$, $\mathbf{\num} = (\ldots, 0, 2)$, $\mathbf{m+n} = (\ldots, \expon_{N-1}+1, M)$
\end{itemize}
Here $M = 2 \expon-2$. We used \eqref{ell_eq} and \eqref{mplusn_eq} to compute $\mathbf{\ell}$ and $\mathbf{m+n}$, respectively. In both cases, the form of $\mathbf{\m+\n}$ contradicts Property (P3) in Lemma \ref{properties_lemma}.

Then $\num_N\leq 1$, as desired. In fact, we will show that when $\expon_N=2$, we have $\num_N=0$, so the remainder of the lemma is vacuously true in this case.
For if $\num_N=1$ then there are three possibilities:
\begin{itemize}
\item $\mathbf{\K} = (\ldots, 0, -1, 1, \ldots, -1, 1, 0)$, $\mathbf{\num} = (\ldots,0, 0, 0, \ldots  0, 0, 1)$, $\mathbf{m+n} = (\ldots, \expon_r, M,$ $ 0, \ldots, M, 0, 1)$
\item $\mathbf{\K} = ( -1, 1, \ldots, -1, 1, 0)$, $\mathbf{\num} = (0, 0, \ldots, 0, 0, 1)$, $\mathbf{m+n} = (M, 0, \ldots, M, 0, 1)$
\item $\mathbf{\K} = (\ldots, 0, 0)$, $\mathbf{\num} = (\ldots, 0, 1)$, $\mathbf{m+n} = (\ldots, \expon_{N-1}, 1)$
\end{itemize}
In each case, the form of $\mathbf{\m+\n}$ contradicts Property (P3) in Lemma \ref{properties_lemma}.
\end{proof}

In the remainder of this paper we will repeatedly reconstruct correlators using Lemma \ref{reconstruction_lemma}.
This lemma allows us to write $X = A + B + C + S$ where $A$, $B$, and $C$ are $k$-point correlators and $S$ is a linear combination of correlators with fewer than $k$ insertions.
If $X$ is a correlator of type $\X{X}{-1}$ it is critical to understand when $A$, $B$, and $C$ have type $\X{X}{-1}$ and how $\mathbf{\K}^A$, $\mathbf{\K}^B$, and $\mathbf{\K}^C$ relate to $\mathbf{\K}^X$.

If $A\neq 0$ has the form of $\eqref{corr_form_eq}$ then by Lemma \ref{properties_lemma} it satisfies (P1) and (P2).
Moreover, if $A$ satisfies (P3) then it is of type $\X{X}{-1}$, and in this case $\bb_i^X = \bb_i^A$ because the changes in $\num_i^{X}, \m_i^X$, and $\n_i^X$ cancel each other out.
Hence $\K_i^X-\K_i^A = \num_i^X-\num_i^A$.

If $A$ does not satisfy (P3), then we reduce its insertions so they are in the standard basis, yielding an equivalent correlator $A'$ of type $\X{X}{-1}$.
Suppose the reconstruction only affected variables in the direct summand $W_j$ of $W$; i.e., $\num_i^{X}=\num_i^{A}, \m_i^X=\m_i^{A}$, and $\n_i^X=\n_i^{A}$ for all $i$ not in $W_j$. Then $\m_i^A = \m_i^{A'}$ and $\n_i^A=\n_i^{A'}$ for all $i \notin W_j$ (though we may have $\m_i^A \neq \m_i^{A'}$ for some $i \in W_j$).
Hence by the above discussion, $\K_{W_k}^X  = \K_{W_k}^A = \K_{W_k}^{A'}$ for all $k \neq j$. 
Then (P2) implies $\K_{W_j}^X = \K_{W_j}^A = \K_{W_j}^{A'}$ as well.

The same argument above works for the other two correlators $B$ and $C$ as well. These observations lead to the following remark.

\begin{remark}\label{rmk:reconstruction_invariance}
Suppose Lemma \ref{reconstruction_lemma} yields an equation $X = A+B+C+S$ with $X$, $A$, $B$, and $C$ correlators. We have the following results for $A$:
\begin{enumerate}
  \item If $A\neq 0$ is of type $\X{X}{-1}$, then $\K_i^X-\K_i^A = \num_i^X-\num_i^A$.
  \item If $A'$ is obtained from $A$ by writing its insertions in the standard basis, and if $\num_i^{X}=\num_i^{A}, \m_i^X=\m_i^{A}$, and $\n_i^X=\n_i^{A}$ for all $i$ not in $W_j$, then $\K_{W_j}^X = \K_{W_j}^A = \K_{W_j}^{A'}$ for all $k$ (including $k=j$).
\end{enumerate}
Furthermore, if $A$ is any nonvanishing correlator of type $\X{X}{-1}$, then
\begin{enumerate}
\item[(3)] If there exists $i \in W_j$ with $\num_i \geq 2$ and $W_j$ is a chain with $\K^A_N \geq 0$, a Fermat, or a loop, then by Lemma \ref{kis0_lemma} we have $\K^A_{W_j}=1$.
\end{enumerate}
The same results above are true for correlators $B$ and $C$ as well. 
\end{remark}

\begin{definition}\label{def-X0}
A correlator $X$ is called of type $\X{X}{0}$ for $W = \bigoplus W_j$ if $X$ is of type $\X{X}{-1}$ with $\K_{W_1}=1$ and $\K_{W_j}=0$ for $j> 1$, and
\begin{equation}\label{eq6}
\sum_{i \in W_1} \num_i \geq 2.
\end{equation}
\end{definition}

The main result of this section is to reconstruct correlators of type $\X{X}{-1}$ from correlators of type $\X{X}{0}$, see Proposition \ref{two_from_a_poly}. By using the Jacobi relations, it is not hard to get the following lemma.
\begin{lemma}\label{lem:reduce_chain}
Let $\alpha = x_1^{m_1}\ldots x_N^{m_N}$ be a monomial in the standard basis of a chain polynomial.
If $i<N$ then either $x_i\alpha=0$, or when $x_i\alpha$ is written in the standard basis as $x_1^{m'_1}\ldots x_N^{m'_N}$, with $m'_N=m_N$.
\end{lemma}
\begin{proposition}[Splitting Principle]\label{two_from_a_poly}
Any correlator of type $\X{X}{-1}$ can be reconstructed from correlators of type $\X{X}{0}$ and correlators with fewer insertions.
\end{proposition}
\begin{proof}
Let $X$ be a correlator of type $\X{X}{-1}$. Using Lemma \ref{splitting_lemma} and reordering the summands of $W$ if necessary, we may assume $\K_{W_1}=1$ and $\K_{W_j}=0$ for $j>1$.
If for all $j > 1$, the summand $W_j$ is a Fermat, then Lemma \ref{kis0_lemma} shows $\num_i = 0$ for $ i\in W_{j}$ for $j>1$. Then since $\sum_{i\in W} \num_i \geq 2$, we know (\ref{eq6}) holds.

Now assume that (\ref{eq6}) does not hold for $X$. Then we can assume that $W_2$ is a loop or chain polynomial and that there is $i \in W_2$ with $\num_i \geq 1$.

If $W_2$ is a chain, we do some preparatory reconstruction so Lemma \ref{kis0_lemma} is applicable.
Let us label the last variable of $W_2$ by $N_2$.
We know from (\ref{eq1}) that $\K_{N_2} \geq -1$. If $\K_{N_2} = -1$, we saw in the proof of Lemma \ref{splitting_lemma} that $\num_{N_2}=0$ (so in particular $i \neq N_2$) and $\m_{N_2} = \n_{N_2} = \expon_{N_2}-1$, so $X  = \fourptcorr{ x_i}{ \ldots}{ x_{N_2}\alpha}{ \beta }$. Now apply the Reconstruction Lemma \ref{reconstruction_lemma} with $\gamma = \beta$, $\delta = x_i$, $\epsilon = x_{N_2}$, and $\phi = \alpha$. Then
\begin{equation}\label{stwitch-nonnegative}
X=
\Fourptcorr{ x_{N_2}}{ \ldots}{ x_i \alpha}{ \beta}-\Fourptcorr{ x_{N_2}}{ \ldots}{ x_i \beta}{\alpha}+\Fourptcorr{ x_i}{\ldots}{ \alpha}{ x_{N_2}  \beta }+S.
\end{equation}
If these correlators are nonvanishing, by Remark \ref{rmk:reconstruction_invariance}(2) they each have $\K_{W_1}=1$. Also $\K_{N_2}\geq0$ for the first two since $\num_{N_2} \geq 1$. If $\K_{N_2}=-1$ for the last correlator,
then it vanishes because $\m_{N_2} = \expon_{N_2}-2 \neq \expon_{N_2}-1$. Thus we may assume $\K_{N_2} \geq 0$.

Now we return to the general case where $W_2$ is a chain or a loop. By Lemma \ref{kis0_lemma}, we know $\num_i=1$ and $X = \langle x_i, x_k, \ldots, x_i \alpha, \beta \rangle$ for some $k\neq i$. Apply the Reconstruction Lemma with $\gamma = \beta$, $\delta = x_k$, $\epsilon = x_i$, and $\phi = \alpha$, yielding
\begin{equation}\label{my_fav_eq}
X=\langle x_i, x_i, \ldots, \beta, x_k \alpha \rangle-\langle x_i, x_i, \ldots, \alpha, x_k \beta \rangle+\langle x_i, x_k, \ldots, \alpha, x_i  \beta \rangle+S.
\end{equation}

We need to check that if $W_2$ is a chain, Lemma \ref{kis0_lemma} is still applicable to each of these correlators; i.e., $\K_{N_2} \geq 0$. Now if $k \in W_2$ and $k = N_2$, swap the values of $i$ and $k$. This way we can assume $k \neq N_2$ (since $i$ and $k$ were distinct). 
There are two cases:
\begin{itemize}
\item If $i=N_2$ then all three of the correlators above have $\num_{N_2} > 0$, so each has $\K_{N_2} \geq 0$. 
\item If $i \neq N_2$, by Lemma \ref{lem:reduce_chain}, the exponents $\m_{N_2}$ and $\n_{N_2}$ do not change when we write the correlator insertions in the standard basis. Thus $\K_{N_2}$ is unaffected, and so is still nonnegative for each correlator. 
\end{itemize}

Now we apply Lemma \ref{kis0_lemma} to the first two correlators in (\ref{my_fav_eq}): if they do not vanish, $\K_{W_2} = 1$, since $\num_i \geq 2$. Thus these correlators have the desired form. Now, the third correlator still has $\K_{W_2} =0 $ by Remark \ref{rmk:reconstruction_invariance}(2).
Therefore, we can repeat this reconstruction on the third correlator.
Eventually the third correlator will have $\m_{i}+\n_{i} \leq \expon_i-1$, which contradicts Lemma \ref{kis0_lemma} (and thus this final correlator vanishes).
\end{proof}

\subsection{Atomic reconstruction and Fermat type}\label{proof_strategy}
Now let us restate Part (1) of Theorem \ref{reconstruction_theorem}. 
\begin{proposition}\label{almost_there_prop}
Let $W$ be an invertible polynomial and write $W^T$ as the sum of monomials $W^T = M_1 + \ldots + M_N$. 
Then the potential $\pot_{0, W^T, \prim}^{\rm SG}$ is completely determined by the Frobenius algebra structure and the correlators
\begin{equation}\label{final_type_corrs}
\Fourptcorr{ x_i}{ x_i}{ M_i/x_i^2}{ \Hessbase{W^T} }
\end{equation}
where $M_i$ is a Fermat summand $x^a$ with $a>2$; any monomial of a loop summand; or the final monomial of a chain summand.

Moreover, given an isomorphism of graded Frobenius algebras $\mirror: \Jac{W^T}\cong(\A{W},\star)$ satisfying \eqref{equ:mirror_sectors}, the potential $\pot_{0, W}^{\rm FJRW}$ is similarly determind by the correlators\\ $\Fourptcorr{ \mirror(x_i)}{ \mirror(x_i)}{ \mirror(M_i/x_i^2)}{ \mirror(\Hessbase{W^T}) }$.
\end{proposition}

We will give a complete proof of this proposition in this section and in Section \ref{chap:reconstruction}. The proof uses the WDVV equations, Jacobi relations (FJRW ring relations), and properties shared by correlators in both models. 
Since our proof of the first claim in Proposition \ref{almost_there_prop} essentially uses only Lemma \ref{vanishing_lemma} and Lemma \ref{reconstruction_lemma}, the second claim is an immediate corollary.

\subsubsection{Atomic reconstruction}
After reordering the summands of $W$ so that $x_i$ is in $W_1$, the correlators in \eqref{final_type_corrs} are all of type $\X{X}{0}$.
We say the correlators in \eqref{final_type_corrs} have \emph{final type}. According to Splitting Principle Proposition \ref{two_from_a_poly},
in order to prove Proposition \ref{almost_there_prop}, it suffices to reconstruct correlators of type $\X{X}{0}$ from correlators of final type.
We will prove this reconstruction in three cases, depending on whether the atomic $W_1$ is a Fermat, chain, or loop.
This is called \emph{atomic reconstruction}.

In each case, we filter the correlators with several types, denoted by $\X{F}{k}$, $\X{C}{k}$, and $\X{L}{k}$, respectively. 
Correlators of type $\X{F}{0}$ (or $\X{C}{0}$, $\X{L}{0}$) are correlators of type $\X{X}{0}$ where $W_1$ is a Fermat (or chain, loop) polynomial.
The types with the largest values of $k$ are correlators of final type.

For each atomic type, we prove Proposition \ref{almost_there_prop} by induction on $k$. 
In the $k$-th step, we reconstruct a correlator of type $\X{F}{k-1}$ (or $\X{C}{k-1}, \X{L}{k-1}$) from correlators of type $\X{F}{\geq k}$ (or $\X{C}{\geq k}, \X{L}{\geq k}$), correlators that vanish, and correlators with fewer insertions.

\begin{remark}
Let $X$ be a correlator of type $\X{X}{-1}$ with $\K_{W_j}=1$ and $\K_{W_i}=0$ for $i \neq j$ and $\sum_{r \in W_j} \num_r \geq 2$. 
By reordering the summands of $W$ we can assume $j=1$, so $X$ is of type $\X{X}{0}$.
In the remainder of our reconstruction argument we will make this assumption whenever possible.
When $X$ is of type $\X{X}{0}$, we let $\mathbf{K} = \mathbf{K_{W_1}}$ and we use $\mathbf{\num}, \mathbf{\m}$, and $\mathbf{\n}$ similarly.
\end{remark}

\subsubsection{Atomic reconstruction of Fermat type}\label{fermat_reconstruction_section}
This subsection proves Proposition \ref{almost_there_prop} for $W = \bigoplus W_i$ when $W_1$ is a Fermat polynomial $W_1= x^{\expon}$ with $\expon >2$.
We start with the following definition.

\begin{definition}
 Let $X$ be of type $\X{F}{0}$ for W. Then
\begin{itemize}
\item
$X$ is of type $ \X{F}{1}$ if $X=\Fourptcorr{ x}{ x}{ x^{\expon-2}  \alpha}{ x^{\expon - 2} \beta }$.
\item
$X$ is of type $ \X{F}{2}$ if $X = \Fourptcorr{ x}{ x}{ x^{\expon-2}}{ \Hessbase{W^T}}$.
\end{itemize}
\end{definition}

Now we prove Proposition \ref{almost_there_prop} in two steps.\\

\noindent
{\bf Step 1.}
Let $X$ be a correlator of type $\X{F}{0}$.
Using (\ref{eq7}), we have
$$\K \geq \num\left(1 - {1\over\expon}\right) + {1\over\expon} - 1.$$
Plugging in $\K=1$ we get $\num \leq 2$, so $\num = 2$. Then (\ref{eq3}) shows that $\m+\n = 2 \expon-4$.

So we know $X = \langle x_{i_1} \not \in W_1, \; \ldots, \; x_{i_s} \not \in W_1, x, x, x^{\expon-2} \alpha, x^{\expon-2}\beta \rangle$ with $\alpha,\; \beta \in W-W_1$.
If $X$ has four insertions, then it is of type $\X{F}{1}$ and we are done.

If not, there is some insertion $x_i$ where $i \not \in W_1$.
Apply the Reconstruction Lemma \ref{reconstruction_lemma} with $\delta = x_i$, $\epsilon = x$, $\phi = x^{\expon-3}  \alpha$, and $\gamma = x^{\expon-2} \beta$.
Then $\epsilon  \gamma$ has a factor of $x^{\expon-1}$ which is zero in $\Jac{W^T}$.
But the two remaining terms (with $\gamma  \delta$ and $\phi  \delta$) have $\num=3$, and so these correlators must also vanish (if $\K_{W_1}=1$ then $\num_1 \leq 2$; if $\K_{W_j}=0$ then $\num_j = 0$ by Lemma \ref{kis0_lemma}).
So we can reconstruct $X$ from correlators with strictly fewer insertions.\qed \\

\noindent
{\bf Step 2.}
Let $X = \langle x, x, x^{\expon-2}  \alpha, x^{\expon - 2}\beta \rangle$ be a correlator of type $\X{F}{1}$. Apply the Reconstruction Lemma \ref{reconstruction_lemma} with $\gamma = x$, $\epsilon = x^{\expon-2}$, $\phi = \alpha$, and $\delta = x^{\expon-2} \beta$. Then $\gamma \epsilon$ and $\gamma  \delta$ both vanish because they have a factor of $x^{\expon-1}$, and we get
$$X= \Fourptcorr{ x}{ x}{ x^{\expon-2}}{ x^{\expon - 2} \alpha\beta }.$$
Now by the Dimension Axiom in Lemma \ref{vanishing_lemma}, if $X\neq0,$ the product $\alpha\beta$ must be proportional to the unique element of top degree in $\Jac{W^T-W_1^T}$. Hence $X$ is a scalar multiple of a correlator of type $ \X{F}{2}$. \qed

The strategies to prove Proposition \ref{almost_there_prop} for chain and loop types are similar to the one for Fermat type, but much more complicated. We leave a complete proof in Section \ref{chap:reconstruction}.


\section{Computation}\label{computation}


The goal of this section is to compute the correlators in Theorem \ref{reconstruction_theorem}.
In the A-model side, the most powerful tool is from an orbifold Grothendieck-Riemann-Roch formula. When the correlator in Theorem \ref{reconstruction_theorem} is concave, then the virtual cycle can be extracted from a top Chern class \eqref{concave-def}, which will imply the very useful formula \eqref{eqhe} by \cite{C}. 
By analyzing the combinatorical aspect of the insertions in the A-model correlators, we will show that most of them in Theorem \ref{reconstruction_theorem} are concave. We will compute these concave correlators in this section and leave the computations of the nonconcave cases in Appendix  \ref{sec-exceptional}.
In the B-model side, the values of the correlators in Theorem \ref{reconstruction_theorem} follow directly from Li-Li-Saito's perturbative formula \cite{LLSaito}.

\subsection{A-model computation: concavity and nonconcavity}
We prove Part (2) of Theorem \ref{reconstruction_theorem}.
\begin{proposition}\label{Aside_corr_theorem}
Let $q_i$ be the $i^{th}$ weight of $W$ and $M_i$ be any monomial of a Fermat or loop summand, or the final monomial of a chain summand in $W^T$. Then
\[
\langle \mirror(x_i), \mirror(x_i), \mirror(M_i/x_i^2), \mirror(\Hessbase{W^T}) \rangle_0^W = q_i.
\]
\end{proposition}

For notational convenience, in this section we will let $\agen_i=\mirror(x_i), S_i=\mirror(M_i/x_i^2),$ and $H=\mirror(\Hessbase{W^T})$.
By the symmetry of a loop polynomial, it suffices to prove Proposition \ref{Aside_corr_theorem} for $i=N$. Thus it suffices to compute the correlator
\begin{equation}\label{correlator-N}
X=\langle \agen_N, \agen_N, S_N, H \rangle.
\end{equation}

\begin{lemma}
Suppose $W = \bigoplus W_i$ and $x_N$ is a variable in the summand $W_j$. Then
\[
\langle \agen_N, \agen_N, S_N, \HessbaseA{W^T} \rangle_0^W = \langle \agen_N, \agen_N, S_N, \HessbaseA{W^T_j} \rangle_0^{W_j}.
\]
\end{lemma}
\begin{proof}
By Theorem 4.2.2 of \cite{FJR}, we have
\[
\Lambda_{0, 4}^W(\agen_N, \agen_N, S_N, \HessbaseA{W^T}) = \Lambda_{0, 4}^{W_j}(\agen_N, \agen_N, S_N, \HessbaseA{W^T_j})\prod_{i \neq j} \Lambda^{W_i}_{0, 4}(\one, \one, \one, \HessbaseA{W^T_i}).
\]
Here $\Lambda^{W_i}_{0, 4}(\one, \one, \one, \HessbaseA{W^T_i})\in H^0(\M{0}{4})$, so we treat it as a scalar. By Axiom C4 of Theorem 4.2.2 in \cite{FJR}, we get
$$\Lambda^{W_i}_{0, 4}(\one, \one, \one, \HessbaseA{W^T_i})=\int_{\M{0}{3}}\Lambda_{0, 3}(\one, \one, \HessbaseA{W^T_i}) = \langle \one, \HessbaseA{W^T_i} \rangle=1.$$
\end{proof}

Because of this result, in the remainder of this section we will assume that $W$ is an atomic polynomial.
Before we start the computation, let us state some useful formulas for each atomic type.
Recall that $\rho_j^{(i)}$ is the $(i,j)$-th entry of the matrix $E_W^{-1}$.\\

\noindent{\bf Fermat formulas.}  Let $W=x^a$. Then $i=N=1$, and
\begin{equation}\label{fermat-weight}
q_1=\rho_1^{(1)}={1\over a}.
\end{equation}

\noindent{\bf Chain formulas.} Let  $W = x_1^{\expon_1} x_2 + x^{\expon_2} x_3 + \ldots + x_{N-1}^{\expon_{N-1}}x_N + x_N^{\expon_N}$. Then
\[
\AW_W = \left( \begin{array}{ccccc}
\expon_1 & 1 & &&\\
& \expon_2 & 1 &&\\
&& \ddots & \ddots &\\
&&&\expon_{N-1} & 1\\
&&&& \expon_N \end{array} \right),
\]
and
\begin{eqnarray}\label{chain-exponent}
\begin{split}
\rho_j^{(i)}
&=
(-1)^{j-i}\prod_{k=i}^{j}{1\over\expon_k},  & j\geq i;\\
\rho_j^{(i)}
&=0,  &j<i.
\end{split}
\end{eqnarray}
Since $q_i = \sum_{j=1}^N \rho_j^{(i)}$, we have
\begin{equation}\label{chain-weight}
q_i =\sum_{j=i}^{N}(-1)^{j-i}\prod_{k=i}^{j}{1\over \expon_k}.
\end{equation}

\noindent{\bf Loop formulas.} Let  $W = x_1^{\expon_1} x_2 + x^{\expon_2} x_3 + \ldots + x_{N-1}^{\expon_{N-1}}x_N + x_N^{\expon_N}x_1$. Then
\[
\AW_W = \left( \begin{array}{ccccc}
\expon_1 & 1 & &&\\
& \expon_2 & 1 &&\\
&& \ddots & \ddots &\\
&&&\expon_{N-1} & 1\\
1&&&& \expon_N \end{array} \right).
\]
Define
$$L_W=\left(\prod_{k=1}^N\expon_k+(-1)^{N+1}\right)^{-1}.$$
Then
\begin{eqnarray}\label{loop-exponent}
\begin{split}
\rho_j^{(i)} &= (-1)^{j-i}\left(\prod_{k=j+1}^{N}\expon_k\right)\left(\prod_{k=1}^{i-1}\expon_k\right) L_W, &  j\geq i,\\
\rho_j^{(i)} &= (-1)^{N+j-i}\left(\prod_{k=j+1}^{i-1}\expon_k\right) L_W,  & j<i.
\end{split}
\end{eqnarray}
Here  we use the convention that an empty product is 1.
These formulas lead to the following expression for the $i^{th}$ weight of $W$:
\begin{equation}\label{loop-weight}
q_i =\sum_{j=i}^{N}(-1)^{j-i}\left(\prod_{k=j+1}^{N}\expon_k\right)\left(\prod_{k=1}^{i-1}\expon_k\right)L_W+\sum_{j=1}^{i-1}(-1)^{N+j-i}\left(\prod_{k=j+1}^{i-1}\expon_k\right)L_W
.
\end{equation}

\subsubsection{Combinatorial preparation}
Let $c$ be an integer such that $c\in [-2,2]$, we define 
\begin{equation}\label{phase-number}
Y_{i,c}:=q_i+c\rho_N^{(i)}.
\end{equation}
The following results are useful later.
\begin{lemma}\label{chain-comb}
For $W = x_1^{\expon_1} x_2 + x_2^{\expon_2} x_3 + \ldots + x_{N-1}^{\expon_{N-1}}x_N + x_N^{\expon_N}$ with $\expon_N>2$, then $
Y_{i,c} \in (0,1)
$ except:
\begin{itemize}
\item $Y_{N,-2} \in (-1,0)$.
\item $Y_{N,-1}=0$.
\item $Y_{N-1,2}=0$ and $Y_{N,2}=1$ if $\expon_N=3$.
\end{itemize}
\end{lemma}
\begin{proof}
From \eqref{phase-number}, \eqref{chain-weight}, and \eqref{chain-exponent}, we have
$$Y_{i,c} =\sum_{j=i}^{N}(-1)^{j-i}\prod_{k=i}^{j}{1\over \expon_k}+c(-1)^{N-i}\prod_{k=i}^{N}{1\over \expon_k}.$$
If $i=N$, then
$$Y_{N,c}={c+1\over\expon_N}$$
and the result follows since $\expon_N>2$.
If $i<N$, then since $q_i$ in \eqref{chain-weight} is an alternating series, with strictly decreasing absolute value for each term, and since $|c|\leq2$, the result follows from
$$0\leq (1-{3\over\expon_N})\prod_{k=i}^{N-1}{1\over\expon_k}\leq\prod_{k=i}^{N-1}{1\over\expon_k}-(|c|+1)\prod_{k=i}^N{1\over\expon_k}\leq Y_{i,c}<{1\over\expon_i}+|c|\prod_{k=i}^{N}{1\over \expon_k}\leq {1\over\expon_i}+{2\over \expon_i\expon_N}<1.$$
Here $Y_{i,c}=0$ if and only if the first three equalities hold. That happens if and only if $\expon_N=3, c=2$, and $i=N-1$.
\end{proof}

\begin{lemma}\label{loop-comb}
Let $W = x_1^{\expon_1} x_2 + x^{\expon_2} x_3 + \ldots + x_{N-1}^{\expon_{N-1}}x_N + x_N^{\expon_N}x_1$.
\begin{enumerate}
\item
If $N-i$ is odd, then $Y_{i,c}\in (0,1)$ except
\begin{itemize}
\item $Y_{1,1}=0$ if $N=\expon_N=2$.
\item $Y_{1,2}=0$ if $N=2, \expon_N=3$.
\item $Y_{N-1,2}\in (-1,0)$ if $N>2$ and $\expon_N=2$.
\end{itemize}

\item
If $N-i$ is even, then $Y_{i,c}\in (0,1)$ except
\begin{itemize}
\item $Y_{2,1}=1$ if $i=N=2, \expon_N=2$.
\item $Y_{N,2}=1$ if $i=N>2, \expon_N=3$.
\item $Y_{N,2}\in (1,2)$ if $i=N>2, \expon_N=2$.
\item $Y_{N,c}\in (-1,0)$ if $N>2$ and $c=-1, -2$.
\end{itemize}
\end{enumerate}
\end{lemma}
\begin{proof}

\noindent{\bf (1) $N-i$ is odd.} In this case, by \eqref{loop-exponent} and \eqref{loop-weight}, we first write $Y_{i,c}$ as
\begin{eqnarray}\label{loop-first-term}
\begin{split}
Y_{i,c}=&\sum_{r=1}^{N-1-i\over2}(\expon_{i-1+2r}-1)\left(\prod_{k=i+2r}^{N}\expon_k\right)\left(\prod_{k=1}^{i-1}\expon_k\right)L_W\\
&+\left(\expon_N-(c+1)\right)\left(\prod_{k=1}^{i-1}\expon_k\right)\,L_W\\
&+\sum_{r=1}^{\floor{{i\over2}}}\left(\prod_{k=2r}^{i-1}\expon_k-\prod_{k=2r+1}^{i-1}\expon_k\right)L_W
.
\end{split}
\end{eqnarray}
If $N=2$, then $i=1$ and the result follows from
$$Y_{1,c}=\left(\expon_N-(c+1)\right)\,L_W.$$
If $N>2$, then the sum of first and third line on the RHS of Equation \eqref{loop-first-term} is strictly positive. We know $Y_{i,c}>0$ as long as the second line is non-negative or the first line is non-zero. Thus $Y_{i,c}<0$ only if
$$c=\expon_N=2 \quad \text{and}\quad i=N-1.$$
In order to prove the other side of the inequality, if $i<N-1$, we rewrite $Y_{i,c}$ as
\begin{eqnarray}\label{loop-second-term}
\begin{split}
Y_{i,c}=&\left(\prod_{k=i+1}^{N}\expon_k\right)\left(\prod_{k=1}^{i-1}\expon_k\right)\,L_W\\
&-\sum_{r=1}^{N-3-i\over2}(\expon_{i+2r}-1)\left(\prod_{k=i+2r+1}^{N}\expon_k\right)\left(\prod_{k=1}^{i-1}\expon_k\right)\,L_W\\
&-\left(\expon_{N-1}\expon_N-(\expon_N-c)\right)\left(\prod_{k=1}^{i-1}\expon_k\right)\,L_W\\
&-\sum_{r=1}^{\floor{{i\over2}}}\left(\prod_{k=2r-1}^{i-1}\expon_k-\prod_{k=2r}^{i-1}\expon_k\right)L_W
.
\end{split}
\end{eqnarray}
Since
$$\left(\expon_{N-1}\expon_N-(\expon_N-c)\right)\left(\prod_{k=1}^{i-1}\expon_k\right)\,L_W\geq0,$$
we get
$$Y_{i,c}\leq\left(\prod_{k=i+1}^{N}\expon_k\right)\left(\prod_{k=1}^{i-1}\expon_k\right)\,L_W<1.$$
If $i=N-1$, the result follows from a similar discussion by rewriting $Y_{N-1,c}$ from \eqref{loop-weight},
$$Y_{N-1,c}=(\expon_N-c)\left(\prod_{k=1}^{N-2}\expon_k\right)\,L_W+\sum_{j=0}^{N-2}(-1)^{j+1}\left(\prod_{k=j+1}^{N-2}\expon_k\right)\,L_W.$$

\noindent{\bf (2) $N-i$ is even.} In this case, the result follows from a similar discussion by rewriting $Y_{i,c}$ as
\begin{eqnarray*}
\begin{split}
Y_{i,c}=&\sum_{r=1}^{N-i\over2}(\expon_{i-1+2r}-1)\left(\prod_{k=i+2r}^{N}\expon_k\right)\left(\prod_{k=1}^{i-1}\expon_k\right)L_W\\
&+\left((1+c)\expon_1-1\right)\left(\prod_{k=2}^{i-1}\expon_k\right)\,L_W\\
&+\sum_{r=1}^{\floor{{i-1\over2}}}\left(\prod_{k=2r+1}^{i-1}\expon_k-\prod_{k=2r+2}^{i-1}\expon_k\right)L_W.
\end{split}
\end{eqnarray*}
\end{proof}

Now we continue with our computation of $X=\langle\agen_N,\agen_N,S_N,H\rangle$. We notice that $S_N=\agen_N^{\expon_N-2}$ when $W$ is a Fermat polynomial and $S_N = \agen_{N-1}\agen_N^{\expon_N-2}$ when $W$ is a chain or loop polynomial as above.
We will sometimes use $\agen_N,S_N$ and $H$ to denote the correponding sector and use the symbols $\phase{\agen_N}^{(i)}$, $\phase{S_N}^{(i)}$, and $\phase{H}^{(i)}$ to refer to the $i$-th phase of these sector.

\begin{lemma}
For each atomic type polynomial $W$ with no variable of weight $1/2$, then
\begin{align*}
\phase{\agen_N}^{(i)} &=q_i+\rho_N^{(i)}-\floor{q_i+\rho_N^{(i)}},\\
\phase{H}^{(i)}&=1-q_i,\\
\phase{S_N}^{(i)}&
=q_i-2\rho_N^{(i)}+\delta_N^i.
\end{align*}
Recall that $\LL_i$ is the $i$-th orbifold line bundle in the $W$-structure. If $(N, \expon_N)\neq(2,2)$, then on each smooth fiber, the degree of $\LL_i$ is
\begin{equation}\label{degree-computation}
\ld_i:=\deg\LL_i=-1-\delta_N^i.
\end{equation}
\end{lemma}
\begin{proof}
The proof is a direct computation using the ring isomorphism in Theorem \ref{mirror-algebra}. Lemma \ref{chain-comb}, Lemma \ref{loop-comb} and Equation \eqref{fermat-weight} show that the quantities listed are in $[0,1)$. In particular, if $N=\expon_N=2$, i.e., $W = x_1^{\expon_1} x_2 + x_2^{2} x_1$, then
$$q_N+\rho_N^{(N)}=1.$$
Otherwise, $q_i+\rho_N^{(i)}\in [0,1)$. 
Then \eqref{degree-computation} follows, since
\begin{align*}
\ld_i:=\deg\LL_i &= 2q_i-2\phase{\agen_N}^{(i)} - \phase{S_N}^{(i)}-\phase{H}^{(i)}\\
&=2q_i-2(q_i+\rho_N^{(i)}) - (q_i-2\rho_N^{(i)}+\delta_N^i) - (1-q_i)\\
&=-1-\delta_N^i.
\end{align*}
\end{proof}
The following $\Gmax{W}$-decorated graphs will be useful in the computation of $X$.
\begin{figure}[H]
\centering
\begin{tikzpicture}[xscale=1,yscale=1]
\draw [-] (0, 0)--(north west:1);
\draw [-] (0, 0)--(south west:1);
\draw [-] (0, 0)--(.85, 0);
\draw [-] (1.15, 0)--(2, 0);
\draw [-] (2, 0)--+(north east:1);
\draw [-] (2, 0)--+(south east:1);
\node [ left] at (north west:1) {$\agen_N$};
\node [ left] at(south west:1) {$\agen_N$};
\node [ right] at (2.7, .7) {$H$};
\node [right] at (2.7, -.7) {$S_N$};
\node [above] at (.5, 0){$\gamma_{1,+}$};
\node [above] at (1.5, 0){$\gamma_{1,-}$};
\draw [fill=black] (0, 0) circle [radius = .05];
\draw [fill=black] (2, 0) circle [radius = .05];

\begin{scope}[shift ={(5, 0)}]
\draw [-] (0, 0)--(north west:1);
\draw [-] (0, 0)--(south west:1);
\draw [-] (0, 0)--(.85, 0);
\draw [-] (1.15, 0)--(2, 0);
\draw [-] (2, 0)--+(north east:1);
\draw [-] (2, 0)--+(south east:1);
\node [ left] at (north west:1) {$\agen_N$};
\node [ left] at(south west:1) {$H$};
\node [ right] at (2.7, .7) {$\agen_N$};
\node [right] at (2.7, -.7) {$S_N$};
\node [above] at (.5, 0){$\gamma_{2,+}$};
\node [above] at (1.5, 0){$\gamma_{2,-}$};
\draw [fill=black] (0, 0) circle [radius = .05];
\draw [fill=black] (2, 0) circle [radius = .05];
\end{scope}

\begin{scope}[shift ={(10, 0)}]
\draw [-] (0, 0)--(north west:1);
\draw [-] (0, 0)--(south west:1);
\draw [-] (0, 0)--(.85, 0);
\draw [-] (1.15, 0)--(2, 0);
\draw [-] (2, 0)--+(north east:1);
\draw [-] (2, 0)--+(south east:1);
\node [ left] at (north west:1) {$\agen_N$};
\node [ left] at(south west:1) {$S_N$};
\node [ right] at (2.7, .7) {$\agen_N$};
\node [right] at (2.7, -.7) {$H$};
\node [above] at (.5, 0){$\gamma_{2,-}$};
\node [above] at (1.5, 0){$\gamma_{2,+}$};
\draw [fill=black] (0, 0) circle [radius = .05];
\draw [fill=black] (2, 0) circle [radius = .05];
\end{scope}
\end{tikzpicture}
\caption{Boundary strata on $\overline{\mathscr{W}}_{0,4}(\agen_N,\agen_N,S_N,H)$}
\label{boundary}
\end{figure}
Note that the two graphs on the right are the same. Here the element $\gamma_{k,\pm}\in G_W$ is chosen uniquely such that the Interger Degree Axiom \eqref{lbd_ax} is satisfied for each component.  It is possible that $\A{\gamma_{k,\pm}}=\emptyset$.
Let $\gamma_{k,\pm}^{(i)}$ be the $i$-th phase of $\gamma_{k,\pm}$ and
\begin{equation}\label{phase-degree}
\left\{
\begin{array}{ll}
h_{1,+}^{(i)}:=q_i-2\phase{\agen_N}^{(i)},\quad
&h_{1,-}^{(i)}:=q_i-\phase{S_N}^{(i)}-\phase{H}^{(i)},\\
h_{2,+}^{(i)}:=q_i-\phase{\agen_N}^{(i)}-\phase{H}^{(i)},\quad
&h_{2,-}^{(i)}:=q_i-\phase{\agen_N}^{(i)}-\phase{S_N}^{(i)}.
\end{array}
\right.
\end{equation}
Let $\ell_{k,+}^{(i)}$ ($\ell_{k,-}^{(i)}$) be the degree of the line bundle $\mathscr{L}_i$ on the left(right) component of the $k$-th graph above for $k=1, 2$. It follows that
\begin{eqnarray}
\ell_{k,\pm}^{(i)}&=&\floor*{h_{k,\pm}^{(i)}}, \label{graph-line-bundle-degree}\\
\gamma_{k,\pm}^{(i)}&=&h_{k,\pm}^{(i)}-\floor*{h_{k,\pm}^{(i)}}. \label{cut-phase}
\end{eqnarray}
In particular, if $(N, \expon_N)\neq(2,2)$, we can use the symbol in \eqref{phase-number} to rewrite the following numbers
\begin{equation}\label{phase-notation}
\left\{
\begin{array}{llll}
q_i=Y_{i,0}, & \agen_N^{(i)}=Y_{i,1}, & H^{(i)}=1-Y_{i,0}, & S_N^{(i)}=Y_{i,-2}+\delta_N^i, \\
h_{1,+}^{(i)}=-Y_{i,2}, & h_{1,-}^{(i)}=Y_{i,2}-1-\delta_N^i, & h_{2,+}^{(i)}=Y_{i,-1}-1, & h_{2,-}^{(i)}=-Y_{i,-1}-\delta_N^i.
\end{array}
\right.
\end{equation}
\subsubsection{Concavity Axiom}
Now we introduce the Concavity Axiom from \cite{FJR} to compute the necessary FJRW invariants. We recall the universal $W$-structure $(\LL_1,\ldots,\LL_N)$ on the universal curve $\pi: \mathscr{C}\to\overline{\mathscr{W}}_{g,k}(\gamma_1,\ldots,\gamma_k)$. A correlator
$\langle\xi_1,\ldots,\xi_k\rangle_{g}$
is called \emph{concave} if all the insertions $\xi_j$ are narrow and 
for each geometric point $[C]\in \overline{\mathscr{W}}_{g,k}(\gamma_1,\ldots,\gamma_k)$,
\begin{equation}\label{concave-revision}
H^0\left(C, \LL_i\right)=0, \quad 1\leq i\leq N.
\end{equation}
In this case, $\pi_*(\bigoplus_{i=1}^N\LL_i)=0$, $R^1\pi_*(\bigoplus_{i=1}^N\LL_i)$ is locally free, and the Concavity Axiom (see Theorem 4.1.8 in \cite{FJR}) implies
\begin{equation}\label{concave-def}
[\W{g}{k}(\Gamma_{\gamma_1,\ldots,\gamma_k})]^{\rm vir}=c_{\rm top}\left(R^1\pi_*(\bigoplus_{i=1}^N\LL_i)\right)^\vee\cap[\W{g}{k}(\Gamma_{\gamma_1,\ldots,\gamma_k})].
\end{equation}
Here $c_{\rm top}$ is the top Chern class and $[\W{g}{k}(\Gamma_{\gamma_1,\ldots,\gamma_k})]$ is the fundamental cycle.
Then Theorem 1.1.1 in \cite{C} expresses the FJRW virtual cycles in terms of tautological classes on $\M{g}{k}$. In particular, on $\M{0}{4}$ we have $R^1\pi_*\LL_{i}\neq0$ for some unique $\LL_i$ and
\begin{equation}\label{eqhe}
\langle\xi_1,\ldots,\xi_k\rangle=
\int_{\M{0}{4}}
\left(\frac{B_{2}(q_i)}{2} \kappa_1 -\sum_{j=1}^4 \frac{B_{2}(\Theta^{(i)}_{\gamma_j})}{2}\psi_j  +\sum_{\Gamma_{cut}}\frac{B_{2}(\Theta^{(i)}_{\gamma^+})}{2}[\Gamma_{cut}].\right)
\end{equation}
Here $\kappa_1$ is the first kappa class, $\psi_j$ is the $j$-th psi class, $B_{2}$ is the second Bernoulli polynomial that $B_2(x)=x^2-x+{1\over6}$, and $\Gamma_{cut}$ are all the fully $G_W$-decorated graphs on the boundary. For the correlator $X=\langle \agen_N, \agen_N, S_N, H \rangle$ in \eqref{correlator-N}, the graphs are listed in Figure \ref{boundary}.

\begin{lemma}\label{lemma-chiodo}
Consider the correlator $X=\langle \agen_N, \agen_N, S_N, H \rangle$ in \eqref{correlator-N}.
Assume $\expon_N>2$. If for $k=1,2$, the unordered pairs $(\ell_{k,+}^{(i)}, \ell_{k,-}^{(i)})$ satisfy
\begin{eqnarray}\label{concave-cond}
\begin{split}
(\ell_{k,+}^{(N)}, \ell_{k,-}^{(N)})\in \{(-2,-1)\}; \quad
(\ell_{k,+}^{(i)}, \ell_{k,-}^{(i)})\in \{(-1,-1), (-1,0)\}, \quad i<N,
\end{split}
\end{eqnarray}
and the sectors of $\Theta_N$, $S_N$ and $H$ are narrow, then
\begin{equation}\label{correlator_formula}
X =\frac{1}{2}\left[ -q_{N}(1 - q_{N}) + \sum_{j=1}^4 \Theta^{(N)}_{\gamma_j}(1 - \Theta^{(N)}_{\gamma_j})  - \sum_{\Gamma_{cut}}\Theta^{(N)}_{\gamma_+}(1 - \Theta^{(N)}_{\gamma_+}) \right].
\end{equation}
\end{lemma}
\begin{proof}
For a singular curve $[C]\in \overline{\mathscr{W}}_{0,4}(\agen_N,\agen_N,S_N,H)$, from \eqref{phase-degree} and \eqref{graph-line-bundle-degree}, we know 
\begin{equation}\label{narrow-singular}
\ell_{k,+}^{(i)}+\ell_{k,-}^{(i)}=\floor*{h_{k,+}^{(i)}}+\floor*{h_{k,-}^{(i)}}=-1-\delta_N^i-\delta_{\rm narrow},
\end{equation}
where $\delta_{narrow}$ is 1 when the local isotropy group at the node acts nontrivially on the fiber and 0 otherwise.  Thus we can check \eqref{concave-cond} holds.

To apply Concavity axiom \eqref{concave-def} we must check \eqref{concave-revision}, which is true if the line bundle degrees are negative on all components of all stratifications.
Combine \eqref{degree-computation}, we only need to check when $(\ell_{k,+}^{(i)}, \ell_{k,-}^{(i)})= (-1,0)$ and $i<N$. According to \eqref{narrow-singular}, the unique node $n\in C$ must be broad. We denote the normalization of $C$ by $p: C_1\coprod C_2\to C$, and get a long exact sequence
\begin{eqnarray*}
0\to H^0(C, \LL_i|_C)\to H^0(C_1, \LL_i|_{C_1})\oplus H^0(C_2, \LL_i|_{C_2})\to H^0(n, \LL_i|_n) \\
\to H^1(C, \LL_i|_C) \to H^1(C_1, \LL_i|_{C_1})\oplus H^1(C_2, \LL_i|_{C_2})\to0.
\end{eqnarray*}
Let us focus on the first line. Since $(\ell_{k,+}^{(i)}, \ell_{k,-}^{(i)})= (-1,0)$, the third term is just $\mathbb{C}$. The broadness implies that the last arrow is an isomorphism. Thus \eqref{concave-revision} follows.

Now we apply Riemann-Roch formula to \eqref{degree-computation} and \eqref{concave-cond}. Then $R^1\pi_*\LL_{i<N}=0$, and $R^1\pi_*\LL_N$ is a vector bundle of rank $1$.  Now formula \eqref{correlator_formula} follows from
Equation \eqref{eqhe} and
$$\int_{\M{0}{4}} \kappa_1 = \int_{\M{0}{4}} \psi_i = \int_{\M{0}{4}} [ \Gamma_{cut}]=1.$$
\end{proof}


\subsubsection{Chain Computation}
Let  $W = x_1^{\expon_1} x_2 + x_2^{\expon_2} x_3 + \ldots + x_{N-1}^{\expon_{N-1}}x_N + x_N^{\expon_N}$ with $\expon_N>2$.
Combine Lemma \ref{chain-comb} with the notation in \eqref{phase-notation}, we get the following corollary.
\begin{corollary}
The sectors $\Theta_N$, $S_N$, and $H$ are narrow. That is, $1\leq i\leq N,$
$\agen_N^{(i)}, H^{(i)},$ and $S_N^{(i)}$ are in $(0,1)$ for all $1\leq i \leq N$. Furthermore, we have
$$-1\leq h_{1,+}^{(i)}\leq0 \quad\text{and}\quad -1\leq h_{2,+}^{(i)}<0.$$
The first equality holds if and only $i=N$ and $\expon_N=3.$
The second equality holds if and only if $i=N-1$ and $\expon_N=3$.
The third equality holds if and only $i=N$.
\end{corollary}
Using \eqref{graph-line-bundle-degree},  \eqref{cut-phase}, and \eqref{narrow-singular}, it is easy to check that the corollary above implies
\[(\ell_{1,+}^{(i)}, \ell_{1,-}^{(i)})=
\left\{
\begin{array}{ll}
(0,-1), & \text{if} \ i=N-1, \expon_N=3,\\
(-1,-1-\delta_N^i),& \text{otherwise}.
\end{array}
\right.
\]
and
\[(\ell_{2,+}^{(i)}, \ell_{2,-}^{(i)})=(-1,-1-\delta_N^i).
\]
Then by Lemma \ref{lemma-chiodo}, only $\mathscr{L}_N$ has nonzero contribution to the correlator $X$ and \eqref{correlator_formula} is applicable. A direct computation shows
$$\agen_N^{(N)}=2q_N, \quad H^{(N)}=S_N^{(N)}=1-q_N, \quad \gamma_{1,+}^{(N)}=1-3q_N, \quad \gamma_{2,+}^{(N)}=0.$$
We plug these numbers into \eqref{correlator_formula} and get
\[
X =
\frac{1}{2}\Big(
 2(2q_N)(1-2q_N) + 2(1-q_N)q_N
- q_N(1 -q_N) - (1-3q_N)3q_N- 2(0)(1-0) \Big)
=q_N.
\]

\subsubsection{Loop polynomial $W = x_1^{\expon_1} x_2 + x^{\expon_2} x_3 + \ldots + x_{N-1}^{\expon_{N-1}}x_N + x_N^{\expon_N}x_1$.}
If $\expon_N>2$, we again have concavity.
Recall the notations in \eqref{phase-notation}, we know that Lemma \ref{loop-comb} implies
$H^{(i)}, S_N^{(i)}, \agen_N^{(i)}\in (0,1)$, and 
\begin{itemize}
\item $h_{1,+}^{(N)},  h_{2,-}^{(N)}\in (-1,0)$ and $h_{1,-}^{(N)},  h_{2,+}^{(N)}\in (-2,-1)$ if $\expon_N>3$.
\item $h_{k,+}^{(i)},  h_{k,-}^{(i)}\in [-1,0]$ otherwise. 
\end{itemize}
Moreover, we know $h_{k,+}^{(i)}+h_{k,-}^{(i)}=-1-\delta_N^i$ for all $k=1,2$. 
According to \eqref{graph-line-bundle-degree}, we have
\begin{itemize}
\item The pair $(\ell_{1,+}^{(N)},\ell_{1,-}^{(N)})=(-1,-2)$ and the pair $(\ell_{2,+}^{(N)},\ell_{2,-}^{(N)})=(-2,-1)$ if $\expon_N>3$.
\item The pair $(\ell_{k,+}^{(i)},  \ell_{k,-}^{(i)})\in \{(-1,-1), (-1,0)\}$ otherwise.
\end{itemize}
Thus the correlator $X$ satisfies the condition in Lemma \ref{lemma-chiodo} and we can apply Formula \eqref{correlator_formula} to compute its value. A direct computation shows
$$\agen_N^{(N)}=Y_{N,1}, \quad H^{(N)}=1-Y_{N,0}, \quad S_N^{(N)}=1+Y_{N,-2}, \quad \gamma_{1,+}^{(N)}=1-Y_{N,2}, \quad \gamma_{2,+}^{(N)}=Y_{N,-1}.$$
As a consequence, we have
\begin{align*}
X &= \frac{1}{2}\left(
\begin{array}{c}
2\,Y_{N,1}(1-Y_{N,1})+(1-Y_{N,0})\,Y_{N,0}+(1+Y_{N,-2})(-Y_{N,-2})\\
\\
-Y_{N,0}(1-Y_{N,0})-(1-Y_{N,-2})(Y_{N,-2})-2\,Y_{N,-1}(1-Y_{N,-1})
\end{array}
\right)=q_N.
\end{align*}

If $\expon_N=2$, then $X$ is never concave. We can classify them into three exceptional families. The computation of each family is known. We list the computations in Appendix \ref{sec-exceptional}.

\subsection{B-model computation: a perturbative formula}
In this section we prove Part (3) of Theorem \ref{reconstruction_theorem}.
\begin{proposition}\label{Bside_corr_theorem}
Let $q_i$ be the $i^{th}$ weight of $W$ and $M_i$ be any monomial of a Fermat or loop summand, or the final monomial of a chain summand in $W^T$, then
\begin{equation}\label{b-correlator-computation}
\langle x_i, x_i, M_i/x_i^2, \Hessbase{W^T} \rangle = -q_i,
\end{equation}
\end{proposition}
\begin{proof}

Since the variables are symmetric in loop case, we only need to deal with $i=N$.

Let $f=W^T$. We recall that the perturbative formula \eqref{compute_potential_eq} takes the following form
\begin{equation}\label{perturbative-computation}
e^{(F-f)/z} \zeta(z, \mathbf{s}) =
\left[d^Nx\left(1 + z^{-1} \sum_{\alpha}\Jfunc_{-1}^{\alpha}\bgen_{\alpha} + z^{-2}\sum_{\alpha}\Jfunc_{-2}^{\alpha}\bgen_{\alpha} + \ldots\right)\right]
\in\; \HH_f\llbracket\mathbf{s}\rrbracket.
\end{equation}
By definition \eqref{sg-potential-g=0} and Equation \eqref{compute_potential_eq}, we know
\begin{equation}
\label{eq:b_model_formula}
\langle x_N, x_N, M_N/x_N^2, \Hessbase{W^T} \rangle=\left.\frac{\partial^4 \pot_{0,f,\prim}^{\rm SG}}{\partial^2 t_{x_N}\partial t_S \partial t_H}\right|_{\mathbf{t}=0}=\left.\frac{\partial^3 \Jfunc_{-2}^1(\mathbf{t})}{\partial^2 t_{x_N}\partial t_S}\right|_{\mathbf{t}=0}.
\end{equation}
Here we denote $S:=M_N/x_N^2, H:=\Hessbase{W^T}$ and $t_{\alpha}$ is the flat coordinate dual to $\phi_\alpha$.
Following the notations in Section \ref{section-perturbative-formula}, we use the subscript $(\leq k)$ to denote the $k$-th Taylor expansion in terms of $\mathbf{s}$ (or $\mathbf{t}$). As shown in Proposition 3.12 in \cite{LLSS}, the perturbative formula implies that $(\pot_{0,f,\prim}^{\rm SG})_{(\leq 4)}({\mathbf t})$ depends on
$\zeta_{(\leq 1)}(\mathbf s)$, the primitive form up to first order, only.
The algorithm described in Section \ref{section-perturbative-formula} shows that
$$
\zeta_{(\leq 1)}(\mathbf{s})=[d^Nx].
$$
Therefore we only need to expand the LHS of \eqref{perturbative-computation} using $\zeta_{(\leq 1)}(\mathbf{s})$
$$\left[d^Nx\left(1+{F-f\over z}+{1\over 2!}({F-f\over z})^2+{1\over 3!}({F-f\over z})^3+\OO({\mathbf s}^4)\right)\right]$$
to compute the 4-point function. The term $\Jfunc_{-2}^1$ corresponds to the coefficient in front of $z^{-2}[\phi_1d^Nx]=z^{-2}[d^Nx]$. The correlation function \eqref{eq:b_model_formula} comes from $t_{X_N}^2t_{S}z^{-2}[d^Nx]$.

The contribution from ${1\over 2!}({F-f\over z})^2$ is ${s_1^2\over2}.$ This has no contribution to the RHS of Equation \eqref{eq:b_model_formula}, since the Equation \eqref{first-order} shows
\[
s_{\alpha}= t_{\alpha} + \OO(\mathbf{t}^2).
\]
Thus the correlator $\langle x_N, x_N, M_N/x_N^2, \Hessbase{W^T} \rangle$ is just twice of the coefficient of $s_{x_N}^2s_{S}\,z^{-2}\phi_1$ in ${1\over 3!}({F-f\over z})^3$ (again using $s_{\alpha}= t_{\alpha} + \OO(\mathbf{t}^2).$). On the other hand, since
$$x_N\, x_N\, {M_N\over x_N^2}=M_N,$$
to obtain \eqref{b-correlator-computation}, we only need to prove the following equation: 
\begin{equation}\label{correlator-brieskorn}
[M_Nd^Nx]=-q_N\, z[d^Nx] \in \HH_f.
\end{equation}
For both Fermat polynomial and chain polynomial, Equation \eqref{correlator-brieskorn} is true because
$$\expon_N [M_Nd^Nx]= -z[d^Nx]\in \HH_f.$$
For the loop polynomial, Equation \eqref{correlator-brieskorn} follows from Equation \eqref{loop-weight} and by cancelling $M_1,\cdots, M_{N-1}$ among the relations
$$\expon_i [M_id^Nx]+[M_{i+1}d^Nx]=-z[d^Nx]\in \HH_f.$$
\end{proof}

\section{A proof of Proposition \ref{almost_there_prop}}\label{chap:reconstruction}

In this section, we complete a proof of Proposition \ref{almost_there_prop} for $W = \bigoplus W_i$ when $W_1$ is a chain or a loop.

\subsection{Atomic reconstruction of Chain type}\label{chain_reconstruction_section}
This subsection proves Proposition \ref{almost_there_prop} for $W = \bigoplus W_i$ when $W_1$ is a chain $W_1= x_1^{\expon_1}x_2 + x_2^{\expon_2}x_3+\ldots + x_N^{\expon_N}$.
\subsubsection{Preliminary facts about chain polynomials.}
We will repeatedly use the following relations in $\Jac{W_1^T}$:
\begin{equation}\label{milnor-chain}
\left\{
\begin{array}{l}
\expon_{1} x_1^{\expon_1-1}=-x_2^{\expon_2};\\
\expon_{i} x_{i-1}x_i^{\expon_i-1}=-x_{i+1}^{\expon_{i+1}}, \quad i=2,\cdots,N-1;\\
x_{N-1}x_N^{\expon_N-1}=0.
\end{array}
\right.
\end{equation}
These relations imply
\begin{equation}\label{chain_vanishing_eq}
x_{i-1}x_i^{\expon_i}=0, \quad i<N.
\end{equation}

Additionally, the following lemma tells us what $\mathbf{K}$ looks like in most cases.
\begin{lemma}\label{chain_kis1_lemma}
If $\K_{W_1} = 1$ and $W_1 = x_1^{\expon_1}x_2 + x_2^{\expon_2}x_3+\ldots + x_N^{\expon_N}$ is a chain, then $\K_N \leq \num_N$. If in addition $\K_N \geq 0$, then $\mathbf{\K}$ is one of the following:
\begin{itemize}
\item A concatenation of (0)s and (-1,1)s followed by (1)
\item A concatenation of (0)s and (-1,1)s with one of (-1,2), (-2,3), or (1), followed by (0).
\end{itemize}
\end{lemma}
\begin{proof}
We know that $\m_N + \n_N = (\num_N - \K_N + 1) - \num_N -2$. Combining this with $\m_N + \n_N \geq 0$, we know $\K_N \leq \num_N$ because
\[
\K_N \leq \left(\num_N + \frac{\expon_N-2}{\expon_N-1}\right)\left(\frac{\expon_N-1}{\expon_N}\right)<\num_N+1.
\]

If $\K_N \geq 0$, then $\{1, \ldots, N-1\}$ obeys the NP-rule.
Otherwise we must have $\K_{N-1} < 0$. However, because $\m_{N-1} + \n_{N-1} \leq 2 \expon_{N-1}-2$, equation (\ref{mplusn_eq}) shows that
\[
(\expon_{N-1}-1)\num_{N-1} - \expon_{N-1} \K_{N-1} + \expon_{N-1} + (\num_N - \K_N) - 1 \leq 2 \expon_{N-1} -2.
\]
Because $\num_{N-1} \geq 0$, and $\K_{N-1} \leq -1$, and $\num_N - \K_N \geq 0$, the left hand side of this inequality must be at least $2 \expon_{N-1}-1$. This is a contradiction.

Thus $\{1, \ldots, N-1\}$ obeys the NP-rule and Lemma \ref{loop_condition_lemma} implies $\sum_{i < N} \K_i\geq0$. Then we have $\{\sum_{i < N} \K_i, \K_N\} = \{0,1\}$. Thus by Lemma \ref{loop_condition_lemma}, we obtain the first case if $\K_N=1$ and we obtain the second case if $\K_N=0$.
\end{proof}

We introduce the following definition.
\begin{definition}
Let $X$ and $X'$ be correlators of type $\X{C}{0}$ with the same number of insertions. Assume $\sum^N_{i=1} \num_i =   \sum^N_{i=1} \num_i'$, where $1, \ldots, N$ are the indices in $W_1$.
We say that $X > X'$ if $\num_N = \num_N', \ldots, \num_{r+1} = \num_{r+1}'$, and $\num_{r} > \num_{r}'$, for some $r\in W_1$.
We say that $X$ is maximal if there does not exist $X' > X$, or equivalently, if 
$$X =\langle x_{i_1} \not \in W_1, \; \ldots, \; x_{i_s} \not \in W_1, x_N, \; \ldots, \; x_N, \; \alpha, \;\beta \rangle.$$
\end{definition}
The relation $>$ is well-defined because of the ordering of primitive insertions in correlators of type $\X{X}{-1}$ (see \eqref{corr_form_eq}).
Also, this relation is transitive.
We immediately have
\begin{lemma}\label{chain_reconstruction_lemma}
Let $X$ be a correlator of type $\X{C}{0}$. If there is $i\in W_1$ such that we can rewrite $X$ as $X = \Fiveptcorr{\ldots}{x_i}{\ldots}{x_{i+1}^{\expon_{i+1}} \alpha}{ \beta }$ with $i+1\in W_1$, then $X$ can be reconstructed from correlators with fewer insertions and correlators $Z$ of type $\X{C}{0}$ satisfying $Z>X$.
\end{lemma}
\begin{proof}
Apply the Reconstruction Lemma \ref{reconstruction_lemma} to $X$ with $\gamma = x_i$, $\delta = \beta$, $\epsilon = x_{i+1}^{\expon_{i+1}}$, and $\phi = \alpha$. Then $\gamma  \epsilon=x_ix_{i+1}^{\expon_{i+1}}$ vanishes by \eqref{chain_vanishing_eq} and the other two correlators have the form $X'=\langle \ldots, x_{i+1}^{\expon_{i+1}}, \ldots, \alpha', \beta' \rangle$.
Apply the Reconstruction Lemma \ref{reconstruction_lemma}  to $X'$ with $\gamma = \alpha'$, $\delta = \beta'$, $\epsilon = x_{i+1}$, and $\phi = x_{i+1}^{\expon_{i+1}-1}$. The correlators with $\phi \delta$ and $\gamma \delta$ have the form $\Fiveptcorr{\ldots}{x_{i+1}}{\ldots}{*}{*}$; the remaining correlator looks like $X'' = \langle\ldots, x_{i+1}^{\expon_{i+1}-1}, \ldots, \alpha'', \beta''\rangle$.

Perform a similar reconstruction on $X''$, this time with $\phi = x_{i+1}^{\expon_{i+1}-2}$. By repeating the process, we can reduce the exponent of $x_{i+1}$ and eventually, we will have determined $X$ from correlators with fewer insertions and correlators of the form $Y=\Fiveptcorr{\ldots}{x_{i+1}}{\ldots}{*}{*}$ with $\num_i^Y = \num_i^X-1$ and $\num_{i+1}^Y = \num_{i+1}^X+1$.
After reducing insertions to the standard basis, all the nonvanishing correlators we get from this process are of type $\X{X}{-1}$. Furthermore, if each such $Y$ is of type $\X{C}{0}$, then the result follows since $Y>X$. Thus we only need to reconstruct those correlators that are not of type $\X{C}{0}$. 
Then we must have $\K_{W_1}=0$ for such a correlator $Y$.

On the other hand, since $X$ is of type $\X{C}{0}$, besides $x_i$, $X$ must have at least one more insertion $x_k$, with $x_k\in W_1$. Since the Reconstruction Lemma \ref{reconstruction_lemma} does not change the insertions in dotted positions of $X = \Fiveptcorr{\ldots}{x_i}{\ldots}{x_{i+1}^{\expon_{i+1}} \alpha}{ \beta }$, we know $Y$ could be rewritten in the form of $Y=\Fiveptcorr{x_{i+1}}{x_k}{\ldots}{*}{*}.$ Since $\K_{W_1}=0$, we may assume $\K_{N_1}\geq0$, otherwise we do a preparatory reconstruction as \eqref{stwitch-nonnegative} to get $\K_{N_1}\geq0$. Thus by Lemma \ref{kis0_lemma}, we know $i+1\neq k$ and $\num_{i+1}=\num_{k}=1$. Then we can repeat the process as in Proposition \ref{two_from_a_poly} to reconstruct $Y$ from type $\X{C}{0}$ correlators $Z$ such that $Z>X$, and correlators with fewer insertions. Such correlators $Z$ will be of the form 
$Z=\Fiveptcorr{x_{i+1}}{x_{i+1}}{\ldots}{\alpha_Z}{\beta_Z}$ if $k\leq i$, or 
$Z=\Fiveptcorr{x_{k}}{x_{k}}{\ldots}{\alpha_Z}{\beta_Z}$ if $k>i$.
We remark that during the process, the ordered pair of inserions $(x_i, x_k)$ in \eqref{my_fav_eq} are replaced by the ordered pair $(x_{i+1}, x_k)$ if $k\leq i$, or by $(x_k, x_{i+1})$ if $k>i$. This guarantees that we have $Z>X$.
\end{proof}
By the above lemma, we have
\begin{proposition}\label{reconstruct_from_greater_prop}
Let $X $ be a correlator of type $\X{C}{0}$ which is not maximal. Then $X$ can be reconstructed from correlators with fewer insertions and correlators $Z$ of type $\X{C}{0}$ satisfying $Z>X$.
\end{proposition}
\begin{proof}
Since $X $ is not maximal, we can choose $i$ to be the largest index such that $i \in W_{1}$ with $i < N$ and $\num_i \geq 1$. So $X = \Fourptcorr{ \ldots}{ x_i}{ x_N^{\m_N}  \alpha_X}{ \beta_X }$ for some $\m_N\geq0$.

If $\m_N\geq1,$ then we apply Reconstruction Lemma \ref{reconstruction_lemma} with $\gamma = \beta_X$, $\delta = x_i$, $\epsilon = x_N$, and $\phi = x_N^{\m_N-1} \alpha_X.$
By Remark \ref{rmk:reconstruction_invariance}, the correlators with $\delta \phi$ and $\delta \gamma$ are type $\X{C}{0}$-correlators of the form $\Fourptcorr{ \ldots}{ x_N}{*}{*}$.
The correlator with $\epsilon  \gamma$ equals $\langle\ldots, x_i, x_N^{\m_N-1} \alpha_X, x_N  \beta_X \rangle$.
By induction reconstruct $X$ from $\X{C}{0}$-correlators of the form $\Fourptcorr{ \ldots}{ x_N}{*}{*}$ and the $\X{C}{0}$-correlator $Y = \langle  \ldots, x_i, \alpha_Y, \beta_Y \rangle$ where $\m^Y_N=0$.

Similarly, we move all $x_{N-1}$ from $\alpha_Y$ to $\beta_Y$, and so on, until we move all $x_{i+1}$ from $\alpha$ to $\beta$.
Thus we reconstruct $X$ from correlators $X^{\prime}$ of type $\X{C}{0}$ with $X^{\prime}>X$, and the correlator $Z = \langle\ldots, x_i, \alpha_Z, \beta_Z \rangle$ where $\m^Z_{i+1} = \ldots = \m^Z_N = 0$.
After reducing to the standard basis, $Z$ is of type $\X{X}{0}$ and $\m^Z_{k} + \n^Z_{k} \leq \expon_k-1$ for $k> i$.

From here on we will speak only of the correlator $Z$ and drop the $Z$-superscript from our notation.
By definition,
\begin{equation}\label{chain-last}
\m_N + \n_N = (\num_N - \K_N + 1) \expon_N - \num_N - 2.
\end{equation}
But $\m_N + \n_N \leq \expon_N-1$, so $(\num_N - \K_N + 1) \expon_N - \num_N - 2 \leq \expon_N-1$. This shows
\begin{equation}\label{chain_kn_eq1}
\K_N \geq \left(\num_N - \frac{1}{\expon_N-1}\right)\left(\frac{\expon_N-1}{\expon_N}\right)>-{1\over \expon_N}.
\end{equation}
Thus $\K_N \geq 0$ and we may use Lemma \ref{chain_kis1_lemma}.
This lemma gives us a list of possible vectors ${\bf \K}$ which we analyze case by case. 
In each case, if the correlator is not in the desired form, we write the insertions in a nonstandard basis and so that there is some $k$ with $\num_k\geq1$ and $\m_k + \n_k \geq \expon_k$. Then we use Lemma \ref{chain_reconstruction_lemma} to finish the reconstruction.\\

\noindent{\bf{Case $\K_N=1$:}}
In this case $\mathbf{\K}$ is a concatenation of (0)s and (-1,1)s, followed by $\K_N=1$.

If $\mathbf{\K} = (\ldots, -1,1,1)$, then $\mathbf{\num} =( \ldots, 0, 0, *)$ and $\mathbf{\m + \n} = (\ldots, 2 \expon_{N-2}-2, 0, *)$ by Lemma \ref{klmn_lemma}. Then $N-2>i$, but $\m_{N-2} + \n_{N-2} \geq \expon_{N-2}$, contradicting our assumption on $Z$.
Similarly, we reach a contradiction if there is $j > i$ such that $(\K_j, \K_{j+1}) = (-1,1)$.

Therefore, $\mathbf{\K} = ( \ldots, \underline{0},0, \ldots, 0,1)$ and $\mathbf{\num} = (\ldots, \underline{1}, 0, \ldots,0,*)$, where the underline marks the $i^{th}$ spot and $\num_i=1$ by \eqref{ell_eq}.
Possibly, $i=N-1$.
If $i \neq N-1$, then by assumption $(\K_{i+1}, \num_{i+1}) = (0,0)$ so $\m_i+\n_i = 2 \expon_i-2$ by \eqref{mplusn_eq}.
If $i=N-1$, then $(\K_{i+1}, \num_{i+1}) = (\K_N,\num_N)$ where $\num_N \geq \K_N$. Then (\ref{mplusn_eq}) shows $\m_i + \n_i \geq 2 \expon_i-2$ so by Property (P3) we know $\m_i + \n_i = 2 \expon_i-2$.
Thus $\mathbf{\m + \n} = ( \ldots, \underline{2 \expon_i-2}, *, \ldots, *, *).$ Now we have three cases. In each case we compute $\mathbf{\m+\n}$ by first using \eqref{ell_eq} to compute $\ell$ and then using (\ref{mplusn_eq}).

\begin{enumerate}
\item
$\mathbf{\K}=(  \ldots, 0, \underline{0}, \ldots, 1),$
$\mathbf{\m+\n}=(\ldots, \expon_{i-1}, \underline{M}, \ldots, *).$
\item
$\mathbf{\K}=(  -1, 1, \ldots, -1, 1, \underline{0}, \ldots, 1),$
$\mathbf{\m+\n}=(M, 0, \ldots, M, 0, \underline{M}, \ldots, *).$
\item
$\mathbf{\K}=( \ldots, 0,-1, 1, \ldots, -1, 1, \underline{0}, \ldots, 1),$
$\mathbf{\m+\n}=( \ldots, \expon_r, M, 0, \ldots, M, 0, \underline{M}, \ldots,* ).$
\end{enumerate}
Here  $M = 2 \expon-2$ with the appropriate subscripts.
In each case if there is a factor of $\alpha$ that equals $x_{i+1}^{\expon_{i+1}}$ ($\alpha$ will not be written in the standard basis), then we can apply Lemma \ref{chain_reconstruction_lemma}. We find a factor of $x_{i+1}^{\expon_{i+1}}$ in $\alpha$ for each case as follows:
\begin{enumerate}
\item Here $\alpha$ has a factor of $x_{i-1} x_{i}^{\expon_i-1}$, which by \eqref{milnor-chain} equals $x_{i+1}^{\expon_{i+1}}$ in $\Jac{W_1^T}$.
\item Repeatedly apply \eqref{milnor-chain} starting with $\expon_1 x_1^{\expon_1-1} = - x_2^{\expon_2}$.
\item Repeatedly apply \eqref{milnor-chain} starting with $\expon_{r+1} x_r x_{r+1}^{\expon_{r+1}-1} = - x_{r+2}^{\expon_{r+2}}$.
\end{enumerate}

%
%

\noindent{\bf{Case $\K_N=0$:}}
In this case, (\ref{chain_kn_eq1}) shows that $\num_N = 0$. Let $i$ be the largest index such that $(\K_i, \num_i) \neq (0,0)$.
 Since $0 \leq \m_i + \n_i$, equation (\ref{mplusn_eq}) shows $\K_i \leq \num_i$.
 Then Lemma \ref{klmn_lemma} shows that $(\K_{i-1}, \K_i)$ cannot be $(-2,3)$, $(-1,2)$, or $(-1,1)$.
 Six cases remain, which are enumerated below.
 In each case an underline is below the $i^{th}$ index, $M = 2 \expon - 2$, and we computed $\mathbf{\num}$ and $\mathbf{\m+\n}$ using (\ref{ell_eq}) and (\ref{mplusn_eq}), respectively.
\begin{itemize}
\item $\mathbf{\K}=( \ldots, -1,  1,  \underline{1},  0,  \ldots,  0) $
$\mathbf{\num}=( \ldots,  0, 0,  \underline{2}, 0,  \ldots,  0) $
$\mathbf{\m+\n}=( \ldots,  M,  0, \underline{M-1},  *,  \ldots,  *) $

\item
$\mathbf{\K}=( \ldots, 0, \underline{1}, 0, \ldots, 0)$,
$\mathbf{\num}=(\ldots, 0, \underline{2}, 0, \ldots, 0)$,
$\mathbf{\m+\n}=(\ldots, \expon_{i-1}, \underline{M-1}, *, \ldots, *)$;

\item
$\mathbf{\K}=(\ldots, 0, \underline{1}, 0, \ldots, 0)$,
$\mathbf{\num}=(\ldots, 1, \underline{1}, 0, \ldots, 0)$,
$\mathbf{\m+\n}=(\ldots, M, \underline{\expon_i-2}, *, \ldots, *)$;

\item
$\mathbf{\K}=(\ldots, 0, \underline{0}, 0, \ldots, 0)$,
$\mathbf{\num}=(\ldots, 0, \underline{1}, 0, \ldots, 0)$,
$\mathbf{\m+\n}=(\ldots, \expon_{i-1}, \underline{M}, * , \ldots, *)$;

\item
$\mathbf{\K}=(\ldots, -1, 1, \underline{0}, 0, \ldots, 0)$,
$\mathbf{\num}=(\ldots, 0, 0, \underline{1}, 0, \ldots, 0)$,
$\mathbf{\m+\n}=(\ldots, M, 0, \underline{M}, *, \ldots, *)$;

\item
$\mathbf{\K}=( \ldots, 1, \underline{0}, 0, \ldots, 0)$,
$\mathbf{\num}=(\ldots, \num_{i-1}, \underline{1}, 0, \ldots, 0)$,
$\mathbf{\m+\n}=( \ldots, *, \underline{M}, *, \ldots, *)$.
\end{itemize}
Notice that in each case, $i$ is the same as the $i$ we had defined previously. Now the discussion is similar to the $\K_N=1$ case. More explicitly, we have the following situations:
\begin{itemize}
\item The first case here splits into Cases (2) and (3) above.

\item The second and fourth case here are same as Case (1) when $\K_N=1$.

\item
For the third case, $(\K_{i-1}, \K_i)=(0,1)$. As we did for $(\K_{N-1}, \K_N)=(0,1)$, we find a factor of $\alpha$ equal to $x_{i}^{\expon_i}$ (note the different index). All three cases above are possible. Since $\num_{i-1}=1$, we can apply Lemma \ref{chain_reconstruction_lemma}.

\item
For the fifth case, let $r$ be the last index before the $(-1,1)$ pairs of $\mathbf{\K}$.
If $\K_r=0$, then we can find a factor of $\alpha$ equal to $x_{i+1}^{\expon_{i+1}}$ (like in Cases (2)-(3) of $(\K_{N-1}, \K_N)=(0,1)$).
If $\K_r=1$ (but not as part of a (-1,1) pair), then $\m_r + \n_r = \num_r(\expon_r-1)$ is 0 only if $\num_r=0$.

Similarly, for the last case, let $r = i-1$. Then $\m_r + \n_r = \num_r(\expon_r-1)$, which will be 0 only if $\num_r=0$.

If $\num_r > 0,$ then $\m_r + \n_r>0$ and as usual we can find that $\alpha$ has a factor $x_{i+1}^{\expon_{i+1}}$.

Thus we only need to check for both of these cases when $\num_r=0$. Since $X$ was of type $\X{X}{0}$ and the reconstruction from $X$ to $Z$ did not change the number of insertions in $W_1$, there is some  $k < r < i$ such that $\num_k \geq 1$. But in the indices less than $r$, the vector $\mathbf{\K}$ is a concatenation of (-1,1)s and (0)s. This means that if we truncate the vector $\mathbf{\K}$ before the $k$-th place, then the truncated vectors will look like the fourth case or the fifth case with $\K_r=0$. So we can use \eqref{milnor-chain} to find a factor of $\alpha$ equal to $x_{k+1}^{\expon_{k+1}}$.

\end{itemize}
\end{proof}

\subsubsection{Reconstruction procedure}
The remainder of this subsection proves Proposition \ref{almost_there_prop} when $W_1 = x_1^{\expon_1}x_2 + x_2^{\expon_2}x_3+\ldots + x_N^{\expon_N}$ is a chain.
\begin{definition}
Let $X$ be a correlator of type $\X{C}{0}$. We say
\begin{itemize}
\item
$X$ is of type $\X{C}{1}$ if
$
X= \langle x_{i_1} \not \in W_1, \; \ldots, \; x_{i_s} \not \in W_1, x_N, \; \ldots, \; x_N, \; \alpha, \;\beta \rangle.
$

\item
 $X$ is of type $\X{C}{2}$ if
$
X = \Fiveptcorr{ x_N}{ \ldots}{ x_N}{ \alpha}{ \beta}.
$

\item
 $X$ is of type $\X{C}{3}$ if
$
X = \Fourptcorr{ x_N}{ x_N}{ \alpha}{ \beta}.
$

\item
 $X$ is of type $\X{C}{4}$ if
$
X = \Fourptcorr{ x_N}{ x_N}{ x_{N-1}x_N^{\expon_N-2}}{ \Hessbase{W^T}  }
$.
\end{itemize}
\end{definition}
We will do the reconstruction in four steps. In $k$-th step, we will reconstruct a correlator of type $\X{C}{k-1}$ from correlators of type $\X{C}{\geq k}$, correlators that vanish, and correlators with fewer insertions. \\

\noindent
{\bf Step 1.} We do this by applying Proposition \ref{reconstruct_from_greater_prop}.\\

\noindent
{\bf Step 2.}
Let $X$ be a correlator of type $\X{C}{1}$. Our discussion breaks into three cases:
\begin{enumerate}
\item $\m_N\n_N\neq0$.
\item $\m_N\n_N=0$ and $(\num_N,\expon_N)\neq(2,2)$.
\item $\m_N\n_N=0$ and $(\num_N,\expon_N)=(2,2)$.
\end{enumerate}

For case (1), we will reconstruct $X$ from correlators with fewer insertions and correlators of type $\X{C}{2}$. We may assume $\m_N \neq 0$.
Let $S$ be the set of insertions $x_i$ in $X$ such that $x_i \not \in W_1$. If $S\neq\emptyset$, choose some $x_i \in S$ and use the Reconstruction Lemma \ref{reconstruction_lemma} with $\gamma = \beta$, $\delta = x_i$, $\epsilon = x_N$, and $\phi = \alpha/x_N.$
Then all correlators coming from the reconstruction have $\num_{N} \geq \num_N^X \geq 2$, so $\K_N \geq 0$ (we saw in the proof of Lemma \ref{splitting_lemma} that when $\K_N<0$, we have $\num_N-0$). Then Remark \ref{rmk:reconstruction_invariance}(3) implies that all the correlators have type $\X{C}{0}$ or they vanish.
Moreover, all these correlators have $\m_N = \m_N^X-1$.
Repeat this same reconstruction until either $\m_N=0$ in all correlators or $S = \emptyset$.

In cases (2) and (3) we must have $\m_N+\n_N \leq \expon_N-1$ so (\ref{chain_kn_eq1}) holds, and $\mathbf{\K}$ is in the form of Lemma \ref{chain_kis1_lemma}.
So $\K_N$ is 0 or 1; if $\K_N=0$ then (\ref{chain_kn_eq1}) shows that $\num_N=0$, and if $\K_N=1$ then (\ref{chain_kn_eq1}) shows that $\num_N$ is 1, 2 or 3.
But since $X$ is of type $\X{C}{1}$ we know $\num_N \geq 2$, so in fact
$\K_N=1$ and $(\num_N,\expon_N)=(2,2), (2,3)$ or $(3,2)$.

If $(\num_N,\expon_N)\neq(2,2)$, then by \eqref{chain-last}, we get
$$\m_N + \n_N = \num_N(\expon_N-1)-2 = \expon_N-1.$$
We have the following cases, where again $M = 2 \expon-2$:
\begin{enumerate}
\item
$\mathbf{\K}=( -1, 1, \ldots, -1, 1, 1),$
$\mathbf{\m+\n}=(M, 0, \ldots, M, 0, \expon_N-1);$
\item
$\mathbf{\K}=( *, \ldots, 0, 1),$
$\mathbf{\m+\n}=(*, \ldots, \expon_{N-1}, \expon_N-1);$
\item
$\mathbf{\K}=(*, \ldots, 0, -1, 1, \ldots, -1, 1, 1),$
$\mathbf{\m+\n}=(*, \ldots, \expon_r, M, 0, \ldots, M, 0, \expon_N-1).$
\end{enumerate}
Without loss of generality assume $\m_N=0$ so $\n_N = \expon_N-1$. Then in every case $X$ cannot have Property (P3), so it vanishes. 

If $(\num_N,\expon_N)=(2,2)$, then $X= \langle x_{i_1} \not \in W_1, \; \ldots, \; x_{i_s} \not \in W_1, x_N, \; x_N, \; \alpha, \;\beta \rangle.$
Since $\K_N=1$ and $\mathbf{\num} = (0, \ldots, 0, 2)$, a calculation using \eqref{mplusn_eq} and Lemma \ref{chain_kis1_lemma} shows that there is some $j \in W_1$ with $\m_j+\n_j \geq \expon_j$. In particular $\m_j>0$.
Reconstruct $X$ as in case (1), beginning with $\delta = x_{i_1}$, $\gamma=\beta$, $\epsilon=x_j$, and $\phi=\alpha/x_j$.
All resulting correlators have type $\X{C}{1}$ as discussed in case (1).
By repeating this reconstruction, we can determine $X$ from correlators with fewer insertions, type $\X{C}{1}$ correlators of the form $\langle x_N,\; x_N, \; x_j, \; \ldots \rangle$, and a type $\X{C}{1}$ correlator with the same primitive insertions as $X$ but with $\m_j=0$ for all $j \in W_1$.
The correlators $\langle x_N,\; x_N, \; x_j, \; \ldots \rangle$ can be reconstructed from correlators $\langle x_N,\; x_N, \; x_N, \; \ldots \rangle$ as in Proposition \ref{reconstruct_from_greater_prop}.
The final correlator vanishes by the discussion at the start of this paragraph.\\

\noindent
{\bf Step 3.}
Let $X$ be a correlator of type $\X{C}{2}$.
From (\ref{eq1}), since $\num_N \geq 2$, we find
\[
\K_N \geq {\expon_N-2\over\expon_N}.
\]
Thus $\K_N \geq 0$ and equality is possible only if $\expon_N=2$. If $\K_N=0$ and $\expon_N=2$, then \eqref{eq1} shows that $\num_N\leq 2$, so in fact $X$ is of type $\X{C}{3}$.

If $\K_N \neq 0$, then $\K_N=1$ by Lemma \ref{chain_kis1_lemma}.
Then (\ref{eq1}) shows $\num_N \leq 2\expon_N/(\expon_N-1)$, so $\num_N=2$ (and $X$ is of type $\X{C}{3}$), or $\num_N=\expon_N=3$, or $\expon_N=2$ and $\num_N=3$ or $4$.
We will show that in each case where $\num_N>2$, the correlator does not satisfy (P3) in Lemma \ref{properties_lemma}, a contradiction.

If $\num_N=\expon_N=3$, then $\m_N + \n_N = 3 \expon_N - 5 = 2 \expon_N-2$.
Either $(\K_{N-1}, \num_{N-1})$ is $(0,0)$ or it is $(1,0)$; in each case, $\m_{N-1} + \n_{N-1} \geq 1$.
Without loss of generality $\m_{N-1} \geq 1$, so that $\alpha$ has a factor of $x_{N-1} x_N^{\expon_N-1}$, violating Property (P3).

Similarly, if $\expon_N=2$ and $\num_N=3$ or 4, we can check all possibilities for $\mathbf{\K}$ and $\mathbf{\num}$ and show that $\mathbf{\m+\n}$ violates Property (P3).\qed\\

\noindent
{\bf Step 4.}
Let $X$ be a correlator of type $\X{C}{3}$.
We know $\mathbf{\num} = (0, \ldots, 0, 2)$. By \eqref{klmn_lemma-5}, if $M = 2\expon-2$ we have three possibilities for $\mathbf{\K}$:
\begin{enumerate}
\item
$\mathbf{\K} = (0, \ldots, 0, 0, 1),$ $\mathbf{\m+\n}=(\expon_1-1, \ldots, \expon_{N-2}-1, \expon_{N-1}, 2\expon_N-4)$.

\item
$\mathbf{\K}=(-1,1, \ldots, -1,1, 1),$ $\mathbf{\m+\n} = (M, 0, \ldots, M, 0, 2\expon_N-4)$.

\item
$\mathbf{\K}=(0, \ldots, 0, 0, -1, 1, \ldots, -1, 1, 1)$, $\mathbf{\m+\n} = (\expon_1-1, \ldots, \expon_{r-1}-1, \expon_r, M, 0, \ldots,  M,$ $0, 2\expon_N-4)$.
\end{enumerate}

In all cases, if $X\neq0$, we must have
\begin{equation}\label{chain-middle}
X = \Fourptcorr{ x_N}{ x_N}{ x_1^{\m_1} \ldots x_{N-1}^{\m_{N-1}}x_N^{\expon_N-2}  \alpha}{ x_1^{\n_1} \ldots x_{N-1}^{\n_{N-1}} x_N^{\expon_N-2}  \beta }
\end{equation}
where $\m_i + \n_i = \expon_i-1$ for $i<N-2$ and $\m_{N-1} + \n_{N-1} = \expon_{N-1}$.

In the first case, both $\m_{N-1}$ and $\n_{N-1}$ are at least 1. If $\m_{N} = \expon_N-1$, then $\alpha=0$ by \eqref{milnor-chain}, since it has a factor of $x_{N-1}x_N^{\expon_N-1}$. This shows that $\m_N = \n_N = \expon_N-2$ and \eqref{chain-middle} follows.

In the second case, $\alpha = x_1^{\expon_1-1} x_3^{\expon_3-1} \ldots x_{N-2}^{\expon_{N-2}-1} x_N^{\m_N}$.
The relations \eqref{milnor-chain} show 
$$\alpha \propto x_2^{\expon_2-1}x_4^{\expon_4-1} \ldots x_{N-1}^{\expon_{N-1}} x_N^{\m_N}.$$
If $\m_N = \expon_N-1$, we have a factor of $\alpha$ equal to $x_{N-1} x_N^{\expon_N-1}$, and $\alpha=0$ by \eqref{milnor-chain}.
Otherwise, $\m_N = \n_N = \expon_N-2$ and \eqref{chain-middle} follows.

In the last case,  $\alpha$ has a factor equal to $x_{r}^{\m_r} x_{r+1}^{\expon_{r+1}-1} \ldots x_{N-2}^{\expon_{N-2}-1} x_N^{\m_N}$. As before we can use the relations \eqref{milnor-chain} to rewrite this as $x_r^{\m_r-1}x_{r+2}^{\expon_{r+2}-1} \ldots x_{N-1}^{\expon_{N-1}} x_N^{\m_N}$, and as before if $X\neq0$, then \eqref{chain-middle} follows.

Finally, we apply the Reconstruction Lemma \ref{reconstruction_lemma} to $X$ in \eqref{chain-middle} with $\gamma = x_N$, $\epsilon = x_{N-1}x_N^{\expon_N-2}$, $\phi = x_1^{\m_1} \ldots x_{N-1}^{\m_{N-1}-1}  \alpha$, and $\delta =  x_1^{\n_1} \ldots x_{N-1}^{\n_{N-1}} x_N^{\expon_N-2}  \beta$. Then $\epsilon  \gamma $ and $\gamma  \delta$ have a factor of $x_{N-1} x_N^{\expon_N-1}$, and hence both are 0 by \eqref{milnor-chain}.
The last correlator is of type $\X{C}{4}$.\qed

\subsection{Atomic reconstruction of Loop type}\label{loop_reconstruction_section}
This section proves Proposition \ref{almost_there_prop} for $W = \bigoplus W_i$ when $W_1$ is a loop $W_1=x_1^{\expon_1}x_2+x_2^{\expon_2}x_3+ \ldots + x_N^{\expon_N}x_1$.

\subsubsection{Preliminary facts about loop polynomials} First, we have
\begin{equation}\label{milnor-loop}
\expon_i x_{i-1} x_i^{\expon_i-1} = -x_{i+1}^{\expon_{i+1}} \in \Jac{W_1^T} \qquad \text{for}\; i=1,\cdots, N.
\end{equation}
Recalling that $\phi_{W_1^T}=\prod_{i=1}^{N}x_{i}^{\expon_i-1}$, we obtain the following vanishing conditions in $\Jac{W_1^T}$:
\begin{equation}\label{loop_vanishing_lemma}
\left\{
\begin{array}{ll}
x_{i-1}  x_i^{\expon_i} = 0, &\\
x_1 ^{\expon_1-1} x_3^{\expon_3-1} \ldots x_{N-2}^{\expon_{N-2}-1} x_N^{\expon_N} = 0, & 2\nmid N, \\
(\phi_{W_1^T}/x_i)x_k=0, & k \neq i, (N, \expon_i, k) \neq (2,2, i+1).
\end{array}
\right.
\end{equation}

Second, we may apply Lemmas \ref{loop_condition_lemma} and \ref{klmn_lemma} and renumber the variable indices so that $\mathbf{\K}$ ends with the exceptional tuple (1), (-1,2), or (-2, 3). Then we get
\begin{lemma}\label{loop_kis1_lemma}
If $\K_{W_1} = 1$ and $W_1=x_1^{\expon_1}+x_1x_2^{\expon_2}+ \ldots + x_N^{\expon_N}x_1$ is a loop, then we can number the indices so that $\mathbf{\K}$ is some concatentation of (0)s and (-1,1)s, followed by one of (1), (-1,2), or (-2,3). The final case can only occur if $\expon_{N-1}=2$. Also, in the last two cases, $\num_{N-1}+\num_N \leq 1$.
\end{lemma}

\subsubsection{Reconstruction procedure}
As in the previous sections, we define progressively simpler correlator types and perform the reconstruction in a number of steps.

%

\begin{definition}
Let $X$ be a correlator of type $\X{L}{0}$.
\begin{itemize}
\item
$X$ is  of type $\X{L}{1}$ if
$
X = \langle x_p, x_p, \ldots, \alpha, \beta \rangle
$
for some $p \in W_1$.

\item
$X$ is of type $\X{L}{2}$ if
$
X = \Fiveptcorr{ x_p}{ \ldots}{ x_p}{ \alpha}{ \beta}
$
for some $p \in W_1$.

\item
$X$ is of type $\X{L}{3}$ if
$
 X = \Fourptcorr{ x_p}{ x_p}{ \alpha}{ \beta}
$
for some $p \in W_1$.

\item
$X$ is of type $\X{L}{4}$ if
$
X = \Fourptcorr{ x_p}{ x_p}{x_{p-1} x_p^{\expon_p-2}}{\Hessbase{W^T}}
$
with $p \in W_1$.
\end{itemize}
\end{definition}

In $k$-th step, we will reconstruct a correlator of type $\X{L}{k-1}$ from correlators of type $\X{L}{\geq k}$ and correlators with fewer insertions.
Each type has $\K_{W_1}=1$ by Lemma \ref{kis0_lemma}(3), so we will generally not need to check this condition.\\

\noindent
{\bf Step 1.}
Let $X$ be a correlator of type $\X{L}{0}$.
If there exists some $\num_i\in W_1$ such that $\num_i \geq 2$, then we are done. If $\num_i \leq 1$ for all $i \in W_1$, then we will show that for some $i\in W_1$, we have $\num_i = 1$ and $\m_i + \n_i \geq \expon_i$.

By Proposition \ref{two_from_a_poly} we can assume there exist distinct $i, k \in W_1$ such that $\num_i=\num_k=1$. Now assume $\K_{W_1}$ is in the form of Lemma \ref{loop_kis1_lemma}. If some $\num_i = 1$ for $i \leq N-2$, then Lemma \ref{bound_ell_lemma} shows $\m_i + \n_i = 2 \expon_i - 2\geq\expon_i$.

Otherwise, $\num_{N-1} = \num_N = 1$, and by Lemma \ref{loop_kis1_lemma} we must have $\K_N=1$. In this case $\K_{N-1}$ is 1 or 0. But if $\K_{N-1}=1$ then $\K_{N-2}=-1$ and $\num_{N-1}=0$ by \eqref{klmn_lemma-1}. So $\K_{N-1}=0$ and (\ref{mplusn_eq}) shows $\m_{N-1}+\n_{N-1} = 2\expon_{N-1}-2\geq\expon_{N-1}$.

Now we do the reconstruction part of this step. Let $k\in W_1$ be any index such that $\num_k=1$ and $\m_k + \n_k \geq \expon_k$. Then $X=\Fiveptcorr{ \ldots}{x_k}{ x_i}{ x_k  \alpha}{ \beta}$ where $\m_k\geq1$ and $i\neq k$ (we do not require $i \in W_1$). Using the Reconstruction Lemma \ref{reconstruction_lemma} with $\gamma = \beta$, $\delta = x_i$, $\epsilon = x_k$, and $\phi = \alpha$, we find
\[
X=\Fiveptcorr{ \ldots}{ x_k}{ x_k}{ \beta}{ x_i  \alpha }-\Fiveptcorr{ \ldots}{ x_k}{ x_k}{ \alpha}{ x_i  \beta }-\Fiveptcorr{ \ldots}{ x_k}{ x_i}{ \alpha}{ x_k  \beta}+S.
\]
After rewriting the last insertion in the standard basis, the first two correlators are of type $\X{L}{1}$. We use the same reconstruction on the third correlator until it has the form $\Fiveptcorr{ \ldots}{x_k}{ x_i}{ \widetilde{\alpha}}{ x_k^{\m_k}  \beta}$, where $\widetilde{\alpha}$ is a monomial in the standard basis with no factor of $x_k$. When we rewrite the last insertion in the standard basis, we find that $\m_k + \n_k < \expon_k$. This contradicts Lemma \ref{kis0_lemma} if $\K_{W_1}=0$, so $\K_{W_1}=1$.
Then either the third correlator is of type $\X{L}{1}$, or it has some $k'$ such that $\num_{k'}=1$ and $\m_{k'}+\n_{k'} \geq \expon_{k'}$. We can repeat the same reconstruction, moving all $x_k'$ from $\alpha$ to $\beta$. Eventually, we will run out of such indices, showing that the third form eventually vanishes.\\

\noindent
{\bf Step 2.}
Let $X$ be a correlator of type $\X{L}{1}$.
This step is similar to Step 1. Since $\num_p \geq 2$, we know that $p$ is $N$ or $N-1$ by Lemma \ref{bound_ell_lemma}. Then $\K_N=1$ by Lemma \ref{loop_kis1_lemma}. We will show that $\m_p + \n_p \geq \expon_p$. Then the same reconstruction argument as in Step 1 gives the desired result.
\begin{itemize}
\item
If $p = N-1$, then \eqref{klmn_lemma-1} implies $\K_{N-1} = 0$ and \eqref{mplusn_eq} implies $\m_p + \n_p \geq 3 \expon_p - 4 \geq \expon_p.$
\item
If $p = N$, then using (\ref{mplusn_eq}) we compute $\m_p + \n_p \geq \num_p( \expon_p -1)- 1 - \K_{p+1}$, since $\num_{p+1} \geq 0$. If $\K_{p+1}=-1$, we are done, since $\num_p \geq 2$ and $2\expon_p - 2 \geq \expon_p$. If $\K_{p+1}=0$, then $\m_p+\n_p \geq \expon_p$ unless $\num_p=2$ and $\expon_p = 2$.
\end{itemize}
Now we address this exceptional case.\\

\noindent
{\bf Exceptional case:}
Let $X$ be a correlator of type $\X{L}{1}$ with $(p, \K_p, \num_p, \expon_p) = (N, 1 ,2,2)$.
If $X$ is not of type $\X{L}{2}$, there exists an index $i\neq N$, such that
\begin{equation}\label{exc-loop-step2}
X = \Sixptcorr{x_N}{ x_N}{ \ldots}{ x_i}{ \alpha}{ \beta }.
\end{equation}

We claim that $\num_{N-1}=0$.
From Lemma \ref{loop_kis1_lemma}, we know that $\K_{N-1}$ is 1 or 0.
If $\K_{N-1}=1$, then it is part of a (-1,1) pair, so \eqref{klmn_lemma-1} implies that $\num_{N-1}=0$.
If $\K_{N-1}=0$, then (\ref{ell_eq}) tells us that $\num_{N-1}=0$.

Next we claim that we can find $\alpha' \propto \alpha$ and $\beta' \propto \beta$ in $\Jac{W^T}$ such that
\begin{equation}\label{exc-loop-step2-claim2}
\n'_{N-2} > 0\ \text{and}\ \m'_{N-1}+\n'_{N-1}\geq \expon_{N-1}.
\end{equation}
This $\alpha'$ and $\beta'$ may not be in the standard basis.

If $\K_{N-1}=0$ then in fact we may take $\alpha = \alpha'$ and $\beta = \beta'$.
For in this case (\ref{mplusn_eq}) implies $\m_{N-1}+\n_{N-1} = \expon_{N-1}$.
If $W_1$ is a 2-variable loop, then similarly $\m_N+\n_N=1$, and we get \eqref{exc-loop-step2-claim2}.
If not, then $\K_{N-2}=0$. Otherwise we will have $(\K_{N-3},\K_{N-2})=(-1,1)$ and $\num_{N-3}=\num_{N-2}=0$. Then \eqref{mplusn_eq} would force $\m_{N-2}+\n_{N-2} = -1$ which is nonsensical. So $\K_{N-2}=0$ and \eqref{mplusn_eq} implies $\m_{N-2}+\n_{N-2} \geq 1$.
Thus again we get \eqref{exc-loop-step2-claim2}.

If $\K_{N-1}=1$, then $(\K_{N-2}, \K_{N-1})=(-1,1)$ and $\mathbf{\K}=(\ldots, 0, -1, 1, \ldots, -1, 1, 1).$ Here the $0$ is the last $0$ before the $(-1,1)$-sequence; say it occurs at the $r^{th}$ spot. We know from (\ref{mplusn_eq}) that
$$\mathbf{\m+\n}=(\ldots, t, 2 \expon_{r+1}-2, 0, \ldots, 2\expon_{N-2}-2, 0, *)$$ where $t \geq \expon_r$. Then we may assume $\alpha$ has a factor of $x_r x_{r+1}^{\expon_{r+1}-1} \ldots x_{N-2}^{\expon_{N-2}-1}$. Using the relation in \eqref{milnor-loop}, this factor is proportional to  a multiple of $x_{N-1}^{\expon_{N-1}}$. Let $\alpha'$ equal $\alpha$ with this replacement, and let $\beta'=\beta$. Then $\m'_{N-1} \geq \expon_{N-1}$ and $\n'_{N-2} = \expon_{N-2}-1$, and \eqref{exc-loop-step2-claim2} follows.

Finally, we use the Reconstruction Lemma \ref{reconstruction_lemma} with $\gamma = \beta$, $\delta = x_i$, $\epsilon = x_{N-1}$, and $\phi = \alpha''$ where $\alpha' = x_{N-1}  \alpha''$ to get
\begin{align*}
X=
\sixptcorr{ \ldots}{ x_N}{ x_N}{ x_{N-1}}{ x_i  \alpha''}{ \beta}&- \sixptcorr{ \ldots}{ x_N}{ x_N}{ x_{N-1}}{ \alpha''}{ x_i  \beta }\\
&+ \sixptcorr{ \ldots}{ x_N}{ x_N}{ x_i}{ \alpha''}{ x_{N-1}\beta }+S.
\end{align*}
After reducing the insertions of the first three correlators, by Remark \ref{rmk:reconstruction_invariance}(3) they all are of type $\X{L}{1}$ if they are nonvanishing.
By Lemma \ref{loop_kis1_lemma} they have $\K_N=1$.
As discussed earlier, a nonvanishing correlator in the form of \eqref{exc-loop-step2} with $\K_N=1$ must have $\num_{N-1} = 0$.
Thus the first two correlators vanish.
Use the same reconstruction on the third correlator until $\n_{N-1} = \expon_{N-1}$.
We still have $\n_{N-2}>0$, so now $\beta$ contains the factor $x_{N-2}x_{N-1}^{\expon_{N-1}}$, which is 0 by \eqref{loop_vanishing_lemma}. So the third correlator is also zero.
Thus we have reconstructed $X$ from correlators with fewer insertions.\qed\\

\noindent
{\bf Step 3.}
Let $X$ be a correlator of type $\X{L}{2}$, so $\num_i=0$ if $i \neq p$.
Since $X$ is also of type $\X{L}{1}$, we know $\K_N=1$, and $p=N-1$ or $N$ from Step 2.

If $p=N-1$, since $\num_p\neq0$, then \eqref{klmn_lemma-1} implies $\K_p\neq 1$. Thus $\K_p=0, \K_{p+1}=1,\num_{p+1}=0$, and \eqref{mplusn_eq}  implies that $\m_p+\n_p = \num_p(\expon_p-1) + \expon_p-2$. Thus Property (P3) in Lemma \ref{properties_lemma} implies that $\num_p \leq 2$ where $\num_p=2$ if and only if $\expon_p=2$.

If $p=N$, then $\K_p=1, \num_{p+1}=\num_1=0,$ and  $\K_{p+1}=\K_1=0$ or $-1$.
Then \eqref{mplusn_eq} and Property (P3) in Lemma \ref{properties_lemma} imply $\m_p+\n_p = \num_p(\expon_p-1)-1\leq 2\expon_p-2$, so $\num_p \leq 3$. If $\num_p=3$, then $\expon_p=2$ and $\K_{p+1}=0$. In this case
\[
X= \Fiveptcorr{ x_N}{x_N}{x_N}{\alpha}{\beta}.
\]

We will reconstruct $X$ from correlators with fewer insertions.
Since Property (P3) in Lemma \ref{properties_lemma} implies $\m_N+\n_N\leq 2\expon_N-2=2$, by \eqref{klmn_lemma-4}, there is no $1$ before any $0$ in the vector ${\bf \K}$. So there are two possibilities:
\begin{itemize}
\item $\mathbf{\K}=(  0, \ldots, 0, 1)$, $\mathbf{\num}=( 0, \ldots, 0, 3)$ and $\mathbf{\m+\n}=(\expon_1-1, \ldots, \expon_{N-1}+1, 2)$.
\item $\mathbf{\K}=(0, \ldots,0, 0, -1, 1, \ldots, -1, 1, 1)$,  $\mathbf{\num}=(0, \ldots, 0,0,0,0, \ldots, 0,0,3)$, and  $\mathbf{\m+\n}=(\expon_1-1, \ldots, \expon_{r-1}-1, \expon_r, M, 0, \ldots, M, 1, 2)$.
\end{itemize}

Here $M = 2\expon-2$, with the appropriate subscripts.
In the first case, $\m_N=\n_N = 1$ and both $\m_{N-1}$ and $\n_{N-1}$ are at least 2 (note that $\expon_{N-1}\neq 2$ in this case since that would imply $\expon_{N-1}+1=3 > 2\expon_{N-1}-2$).
We can choose $\alpha$ so that $\m_{N-2} > 0$. Now apply the Reconstruction Lemma \ref{reconstruction_lemma} with $\gamma = \alpha$, $\delta = x_N$, $\epsilon = x_{N-1}^{\n_{N-1}-1}$, and $\phi=\beta/\epsilon$. Thus $\alpha$ contains $x_{N-1}^2x_N$ and $\phi$ contains $x_{N-1}x_N$. Then $\phi  \delta$ and $\gamma  \delta$ both have a factor of $x_{N-1}x_N^2$, so they vanish by \eqref{loop_vanishing_lemma}. Similarly $\gamma  \epsilon$ has a factor of $x_{N-2}x_{N-1}^{\expon_{N-1}}=0$.


In the second case, again $\m_N=\n_N = 1$. Without loss of generality we can assume that $\n_{N-1}=1$. Apply the Reconstruction Lemma with $\gamma = \alpha$, $\delta = x_N$, $\epsilon = x_{r+1}$, and $\beta = \epsilon  \phi$. 
Then $\delta  \phi$ has a factor equal to $x_{N-1}x_N^2$, which equals 0 by \eqref{loop_vanishing_lemma}.
Similarly, $\epsilon  \gamma$ has a factor equal to $x_r x_{r+1}^{\expon_{r+1}}=0$.
Finally, in the third correlator, $\delta  \gamma$ has a factor equal to $x_r x_{r+1}^{\expon_{r+1}-1} x_{r+3}^{\expon_{r+3}-1} \ldots x_{N-2}^{\expon_{N-2}-1} x_{N}^2.$
As in the exceptional case in Step 2, we can use the relation \eqref{milnor-loop}  to rewrite the first terms to have a factor $x_{N-1}^{\expon_{N-1}}$. Since this is multiplied by $x_N^{\expon_N}$, it is zero by \eqref{loop_vanishing_lemma}.\qed\\

The last step will use the following lemmas.
\begin{lemma}\label{loop_corr_lemma}
If the correlator
\begin{equation}\label{loopy_form}
X_{k,i}=\Fourptcorr{ x_k}{ x_i}{ x_{k-1}x_k^{\expon_k-1}}{ \alpha\,\phi_{W_1^T}/x_i }, \quad \alpha\in \Jac{W^T-W_1^T}
\end{equation}
is of type $\X{X}{-1}$ then it can be reconstructed from correlators of type $\X{L}{4}$.
 \end{lemma}
\begin{proof}
If $(N, \expon_i) = (2,2)$, then up to symmetry $W_1 = x_1^{\expon}x_2+x_1x_2^2$ and $X = \fourptcorr{ x_1}{ x_2}{ x_1^{\expon-1}x_2}{\alpha x_1^{\expon-1} }$. If $\expon=2$, then $X$ is already of final type. If $\expon>2$, the result follows from a reconstruction as in \eqref{loop-exc-2} and \eqref{loop-exc-3} in Section \ref{sec-exceptional}.

Now we assume $(N, \expon_i) \neq (2,2)$.
We apply the Reconstruction Lemma \ref{reconstruction_lemma}  to the correlator $A_i = \Fourptcorr{ x_i}{ x_{i-1}}{ x_i^{\expon_i-1}}{ \Hessbase{W_1^T} \alpha}$ with $\epsilon = x_i$, $\phi = x_i^{\expon_i-2}$, $\delta = x_{i-1}$, and $\gamma = \Hessbase{W_1^T} \alpha$. Because $\gamma  \delta = \gamma  \epsilon = 0$ for degree reasons. Also we know
\begin{equation}\label{loop-aux}
A_i= \Fourptcorr{ x_i}{ x_i}{ x_{i-1}x_i^{\expon_i-2}}{ \Hessbase{W_1^T} \alpha }.
\end{equation}
By Dimension Axiom in Lemma \ref{vanishing_lemma}, if $A_i$ is nonzero, then it is a correlator of type $\X{L}{4}$.

Now let us start the reconstruction. Let $X_{k,i}$ be the correlator in \eqref{loopy_form}.
We replace $x_{k-1}x_k^{\expon_k-1}$ with $-x_{k+1}^{\expon_{k+1}}/\expon_k$ in $X_{k,i}$ using the relation \eqref{milnor-loop},
and then apply the Reconstruction Lemma \ref{reconstruction_lemma} with
$\delta = x_k, \epsilon = x_{k+1}, \phi = x_{k+1}^{\expon_{k+1}-1},$ and $\gamma =\alpha\,\phi_{W_1^T}/x_{i}.$

If $i\neq k, k+1$, the terms with $\gamma  \delta$ and $\gamma  \epsilon$ are both zero by \eqref{loop_vanishing_lemma}, and
\begin{equation}\label{loop-aux1}
X_{k,i}
=-{1\over\expon_k}\Fourptcorr{ x_{k+1}}{ x_i}{x_kx_{k+1}^{\expon_{k+1}-1}}{ \alpha\,\phi_{W_1^T}/x_i }
=-{X_{k+1,i}\over\expon_k}.
\end{equation}
For $k=i-1$ or $k=i$, one of $\gamma  \delta$ or $\gamma  \epsilon$ is nonzero, and we get
\begin{align}
&X_{i-1,i}
=-{1\over\expon_{i-1}}\left(A_{i}+\Fourptcorr{ x_{i}}{ x_i}{x_{i-1}x_{i}^{\expon_{i}-1}}{ \alpha\,\phi_{W_1^T}/x_i }\right)
=-{A_{i}+X_{i,i}\over\expon_{i-1}},\label{loop-aux2}
\\
&X_{i,i}
=-{1\over\expon_i}\left(A_{i+1}+\Fourptcorr{ x_{i+1}}{ x_i}{x_ix_{i+1}^{\expon_{i+1}-1}}{ \alpha\,\phi_{W_1^T}/x_i}\right)
=-{-A_{i+1}+X_{i+1,i}\over\expon_{i}}.\label{loop-aux3}
\end{align}
Combining \eqref{loop-aux1}, \eqref{loop-aux2}, and \eqref{loop-aux3}, we find
$$
X_{k,i}=\prod_{r=k}^{i-1}\left(-{1\over\expon_r}\right)A_i-\prod_{r=k}^{i}\left(-{1\over\expon_r}\right)A_{i+1}+\prod_{r=1}^{N}\left(-{1\over\expon_r}\right)X_{k,i}.
$$
Using \eqref{loop-aux} we know $X_{k,i}$ can be reconstructed from correlators of type $\X{L}{4}$.

\end{proof}

\begin{lemma}\label{loop_final}
The correlator
 \begin{equation*}
X_{k}=\Fourptcorr{ x_k}{ x_{k-1}x_k^{\expon_k-1}}{ \alpha}{\beta\,\phi_{W_1^T}}, \quad \alpha, \beta \in \Jac{W^T-W_1^T}
\end{equation*}
vanishes.
Furthermore, the correlator
\begin{equation*}
Y_{k}=\Fourptcorr{ x_k}{ x_k}{ x_{k-1}x_k^{\expon_k-2} \alpha}{ \Hessbase{W_1^T} \beta }, \quad \alpha, \beta \in \Jac{W^T-W_1^T}
\end{equation*}
 can be reconstructed from correlators of type $\X{L}{4}$.
\end{lemma}

\begin{proof}
 In $X_{k,i}$, replace the insertion $x_{k-1}x_k^{\expon_k-1}$ with $-x_{k+1}^{\expon_{k+1}}/\expon_k$ in $X_{k,i}$ using the relations \eqref{milnor-loop},
and then apply the Reconstruction Lemma \ref{reconstruction_lemma} with
$\delta = x_k, \epsilon = x_{k+1}, \phi = x_{k+1}^{\expon_{k+1}-1},$ and $\gamma =\beta\,\phi_{W_1^T}.$ Via a similar argument to the one used in the proof of Lemma \ref{loop_corr_lemma}, we have
$$
X_{k}=\prod_{r=1}^{N}\left(-{1\over\expon_r}\right)X_{k}.
$$
Then $X_k=0$ as desired. Now apply the Reconstruction Lemma to $Y_k$ with $\delta = x_k, \epsilon = \alpha, \phi = x_{k+1}^{\expon_{k+1}-1},$ and $\gamma =\beta\,\phi_{W_1^T}.$ We know $\gamma  \delta = 0$ for degree reasons. The term with $\epsilon \gamma$ is of type $\X{L}{4}$. The term with $\delta \phi$ is $X_{k+1}$, which we have seen is 0.
\end{proof}

\noindent
{\bf Step 4.}
We saw in Step 3 that if $X$ is a correlator of type $\X{L}{3}$, there are two general possibilities for $\mathbf{\K}$: $(0, \ldots, 0, -1, 1, \ldots, -1, 1, 1)$ and $(-1, 1, \ldots, -1, 1, 0, 1)$.
 In the first case, $p=N$; in the second, $p = N-1$. We will (a) reconstruct the second case from the first case, (b) reconstruct the first case from correlators with $\mathbf{\K}=(0, \ldots, 0, 1)$, and (c) reconstruct these last correlators from the correlators of final type.\\

\noindent
{\bf Step 4a.}
Assume $\mathbf{K}=(-1, 1, \ldots, -1, 1, 0, 1)$. We saw in Step 3 that $p=N-1$ and $\expon_p=2$. Let $\alpha_1$ and $\beta_1$ be the factors of $\alpha$ and $\beta$ in $W_1$, respectively.
By \eqref{mplusn_eq}, we know
$$\mathbf{\m+\n} = (2\expon_1-2, 0, \ldots, 2 \expon_{p-2}-2, 1, 2\expon_p-2, 0),$$
so without loss of generality $\beta_1 = \alpha_1/  x_{p-1}$ and
\[
\alpha_1 = x_1^{\expon_1-1} x_3^{\expon_3-1} \ldots x_{p-2}^{\expon_{p-2}-1} x_{p-1}x_p^{\expon_p-1}\propto x_2^{\expon_2-1} x_4^{\expon_4-1} \ldots x_{p-3}^{\expon_{p-3}-1} x_{p-1}^{\expon_{p-1}}x_{p+1}^{\expon_{p+1}-1}.
\]
Apply the Reconstruction Lemma to $X$ with $\delta = x_p$, $\epsilon = x_{p-1}$, $\gamma = \beta$, and 
$$\phi = x_2^{\expon_2-1} x_4^{\expon_4-1} \ldots x_{p-3}^{\expon_{p-3}-1} x_{p-1}^{\expon_{p-1}-1}x_{p+1}^{\expon_{p+1}-1}.$$
We can check that the correlator with $\epsilon  \gamma$ is of type $\X{L}{3}$ and ${\bf \K}=(0, \ldots, 0, -1, 1, \ldots, -1, 1, 1,0)$. It fits into the first case by shifting the index by $1$.
The remaining two correlators are
\[
A = \fourptcorr{x_p}{ x_{p-1}}{  \alpha_A}{ \beta_A} \quad \text{and} \quad B = \fourptcorr{ x_p}{ x_{p-1}}{ \alpha_B}{ \beta_B },
\]
where
\begin{align*}
 \alpha_A &= x_2^{\expon_2-1}x_4^{\expon_4-1} \ldots x_{p-1}^{\expon_{p-1}-1}x_px_{p+1}^{\expon_{p+1}-1} \alpha', & \beta_A &= x_1^{\expon_1-1} x_3^{\expon_3-1} \ldots x_{p-2}^{\expon_{p-2}-1} x_p^{\expon_p-1} \beta',\\
   \alpha_B &= x_2^{\expon_2-1}x_4^{\expon_4-1} \ldots x_{p-1}^{\expon_{p-1}-1}x_{p+1}^{\expon_{p+1}-1} \alpha', & \beta_B&=x_1^{\expon_1-1} x_3^{\expon_3-1} \ldots x_{p-2}^{\expon_{p-2}-1} x_p^{\expon_p} \beta'.
\end{align*}
Apply the Reconstruction Lemma to $B$ with $\delta = x_p$, $\epsilon = x_{p-1}$, $\phi= \alpha_B /x_{p-1} $, and $\gamma = \beta_B$. Then we get
\begin{equation}\label{loop-B}
B= \Fourptcorr{x_{p-1}}{ x_{p-1}}{ x_2^{\expon_2-1}x_4^{\expon_4-1} \ldots x_{p-1}^{\expon_{p-1}-2}x_px_{p+1}^{\expon_{p+1}-1} \alpha'}{
x_1^{\expon_1-1} x_3^{\expon_3-1} \ldots x_{p-2}^{\expon_{p-2}-1} x_p^{\expon_p} \beta' }.
\end{equation}
Since $\epsilon  \gamma$ has a factor of $x_{p-1}  x_p^{\expon_p}$, this monomial vanishes. Also $\gamma  \delta = 0$ because we can apply the Jacobi relations beginning with $x_p^{\expon_p} = (-\expon_{p-1})x_{p-2}x_{p-1}^{\expon_{p-1}-1}$ to get
$$\gamma  \delta\propto x_2^{\expon_2-1}x_4^{\expon_4-1}\ldots x_{p-1}^{\expon_{p-1}-1}x_p^2x_{p+1}^{\expon_{p+1}-1}.$$
This has a factor of $ x_{p-1}x_p^2 = x_{p-1}  x_p^{\expon_p} = 0$. Now we write the last insertion of \eqref{loop-B} in the standard basis (it equals $x_2^{\expon_2-1}x_4^{\expon_4-1} \ldots x_{p-1}^{\expon_{p-1}-1}x_px_{p+1}^{\expon_{p+1}-1}\beta'$). Then $B$ is of type $\X{L}{3}$ with the vector ${\bf \K}=(1,-1,1,\ldots,-1,1,1,0,-1)$. Again, $B$ fits into the first case by shifting the index.

Now apply the Reconstruction Lemma to $A$ with $\gamma = x_p, \epsilon = x_{p-1}x_p$, $\phi=\alpha_A / \epsilon$, and $\delta = \beta_A$. Then $\epsilon  \gamma=x_{p-1}x_p^{\expon_p}=0$. The remaining correlators are
\[
D = \Fourptcorr{ x_{p-1}}{ x_p}{ x_{p-1}x_p}{ x_1^{\expon_1-1}x_2^{\expon_2-1}\ldots x_{p-1}^{\expon_{p-1}-2}x_p^{\expon_p-1}x_{p+1}^{\expon_{p+1}-1}  \alpha''},
\]
and $E =\fourptcorr{x_{p-1}}{ x_{p-1}x_p}{ \alpha_E}{ \beta_E }$ with
\[
\alpha_E =  x_2^{\expon_2-1}x_4^{\expon_4-1}\ldots x_{p-1}^{\expon_{p-1}-2}x_{p+1}^{\expon_{p+1}-1}  \alpha', \quad \beta_E =  x_1^{\expon_1-1}x_3^{\expon_3-1} \ldots x_{p-2}^{\expon_{p-2}-1}x_p^{\expon_p} \beta'.
\]
Now $D$ can be reconstructed via Lemma \ref{loop_corr_lemma}. For $E$, apply the Reconstruction Lemma one last time with $\epsilon = x_{p-1}$, $\phi = x_p$, $\delta = \alpha_E$, and $\gamma = \beta_E$. Then $\epsilon  \gamma$ has a factor of $x_{p-1}x_p^{\expon_p}$ which is 0. Also, the correlator with $\delta  \phi$ equals $B$, which we already saw satisfies $p=N$. Finally, $\gamma  \delta = x_1^{\expon_1-1}x_2^{\expon_2-1}\ldots x_{p-1}^{\expon_{p-1}-2}x_p^{\expon_p}x_{p+1}^{\expon_{p+1}-1}$ is 0 by \eqref{loop_vanishing_lemma} when $(N, \expon_{p-1}) \neq (2,2)$. If $(N, \expon_{p-1}) = (2,2)$ then the correlator with $\gamma \delta$ is of type $\X{L}{4}$.\\

%

\noindent
{\bf Step 4b.}
Now assume $\mathbf{\K} = (0, \ldots, 0, -1, 1, \ldots, -1, 1, 1)$ and $p=N$.
Let $\alpha_1$ and $\beta_1$ be the factors of $\alpha$ and $\beta$ in $W_1$.
If $\mathbf{\K}\neq(0, \ldots, 0, 1),$ we claim that we can find monomials $\alpha_1' \propto \alpha_1$ and $\beta_1' \propto \beta_1$ ($\alpha_1'$ and $\beta_1'$ may not be in the standard basis) such that
\begin{align}
&\m'_i+\n'_i = \expon_i-1 \; \text{for}\; i<N-1 \notag\\
&(\m'_{N-1},\n'_{N-1}, \m'_N, \n'_N) = (\expon_{N-1},0, \expon_N-2, \expon_N-1) \label{loop5_fact1}
\intertext{and}
&x_N  \beta'_1 = 0. \label{loop5_fact2}
\end{align}
We find $\alpha_1'$ and $\beta_1'$ by analyzing the possibilities for $\mathbf{\K}$.

If $\mathbf{\K} = (-1,1, \ldots, -1,1,1)$ then
$$\mathbf{\m+\n} = (2\expon_1-2, 0, \ldots, 2 \expon_{N-2}-2, 0, 2\expon_N-2),$$ so
\[
\alpha_1 = \beta_1 = x_1^{\expon_1-1} x_3^{\expon_3-1} \ldots x_{N-2}^{\expon_{N-2}-1}x_N^{\expon_N-1}.
\]
Then $x_N  \beta_1 = 0$ by \eqref{loop_vanishing_lemma}. We may take $\beta_1' = \beta_1$ and
\[
\alpha_1\propto \alpha_1'=x_2^{\expon_2-1} x_4^{\expon_4-1} \cdots x_{N-3}^{\expon_{N-3}-1} x_{N-1}^{\expon_{N-1}}x_N^{\expon_N-2}\in\Jac{W_1^T}.
\]

If $\mathbf{\K} = (0, \ldots, 0, 0, -1, 1, \ldots, -1, 1, 1)$ then
$$\mathbf{\m+\n} = (\expon_1-1, \ldots, \expon_{r-1}-1, \expon_r, 2 \expon_{r+1}-2, 0, \ldots, 2 \expon_{N-2}-2, 0, 2\expon_N-3).$$
We can assume that
\[
\alpha_1 = x_1^{\m_1} \ldots x_{r_1}^{\m_{r-1}}x_r^{\m_r}x_{r+1}^{\expon_{r+1}-1} x_{r+3}^{\expon_{r+3}-1} \ldots x_{N-2}^{\expon_{N-2}-1} x_N^{\expon_N-2}.
\]
Since $\m_r + \n_r = \expon_r$, we know that $\m_r \geq 1$. Then we find that
\[
\alpha_1\propto x_1^{\m_1} \ldots x_{r-1}^{\m_{r-1}}x_r^{\m_r-1}x_{r+2}^{\expon_{r+2}-1} \ldots x_{N-3}^{\expon_{N-3}-1}x_{N-1}^{\expon_{N-1}} x_N^{\expon_N-2}.
\]
Secondly, in this case,
\[
x_N  \beta_1 = x_1^{\n_1} \ldots x_{r-1}^{\n_{r-1}}x_r^{\n_r}x_{r+1}^{\expon_{r+1}-1} x_{r+3}^{\expon_{r+3}-1} \ldots x_{N-2}^{\expon_{N-2}-1} x_N^{\expon_N}.
\]
Then
\[
x_N  \beta_1 \propto x_1^{\n_1} \ldots x_{r-1}^{\n_{r-1}}x_r^{\n_r-1}x_{r+2}^{\expon_{r+2}-1} \ldots x_{N-3}^{\expon_{N-3}-1}x_{N-1}^{\expon_{N-1}} x_N^{\expon_N}=0
\]
because it has a factor of $x_{N-1} x_N^{\expon_N}$. Thus we may take $\beta_1' = \beta_1$ and
\[
\alpha_1' = x_1^{\m_1} \ldots x_r^{\m_r-1}x_{r+2}^{\expon_{r+2}-1} \ldots x_{N-1}^{\expon_{N-1}} x_N^{\expon_N-2}.
\]

Having found monomials $\alpha_1'$ and $\beta_1'$ satisfying \eqref{loop5_fact1} and \eqref{loop5_fact2}, in the correlator $X$ replace $\alpha_1$ with $\alpha_1'$ in $\alpha$, and similarly with $\beta$. Let $\alpha' = \alpha/\alpha_1'  $ and $\beta' = \beta/\beta_1'  $.
Apply the Reconstruction Lemma \ref{reconstruction_lemma} with $\delta = x_N$, $\gamma = \beta$, $\epsilon = x_{N-1}$, and $\phi =  \alpha / \epsilon.$ Then the term with $\gamma  \epsilon$ has $\mathbf{\K} = (0, \ldots, 0, 1)$, as can be checked by computing $\mathbf{\m+\n}$ (note that all elements are already in the standard basis). The term with $\gamma  \delta$ vanishes by \eqref{loop5_fact2}. The final term with $\delta  \phi$ has the form
\[
\Fourptcorr{x_N}{ x_{N-1}}{x_1^{\m_1} \ldots x_{N-2}^{\m_{N-2}} x_{N-1}^{\expon_{N-1}-1} x_N^{\expon_N-1}  \alpha'}{\beta}.
\]
Apply the Reconstruction Lemma again with $\gamma = x_N$, $\delta = \beta$, $\epsilon = x_{N-1}x_N^{\expon_N-1}$, and $\phi = x_1^{\m_1} \ldots x_{N-2}^{\m_{N-2}} x_{N-1}^{\expon_{N-1}-2}  \alpha'.$ Then $\gamma  \epsilon$ has a factor equal to $x_{N-1} x_N^{\expon_N}$, so this correlator vanishes. The correlator with $\gamma  \delta$ also vanishes by \eqref{loop5_fact2}. Finally, by \eqref{loop5_fact1}, the correlator with $\delta  \phi$ is in the form of Lemma \ref{loop_corr_lemma}.\\

\noindent
{\bf Step 4c.}
Finally, we reconstruct correlators $X$ of type $\X{L}{3}$ with $\mathbf{\K} = (0, \ldots, 0, 0, 1)$ from correlators of final type. For such $X$, 
$$\mathbf{\m+\n} = (\expon_1-1, \ldots, \expon_{N-2}-1, \expon_{N-1}, 2\expon_N-3).$$
Thus both $\m_{N-1}$ and $\n_{N-1}$ are at least 1, and we can assume $\m_N=\expon_N-2$ and $\n_N = \expon_N-1$. If $\mathbf{\m} = (0, \ldots, 0, 1, \expon_N-2)$ we are done; otherwise there is some $i$ such that $\m_i$ is larger than it should be.

Apply the Reconstruction Lemma with $\delta = x_N$, $\gamma = \beta$, $\epsilon = x_i$, and $\phi =\alpha/\epsilon$.
Then $\gamma  \delta$ has a factor equal to $x_{N-1} x_N^{\expon_N}$, which is 0.
We can check that the correlator with $\gamma  \epsilon$ is of type $\X{L}{3}$ and $\mathbf{\K} = (0, \ldots, 0, 0, 1)$, but $\deg(\alpha) < \deg(\alpha_X)$.
So if we use the same reconstruction on this correlator, eventually the second form in the reconstruction will be
 $$\Fourptcorr{ x_i}{ x_i}{ x_{i-1}x_i^{\expon_i-2} \alpha}{ \Hessbase{W_1^T} \beta }, \quad \alpha, \beta \in \Jac{W^T-W_1^T}.$$
 This correlator can be reconstructed from correlators of type $\X{L}{4}$ by Lemma \ref{loop_final}.

The correlator containing $\delta  \phi$ equals
$$
\Fourptcorr{ x_N}{ x_i}{ x_1^{\m_1} \ldots x_i^{\m_i-1} \ldots x_{N-1}^{\m_{N-1}}x_N^{\expon_N-1}  \alpha'}{  x_1^{\n_1} \ldots x_{N-1}^{\n_{N-1}} x_N^{\expon_N-1}  \beta' }
$$
Here $\alpha'$ and $\beta'$ are the factors of $\alpha$ and $\beta$ not in $W_1$.
Apply the Reconstruction Lemma with $\gamma = x_N$, $\delta = \beta$, $\epsilon = x_{N-1}x_N^{\expon_N-1}$, and $\phi = \alpha/\epsilon$. Then $\gamma  \delta$ and $\gamma  \epsilon$ vanish because they have a factor equal to  $x_{N-1} x_N^{\expon_N}$. The final term with $\delta  \phi$ is in the form of Lemma \ref{loop_corr_lemma}.\qed

\appendix

\section{{Loop polynomial $W = x_1^{\expon_1} x_2 + x^{\expon_2} x_3 + \ldots + x_{N-1}^{\expon_{N-1}}x_N + x_N^{\expon_N}x_1, \expon_N=2$.}}\label{sec-exceptional}

For the loop polynomial $W = x_1^{\expon_1} x_2 + x^{\expon_2} x_3 + \ldots + x_{N-1}^{\expon_{N-1}}x_N + x_N^{\expon_N}x_1$, if $\expon_N=2$, then $X=\langle\agen_N,\agen_N,S_N,H\rangle$ is never concave. 
They can be classified into three exceptional families listed below.
For the first two families, we use WDVV equations to solve $X$ from concave correlators. For the last family, we apply a result of Gu\'er\'e \cite{Gu}.

We can use Lemma \ref{loop-comb} to compute the phases. 
Now we list all the cases as follows:
\begin{enumerate}
\item [{\bf Case 1.}]
$N =2$ and $\expon_{N-1}=2$. In this case, $\agen_N$ and $\agen_{N-1}$ are broad.

\item [{\bf Case 2.}]
$N=2$ and $\expon_{N-1}>2$. In this case, $\agen_N$ is broad.

\item [{\bf Case 3.}]
$N \geq 3$. In this case, $h_{1,+}^{(N-1)}\in (0,1)$ and $h_{1,-}^{(N-1)}\in (-2,-1)$. 
This implies
$\ell_{1,+}^{(N-1)}=0$ and $\ell_{1,-}^{(N-1)}=-2.$
A similar discussion using the normalization exact sequence as in Lemma \ref{lemma-chiodo} implies there is a singular curve $[C]\in \overline{\mathscr{W}}_{0,4}(\agen_N,\agen_N,S_N,H)$, such that
$H^0(C, \LL_{N-1}|_C)=\mathbb{C}$.
Thus the correlator is not concave.
\end{enumerate}

Now we compute the correlator $X=\langle\agen_N,\agen_N,S_N,H\rangle$ for each case as shown above. \\

\noindent
{\bf Case 1:} In this case $W=x_1^2 x_2+x_1 x_2^2$ and both $\agen_1$ and $\agen_2$ are broad. We recall that the mirror map $\mirror$ in \eqref{Krawitz-mirror-map} is given by
\begin{eqnarray*}
1=\mirror(1)=\lceil 1\; ; \; J_W\rfloor,  \quad J^2:=\mirror(x_1x_2)=\lceil 1\; ; \; J^{-1}_W\rfloor, \quad \theta_i=\mirror(x_i)=\lceil x_i\; ; \; 1\rfloor, \quad i=1,2.
\end{eqnarray*}
Since two variables $x_1$ and $x_2$ are symmetric in $W$, we only need to compute
$$X=\langle \theta_1, \theta_1, \theta_2, J^2\rangle= \langle \theta_1, \theta_2, \theta_2, J^2\rangle.$$
Both $\theta_1$ and $\theta_2$ are broad. It is very difficult for us to compute $X$ directly. However, all the correlators can still be determined by WDVV equations and the correlator
$$X_0:=\langle J^2, J^2, J^2, J^2, J^2, J^2, J^2\rangle.$$
We can check that $X_0$ is concave and
$$\deg\LL_1=\deg\LL_2=-3.$$
This correlator can be calculated by the Concavity Axiom using Theorem 1.1.1 in \cite{C}.
The computation of  $X_0$ is exactly the same as the computation of
$\langle J_{D_4}^2, J_{D_4}^2, J_{D_4}^2, J_{D_4}^2, J_{D_4}^2, J_{D_4}^2, J_{D_4}^2\rangle$ in the FJRW theory of a pair $(D_4=x_1^2 x_2+x_2^3, \{J_{D_4}\})$. The later is worked out in \cite{D_4}. Hence we get
$$X_0={2\over 27}.$$
By the Dimension Axiom and Integer Degrees Axiom on the B-side (see Lemma \ref{vanishing_lemma}), besides $X$ and $X_0$, all other possible nonvanishing primary correlators with at least four insertions are
\begin{eqnarray*}
\begin{split}
X_1&=\langle \theta_1, \theta_1, \theta_1, J^2\rangle=\langle \theta_2, \theta_2, \theta_2, J^2\rangle, \\
X_2&=\langle \theta_1, \theta_1, J^2, J^2, J^2\rangle=\langle \theta_2, \theta_2, J^2, J^2, J^2\rangle, \\
X_3&=\langle \theta_1, \theta_2, J^2, J^2, J^2\rangle.
\end{split}
\end{eqnarray*}
Since the pairing satisfies
$$\eta_{\theta_1,\theta_1}=\eta_{\theta_2,\theta_2}=-2, \quad \eta_{\theta_1,\theta_2}=1,$$
the inverse of the pairing matrix is
$$
\left(\eta^{\theta_i,\theta_j}\right)=
\left( \begin{array}{cc}
-{2\over3} & -{1\over3} \\
-{1\over3} & -{2\over3}
\end{array} \right).
$$
We apply Reconstruction Lemma \ref{reconstruction_lemma} to $X_1$ with $\gamma=\delta=\theta_2$, $\xi=\phi=\theta_1$. We find
$$
-2X_1=X+X-(-2X).
$$
Apply Reconstruction Lemma \ref{reconstruction_lemma} to $X_1, X_2$ and $X_3$, with $\gamma=\delta=J^2$, $\xi=\phi=\theta_1$. We get
\begin{eqnarray*}
\begin{split}
-2X_2&=2\left(-{2\over3}X_1^2-{1\over3}(2XX_1)-{2\over3}X^2\right),\\
-2X_3&=2\left(-{2\over3}XX_1-{1\over3}XX_1-{1\over3}X^2-{2\over3}X^2\right),\\
-2X_0&={4\choose2}\left(-{2\over3}X_2^2-{1\over3}(2X_2X_3)-{2\over3}X_3^2\right).
\end{split}
\end{eqnarray*}
Combine all the equations together, we get
\begin{equation}\label{loop-exc-22}
X_1=-2X, \quad X_2=2X^2, \quad X_3=-X^2, \quad X_0=6X^4.
\end{equation}
This implies
\begin{equation}\label{constant-c}
X={c\over3} \quad \text{for some fourth root of unity } c.
\end{equation}
Now the result follows by adjusting the mirror map $\mirror$ via
$$\mirror(x_i)=c^{-1}\lceil x_i\; ; \; 1\rfloor, \quad i=1,2$$
and adjusting the pairing similarly.\\

\noindent
{\bf Case 2:} If $N=2$, $\expon_N=2$, and $\expon_{N-1}>2$, then $W = x_1^{\expon}x_2+x_2^2x_1$ and $$(q_1,q_2) = \left({1\over 2\expon-1}, {\expon-1\over 2\expon-1}\right).$$
In this case $\Theta_2$ is broad, but the correlator $\langle \agen_1, \agen_1, \agen_1^{\expon-2}\agen_2, \agen_1^{\expon-1}\agen_2 \rangle$ is concave and we can apply Formula \eqref{correlator_formula}  to get
\begin{equation}\label{loop-exc-q1}
\langle \agen_1, \agen_1, \agen_1^{\expon-2}\agen_2, \agen_1^{\expon-1}\agen_2 \rangle=q_1.
\end{equation}
Next we will use reconstruction to compute the correlator $X = \langle \agen_2, \agen_2, \agen_1, \agen_1^{\expon-1}\agen_2 \rangle$. We notice that $W^T=W$, and under the Krawitz's map $\mirror$, the relations in $\Jac{W^T}$ become
\begin{equation}\label{jacobi-exc}
\expon \agen_1^{\expon-1}\agen_2+\agen_2^2=0, \quad  \agen_1^{\expon}+2\agen_1\agen_2=0.
\end{equation}
Now we apply the Reconstruction Lemma \ref{reconstruction_lemma} with $\gamma = \delta = \agen_2$, $\epsilon = \agen_1$, and $\phi = \agen_1^{\expon-2}\agen_2$. Then $\delta \star \phi = \agen_1^{\expon-2}\agen_2^2=0$ by \eqref{jacobi-exc} and our assumption that $\expon>2$. Then
\begin{eqnarray}\label{loop-exc-1}
\begin{split}
X&= \langle \agen_1, \agen_2, \agen_1^{\expon-2}\agen_2, \agen_1\agen_2 \rangle - \langle \agen_1, \agen_1, \agen_1^{\expon-2}\agen_2, \agen_2^2 \rangle\\
&=\langle \agen_1, \agen_2, \agen_1^{\expon-2}\agen_2, -{1\over2}\agen_1^{\expon} \rangle - \langle \agen_1, \agen_1, \agen_1^{\expon-2}\agen_2, -\expon\agen_1^{\expon-1}\agen_2 \rangle\\
&=-{1\over2}\langle \agen_1, \agen_2, \agen_1^{\expon-2}\agen_2, \agen_1^{\expon} \rangle +\expon q_1.
\end{split}
\end{eqnarray}
The second equality follows from \eqref{jacobi-exc} and the third equality is a consequence of the formula \eqref{loop-exc-q1}.
Now we apply the Reconstruction Lemma \ref{reconstruction_lemma} again to $\langle \agen_1, \agen_2, \agen_1^{\expon-2}\agen_2, \agen_1^{\expon} \rangle$ with $\delta = \agen_2$, $\gamma = \agen_1^{\expon-2}\agen_2$, $\epsilon = \agen_1$, and $\phi = \agen_1^{\expon-1}$. This time $\delta \star \gamma = 0$, and we find
\begin{eqnarray}\label{loop-exc-2}
\begin{split}
\langle \agen_1, \agen_2, \agen_1^{\expon-2}\agen_2, \agen_1^{\expon} \rangle
&=\langle \agen_1, \agen_1, \agen_1^{\expon-2}\agen_2, \agen_1^{\expon-1}\agen_2 \rangle + \langle \agen_1, \agen_2, \agen_1^{\expon-1}, \agen_1^{\expon-1}\agen_2 \rangle\\
&=q_1+\langle \agen_1, \agen_2, \agen_1^{\expon-1}, \agen_1^{\expon-1}\agen_2 \rangle.
\end{split}
\end{eqnarray}
Finally, we apply the Reconstruction Lemma \ref{reconstruction_lemma} to $\langle \agen_1, \agen_2, \agen_1^{\expon-1}, \agen_1^{\expon-1}\agen_2 \rangle$ with $\delta = \agen_2$, $\gamma = \agen_1^{\expon-1}\agen_2$, $\epsilon = \agen_1$, and $\phi = \agen_1^{\expon-2}$. Then $\gamma \star \delta = \agen_1^{\expon-1}\agen_2^2 = 0$ and $\epsilon \star \gamma = \agen_2\agen_1^{\expon}=0$ by \eqref{jacobi-exc}, so we get
\begin{equation}\label{loop-exc-3}
\langle \agen_1, \agen_2, \agen_1^{\expon-1}, \agen_1^{\expon-1}\agen_2 \rangle= \langle \agen_1, \agen_1, \agen_1^{\expon-2}\agen_2, \agen_1^{\expon-1}\agen_2 \rangle=q_1.
\end{equation}
Thus from \eqref{loop-exc-q1}, \eqref{loop-exc-1}, \eqref{loop-exc-2}, and \eqref{loop-exc-3}, we have deduced that
$$X = (\expon-1)q_1=q_2.$$

\noindent
{\bf Case 3:} In this case, $N \geq 3$ and $\expon_N=2$.
Now the correlator is not concave and Formula \eqref{correlator_formula} is not applicable directly. However, we can use the following techniques of Gu\'er\'e \cite{Gu} for computing the Polishchuk-Vaintrob virtual class. 
\begin{theorem}\cite{Gu}\label{jeremys_theorem}
Let $W$ be an invertible polynomial of atomic type. Let $Y$ be a correlator such that there is some index $i$ such that $H^0(C, \LL_i)=0$ for any geometric fiber $C$. Let $t(j)$ be the unique index such that $x_j^{\expon_{j}}x_{t(j)}$ is a monomial of $W$. Define
\begin{center}
\begin{tabular}{ll}
$\lambda_{t(j)} = \lambda_j^{-\expon_j},$ & if $H^0(C, \LL_i)\neq0$,\\
$\lambda_j = \lambda,$ & for every remaining index $j$.
\end{tabular}
\end{center}
Then the corresponding Polishchuk-Vaintrob virtual class $c_{\rm vir}(Y)$ in $H^*(\W{g}{k},\CC)$ is
\begin{equation}\label{jeremys_eq}
c_{\rm vir}(Y)= \lim_{\lambda \rightarrow 1}\left( \prod_{j=1}^N(1-\lambda_j)^{-\Ch_0(R\pi_*\LL_j)} \right) \exp\left( \sum_{j=1}^N\sum_{\ell \geq 1} s_\ell(\lambda_j)\Ch_\ell(R\pi_*\LL_j)\right),
\end{equation}
where $\Ch_\ell$ is the term of degree $\ell$ of the Chern character,
\[
s_\ell(x) = \frac{B_\ell(0)}{\ell} +(-1)^\ell\sum_{k=1}^\ell(k-1)!\left(\frac{x}{1-x} \right)^k\gamma(\ell,k),
\]
and $\gamma(\ell,k)$ is defined by the generating function
\[
\sum_{\ell \geq 0} \gamma(\ell,k) \frac{z^\ell}{\ell!}= \frac{(e^z-1)^k}{k!}.
\]
\end{theorem}
On the other hand, when all insertions in the correlator are narrow, Chang-Li-Li \cite{CLL} showed that the Polishchuk-Vaintrob and Fan-Jarvis-Ruan-Witten virtual classes are equal. Thus we can use Formula \eqref{jeremys_eq} to compute the correlator $X$ in {\bf Case 3}, where $t(j)=j+1$. According to \eqref{degree-computation}, on a generic fiber, the line bundle degrees are
$$\deg\LL_j=-1 \text{\ for\ } j<N, \quad \text{and}\ \deg\LL_N=-2.$$
Then $\Ch_0(R\pi_*\LL_j) = h^0(C, \LL_j)-h^1(C, \LL_j) = \deg(\LL_j)+1$ by Riemann-Roch. Thus $$\Ch_0(R\pi_*\LL_j)=0 \text{\ for\ } j<N, \quad \text{and}\ \Ch_0(R\pi_*\LL_N)=-1.$$
Also, since we are working on $\M{0}{4}$, by degree considerations the sum over $\ell$ has only the summand $\ell=1$, and the power series defined by the exponential terminates after the linear part. Thus, plugging in the function $s_j$, Formula \eqref{jeremys_eq} becomes
\begin{align*}
c_{\rm vir}(X) &= \lim_{\lambda \rightarrow 1} (1-\lambda_N) \left(1+ \sum_{j=1}^N \left(-\frac{1}{2}-\frac{\lambda_j}{1-\lambda_j}\right)\Ch_1(R\pi_*\LL_j)\right) \\
&=-\sum_{j=1}^N \lim_{\lambda \rightarrow 1} \frac{\lambda_j(1-\lambda_N)}{1-\lambda_j}\Ch_1(R\pi_*\LL_j),
\end{align*}
where $\lambda_N = \lambda^{-\expon_{N-1}}$ and $\lambda_j = \lambda$ for $j<N$. Because $\LL_j$ is concave and $\ld_j=-1$ for $j<N-1$, so $\Ch_1(R\pi_*\LL_j)=0$. Thus
\begin{eqnarray*}
c_{\rm vir}(X) &=&\left(- \lim_{\lambda \rightarrow 1} \frac{\lambda(1-\lambda^{-\expon_{N-1}})}{1-\lambda} \Ch_1(R\pi_*\LL_{N-1})-\Ch_1(R\pi_*\LL_N)\right)\\
&=&\expon_{N-1} \Ch_1(R\pi_*\LL_{N-1})-\Ch_1(R\pi_*\LL_{N}).
\end{eqnarray*}
As in the derivation of \eqref{correlator_formula}, we can apply Theorem 1.1.1 in \cite{C} to compute
\[
X =\expon_{N-1}(-q_{N-1}-2\rho_N^{(N-1)})-(-1+2\rho_N^{(N)}) =q_N.
\]

\end{document}